\documentclass[10pt]{amsart}
\usepackage{amssymb,amsfonts,amsthm,amscd,stmaryrd,dsfont,esint,upgreek}
\usepackage[numbers]{natbib}
\usepackage[mathcal]{euscript}
\usepackage{srcltx}
\usepackage[english]{babel}
\usepackage{graphicx}
\usepackage{amsmath}
\usepackage{amsthm}
\usepackage{amssymb}
\usepackage[T1]{fontenc}
\usepackage[table]{xcolor}
\usepackage{mathrsfs}
\usepackage{exscale}
\usepackage{fancyhdr}
\usepackage{bbm} 
\usepackage{enumitem}
\usepackage{soul}
\usepackage{float}
\usepackage{array}
\usepackage{tabularx}
\usepackage{epic}

\makeatletter

\usepackage[usenames,dvipsnames]{pstricks}
\usepackage{epsfig}
\usepackage{pst-grad} 
\usepackage{pst-plot} 

\usepackage{eepic}
\usepackage[textwidth=16.01cm,textheight=22cm,
marginratio=1:1,heightrounded]{geometry}

\newcommand{\proj}{\operatorname{proj}}

 \newcommand{\sublim}{\operatornamewithlimits{\longrightarrow}}
\providecommand{\norm}[1]{\lVert#1 \rVert}

\numberwithin{equation}{section}
\newtheorem{theorem}{Theorem}[section]
\newtheorem{lemma}[theorem]{Lemma}

\newtheorem{Def}[theorem]{Definition} 

\newtheorem{prop}[theorem]{Proposition}
\newtheorem{cor}{Corollary}[section]
\newtheorem{remark}[theorem]{Remark}

\author{Jean-Paul Daniel}
\title[A game interpretation of the Neumann problem for fully nonlinear equations]{A game interpretation of the Neumann problem for fully nonlinear parabolic and elliptic equations} 
 \date{\today} 
\begin{document}
\newcommand{\R}{\mathbb{R}}
\newcommand{\Q}{\mathbb{Q}}
\newcommand{\N}{\mathbb{N}}
\newcommand{\C}{\mathcal{C}}
\newcommand{\W}{\mathbb{W}}
\newcommand{\E}{\mathbb{E}}
\newcommand{\p}{\mathbb{P}}
\newcommand{\1}{\mathbbm{1}}
\newcommand{\eps}{\varepsilon}
\newcommand{\vphi}{\varphi}
\newcommand{\obl}{\varsigma}
\newcommand{\ds}{\displaystyle}
\newcommand{\udl}{\underline}
\newcommand{\ol}{\overline}
\maketitle

\begin{abstract} 
We provide a deterministic-control-based interpretation for a broad class of fully nonlinear parabolic and elliptic PDEs with 
continuous Neumann boundary conditions in a smooth domain.
We construct families of two-person games depending on a small parameter~$\eps$ which extend those proposed by Kohn and Serfaty 
\cite{kohns}. These new games treat a Neumann boundary condition by introducing some specific rules near the boundary. 
 We show that the value function converges, in the viscosity sense, to the solution of the PDE as $\eps$ tends to zero.
Moreover, our construction allows us to treat both the oblique and the mixed type Dirichlet-Neumann boundary conditions. 
 \end{abstract}

\maketitle

\section{Introduction}

In this paper, we  propose a deterministic control interpretation, via ``two persons repeated games'', for a broad class of fully nonlinear 
equations of elliptic  or parabolic type with a continuous Neumann boundary condition in a smooth (not necessarily bounded)  domain.
 In their seminal paper \cite{kohns}, Kohn and Serfaty focused on the one hand on the whole space case in the parabolic setting 
and on the other hand on the Dirichlet problem in the elliptic framework. 
They construct a monotone and consistent difference approximation of the operator from the dynamic programming principle associated to the game.

Our motivation here is to adapt their approach to the Neumann problem in both settings. Furthermore, once this issue is solved, 
we will see how the oblique or the mixed type Dirichlet-Neumann boundary problem can also be treated by this analysis. 
  We consider equations in a domain $\Omega \subset \R^N$ having the form
\begin{equation}
 -u_t+f(t,x,u,Du,D^2u)=0  \label{def_eq_par}
\end{equation}
or
\begin{equation}
 f(x,u,Du,D^2u)+\lambda u=0 , \label{def_eq_el}
\end{equation}
where $f$ is elliptic in the sense that $f$ is monotone in its last variable, subject to the Neumann boundary condition  
\begin{equation}\label{def_cd_neumann}
 \frac{\partial u}{\partial n}=h. 
\end{equation}
As in \cite{kohns}, the class of functions $f$ considered is large, including those that are non-monotone in the $u$ argument and degenerate in the $D^2u$ argument.
We make the same hypotheses on the continuity, growth, and $u$-dependence of $f$ imposed in \cite{kohns}. They are recalled at the end of the section. 
 In the stationary setting \eqref{def_eq_el}, we focus on the Neumann problem, solving the equation in a domain 
$\Omega$ with \eqref{def_cd_neumann} at $\partial \Omega$. 
In the time-dependent setting \eqref{def_eq_par}, we address the Cauchy problem, solving the equation 
with \eqref{def_cd_neumann} at $\partial \Omega$ for $t<T$ and $u=g$ at terminal time $t=T$. 
The PDEs and boundary conditions are always interpreted in the ``viscosity sense'' (Section~\ref{convergence} presents a review of this notion).

Our games have two opposite players, Helen and Mark, who always make decisions rationally and deterministically. 
The rules depend on the form of the equation, but there is always a small parameter $\eps$, which governs
the spatial step size and (in time-dependent problems) the time step.  Helen's goal is to optimize her 
worst-case outcome.  
 When $f$ is independent of $u$, we shall characterize her value function $u^\eps$ by the dynamic programming principle. 
 If $f$ depends also on $u$, the technicality of ours arguments requires to introduce a level-set formulation since the uniqueness of the viscosity solution 
 is no longer guaranteed. 
 The score $U^\eps$ of Helen now depends on a new parameter $z \in \R$. In the parabolic setting, it is 
 defined by an induction backward in time given by 
 \begin{equation*} 
\forall z\in \R, \quad  U^\eps(x,z,t)=\max_{p,\Gamma}  \min_{\Delta \hat x} U^\eps(x+\Delta x, z+\Delta z, t+\Delta t),
 \end{equation*}
endowed with the final-time condition $ U^\eps(x,z,t)=g(x)-z$. The max on $p$, $\Gamma$ and the min on $\Delta \hat x$ are given by 
some constrains depending on the rules of the game and some powers of $\eps$. 
This dynamic programming principle is similar to the one given in \cite[Section 2.3]{kohns}. 
In that case, our value functions $u^\eps$ of interest are defined through the 0-level set of $U^\eps$ with respect to $z$ 
as the maximal and the minimal solutions of $U^\eps(x,z,t)=0$. 
They satisfy two dynamic programming inequalities 
(for the details of our games and the definition of Helen's value function, see Section~\ref{games_presentation}). 

Roughly speaking, our main result states that
\begin{align*}
& \limsup_{\eps \rightarrow 0} u^\eps \text{ is a viscosity subsolution of the PDE, and  } \\
& \liminf_{\eps \rightarrow 0} u^\eps \text{ is a viscosity supersolution of the PDE. }
\end{align*}
For the general theory of viscosity solutions to fully nonlinear equations with Neumann (or oblique) boundary condition 
the reader is referred to \cite{user_s_guide,barles_fully_o2,ishii_oblique1991}.
As for the Neumann boundary condition, its relaxation in the viscosity sense was first proposed by Lions \cite{lions_neumann_type}.

Our result is most interesting when the PDE has a comparison principle, i.e. when every subsolution must lie below any supersolution. 
 For such equations, we conclude that $\lim u^\eps$ exists and is the unique viscosity solution of the PDE. 
In the case when $f$ is continuous in all its variable, there are already a lot of comparison and existence results 
for viscosity solutions of second order parabolic PDEs with general Neumann type boundary conditions. 
We refer for this to \cite{barles_fully_o2,barles_quasilinear_elliptic,lions_neumann_type,ishii_oblique1991} and references therein. 
For homogeneous Neumann conditions, Sato \cite{sato_interface} has obtained such a comparison principle for certain parabolic PDEs.

We are interested here in giving a game interpretation for fully nonlinear parabolic and elliptic equations with a Neumann condition.  
Applications of the Neumann condition 
to deterministic optimal control and differential games theory in \cite{lions_neumann_type} rely much on a reflection process, the 
solution of the deterministic Skorokhod problem. 
Its properties in differents situations are studied in many articles such as \cite{tanaka_reflecting,lions_sznitman,dupuis_ishii}. 
The case of the Neumann problem for  the motion by mean curvature was studied by  Giga and Liu \cite{giga_liu_billiard}.
There, a billiard game was introduced to extend the interpretation made by Kohn and Serfaty \cite{kohns_mcm} via the game of Paul and Carol. 
It was based on  the natural idea that a homogeneous Neumann condition will be well-modeled by a reflection on the boundary.
Liu also applies this billiard dynamics to study some first order Hamilton-Jacobi equations 
with Neumann or oblique boundary conditions \cite{liu_billiard_control_1order}. 
Nevertheless, in our case, if we want to give a billiard interpretation with a bouncing rule
which can send the particle far from the boundary, we can only manage to solve the homogeneous case. 
This is not too surprising because the reflection across $\partial \Omega$ is precisely associated to a homogeneous Neumann condition.

Another approach linked to the Neumann condition is  to proceed by penalization on the dynamics. 
For a bounded convex domain, Lions, Menaldi and Sznitman \cite{lions_menaldi_sznitman} construct a sequence of stochastic differential equations 
with a term in the drift coefficients that 
strongly penalizes the process from leaving the domain.
Its solution converges towards a  diffusion process which reflects across the boundary with respect to the normal vector.
Barles and Lions \cite{barles_lions_oblique} also treat  the oblique case by precisely establishing the links
between some approximated processes and the  elliptic operators associated to the original oblique stochastic dynamics. 

Instead of a billiard, our approach here proceeds by a suitable penalization on the dynamics depending on the Neumann boundary condition. 
 It will be favorable to one player or the other according to its sign. We modify the rules of the game only in a small neighborhood of the boundary. 
The particle driven by the players  can leave the domain but then it is projected within. 
This particular move, combined with a proper weight associated to the Neumann boundary condition, gives the required penalization.  
Outside this region, the usual rules are conserved. Therefore the previous analysis within $\Omega$ done by Kohn and Serfaty can be preserved.  
We focus all along this article on the changes near the boundary and their consequences on the global convergence theorem.
In this context, the modification of the rules of the original game introduces many additional difficulties intervening at the different steps of the proof.
Most of all, they are due to the geometry of the domain or the distance to the boundary. 
As a result, our games seem like a natural adaptation of the games proposed by Kohn and Serfaty by permitting to solve 
an inhomogeneous Neumann condition $h$ depending on $x$ on the boundary. 
We only require $h$ to be continuous and uniformly bounded, 
the domain to be $C^2$ and to satisfy some natural geometric conditions in order to ensure the well-posedness of our games.
Moreover our approach can easily be extended both to the oblique and the mixed Neumann-Dirichlet boundary conditions in both parabolic and elliptic settings.
Our games can be compared to those proposed in \cite{kohns} for the elliptic Dirichlet problem: 
if the particle crosses the boundary, the game is immediately stopped and Helen receives a bonus $b(x_F)$ where $b$ corresponds to
 the Dirichlet boundary condition and $x_F$ is the final position. 
Meanwhile, our games cannot stop unexpectedly, no matter the boundary is crossed or not.

Our games, like the ones proposed by Kohn and Serfaty, are deterministic but closely related to a recently developed 
stochastic representation due to Cheridito, Soner, Touzi and Victoir \cite{cheridito_soner_touzi_victoir} 
(their work uses a backward stochastic differential equation, BSDE, whose structure depends on the form of the equation). 

Another interpretation is to look our games as a numerical scheme whose solution is an approximation of a solution of a certain PDE. 
This aspect is classical and has already been exploited in several contexts. 
We mention the work of Peres, Schramm, Sheffield and Wilson~\cite{peres_schramm_sheffield_wilson} who showed that the infinity Laplace equation 
describes the continuum limit of the value function of a two-player, random-turn game called $\eps$\textit{-step tug-of-war}. 
In related work, Armstrong, Smart and Sommersille \cite{armstrong_smart_sommersille} obtained existence, uniqueness and stability results 
for  an infinity Laplace equation with mixed Dirichlet-Neumann boundary terms by comparing solutions of the PDE to subsolutions and supersolutions 
of a certain finite difference scheme, by following a previous work of Armstrong and Smart for the Dirichlet case \cite{armstrong_smart_tow}.

This paper is organized as follows:
\begin{itemize}
 \item Section~\ref{games_presentation} presents the two-person games that we associate with the PDEs \eqref{def_eq_par} and \eqref{def_eq_el}, 
motivating and stating our main results.
 The section starts with a simple case before adressing the general one. 
Understanding our games is still easy, though the technicality of our proofs is increased. 
Since $f$ depends on $u$, the game determines a pair of value functions $u^\eps$ and $v^\eps$. 
Section~\ref{heuristic_derivation} gives a formal argument linking the principle of dynamic programming  
to the PDE in the limit $\eps \rightarrow 0$ 
 and giving the optimal strategies for Helen that will be essential to obtain consistency at Section~\ref{consistance}.
 \item Section~\ref{convergence} addresses the link between our game and the PDE with full rigor.
The proofs of convergence follow the background method of Barles and Souganidis \cite{barles_souga_cv_schemes}, 
 i.e. they use the stability, monotonicity and consistency  of the schemes provided by our games. 
Their theorem states that if a numerical scheme is monotone, stable, and consistent, then the associated 
``lower semi-relaxed limit'' is a viscosity supersolution and the associated  ``upper semi-relaxed limit'' is a viscosity subsolution. 
 The main result in Section~\ref{convergence} is a specialization of their theorem in our framework: 
if $v^\eps$ and $u^\eps$ remain bounded as $\eps \rightarrow 0$ then the lower relaxed semi-limit of $v^\eps$ is a viscosity supersolution  
and the upper relaxed semi-limit of $u^\eps$ is a viscosity subsolution.  
We also have  $v^\eps \leq u^\eps$ with no extra hypothesis in the parabolic setting, or if $f$ is monotone in $u$ in the elliptic setting. 
If the PDE has a comparison principle (see \cite{barles_souga_cv_schemes}) then it follows that 
$\lim u^\eps=\lim v^\eps$ exists and is the unique viscosity solution of the PDE.

 \item The analysis in Section \ref{convergence} shows that consistency and stability imply convergence. 
Sections~\ref{consistance} and \ref{stability} provide the required consistency and stability results. 
The new difficulties due to the penalization corresponding to the Neumann condition arise here. 
The main difficulty is to control the degeneration of the consistency estimate obtained in \cite{kohns}
with respect to the penalization. 
Therefore we will mainly focus on the consistency estimates whereas the needed changes for stability 
will be simply indicated.
 
 \item Section~\ref{generalizations} describes the games associated on the one hand to the oblique problem 
 in the parabolic setting 
and on the other hand to the mixed type Dirichlet-Neumann boundary conditions in the elliptic framework. 
By combining the results associated to the game associated to the Neumann problem in Section~\ref{games_presentation} 
with the ideas already presented in \cite{kohns}, we can obtain the results of convergence. 
\end{itemize}

\textbf{Notation: } The term domain will be reserved for a nonempty,  connected, and open subset of $\R^N$.
If $x,y \in \R^N$,  $\ds \left\langle x,y \right\rangle$ denotes the usual Euclidean inner product 
and $\norm{x}$  the Euclidean length of $x$.
If $A$ is a $N \times N$ matrix, $\norm{A}$ denotes the operator norm 
$\ds \norm{A} = \sup_{\norm{x} \leq 1} \norm{Ax}$. $\mathcal{S}^N$ denotes the set of symmetric $N \times N$  matrices 
and $E_{ij}$ the $(i,j)$-th matrix unit, the matrix whose only nonzero element is equal to 1 and occupies the $(i,j)$-th position.

Let $\mathcal{O}$ be a domain in $\R^N$ and $C^k_b(\mathcal{O})$ be the vector space of $k$-times continuously differentiable functions $u$:  
$\mathcal{O} \rightarrow \R$, such that all the partial derivatives of $u$ up to order $k$ are bounded on $\mathcal{O}$. 
For a domain $\Omega$, we define 
\begin{equation*}
C^k_b(\overline \Omega)= \left\{ u \in L^\infty (\ol \Omega) : \exists  \mathcal{O} \supset  \ol \Omega, \mathcal{O} \text{ domain}, 
\exists v \in C^k_b(\mathcal{O})\text{ s.t. }  u=v_{|\ol \Omega} \right\}.
\end{equation*}
It is equipped with the norm $\norm{\cdot}_{C^k_b(\overline \Omega)}$ given by 
$\ds \norm{\phi}_{C^k_b(\overline \Omega)} =\sum_{i=0}^k\norm{D^i\phi}_{L^\infty(\overline \Omega)}$.

If $\Omega$ is a  smooth domain, say $C^2$, the distance function to $\partial \Omega$ is denoted by $d= d(\cdot, \partial \Omega)$, 
and we recall that, for all $x \in \partial \Omega$, the outward normal $n(x)$ to $\partial \Omega$ at $x$ is given by $n(x) = - Dd(x)$.

Observe that, if $\partial \Omega$ is assumed to be bounded and at least of class $C^{2}$,
any $x\in \R^N$ lying in a sufficiently small neighborhood of the boundary admits a unique projection onto
$\partial \Omega$, denoted by  
\begin{equation*}
 \bar{x} = \proj_{\partial \Omega}(x).
\end{equation*}
In particular, the vector $x- \bar x$ is parallel to $n(\bar x)$.
The projection onto $\ol  \Omega$ will be denoted by $\proj_{\overline \Omega}$. When it is well-defined, it can be decomposed as
\begin{equation*}
\proj_{\overline \Omega} (x) =
 \begin{cases}
 \proj_{\partial \Omega}(x),  & \text{ if } x \notin \Omega,   \\
  x,                         & \text{ if } x\in \Omega .
 \end{cases} 
\end{equation*}
For each $a>0$, we define $ \Omega(a)=\{x\in \overline \Omega, d(x) < a\}$. 
We recall the following classical geometric condition (see e.g. \cite{evans_pde}). 
\begin{Def}[Interior ball condition] 
The domain $\Omega$ satisfies the interior ball condition at $x_0 \in \partial \Omega$ if there exists an open ball $B\subset \Omega$ with $x_0\in \partial B$. 
\end{Def}

We close this introduction by listing our main hypotheses on the form of the PDE. 
First of all we precise some hypotheses on the domain $\Omega$. 
Throughout this article, $\Omega$ will denote a $C^2$-domain.  
In the unbounded case, we impose the following slightly stronger condition than the interior ball condition.
\begin{Def}[Uniform interior/exterior ball condition]  \label{unif_int_ball_cd}
The domain $\Omega$ satisfies the uniform interior ball condition if there exists $r>0$ such that for all $x \in \partial \Omega$ there exists an open ball $B\subset \Omega$ 
with $x\in \partial B$ and radius~$r$. 
Moreover, the domain $\Omega$ satisfies the uniform exterior ball condition if $\R^N \backslash \ol \Omega$ satisfies the uniform interior ball condition.
\end{Def}

We observe that the uniform interior ball condition implies the interior ball condition 
and that both the uniform interior and exterior ball conditions hold automatically for a $C^2$-bounded domain.

The Neumann boundary condition $h$ is assumed to be continuous and uniformly bounded on $\partial \Omega$. 
Similarly, in the parabolic framework, the final-time data $g$ is supposed to be continuous and uniformly bounded on $\ol \Omega$.

The real-valued function $f$ in \eqref{def_eq_par} is defined on $\R \times \ol \Omega \times \R \times \R^N \times \mathcal{S}^N$. 
 It is assumed throughout to be a continuous function of all its variables, and also that 
\begin{itemize}
 \item $f$ is monotone in $\Gamma$ in the sense that 
\begin{equation} \label{ellipticity_f} 
f(t,x,z,p,\Gamma_1+\Gamma_2) \leq f(t,x,z,p,\Gamma_1) \quad \text{ for } \Gamma_2 \geq 0.
\end{equation}
\end{itemize}
In the time-dependent setting \eqref{def_eq_par} we permit $f$ to grow linearly in $|z|$ (so solutions can grow exponentially, but cannot blow up). However
 we require uniform control in $x$ (so solutions remain bounded as $\norm{x} \rightarrow \infty$ with $t$ fixed). In fact we assume that
\begin{itemize}
 \item  $f$ has at most linear growth in $z$ near $p=0$, $\Gamma=0$, in the sense that for any $K$ we have
\begin{equation}  \label{lin_grow}
|f(t,x,z,p,\Gamma)|\leq C_K(1+|z|) , 
\end{equation}
for some constant $C_K \geq 0$, for all $x\in \ol \Omega$ and $t,z \in \R$,  when $\norm{(p,\Gamma)}\leq K$.
 \item $f$ is locally Lipschitz in $p$ and $\Gamma$ in the sense that for any $K$ we have
\begin{equation} \label{loc_lip_p_Gamma}
|f(t,x,z,p,\Gamma)-f(t,x,z,p',\Gamma')|  \leq C_K(1+|z|) \norm{(p,\Gamma)-(p',\Gamma')},       
\end{equation}
for some constant $C_K \geq 0$, for all $x \in \ol \Omega$ and $t,z\in \R$, when $\norm{(p,\Gamma)}+\norm{(p',\Gamma')} \leq K$.
\item $f$ has controlled growth with respect to $p$ and $\Gamma$, in the sense that for some constants $q, r\geq 1$, $C>0$, we have
\begin{equation}  \label{cont_growth_p_Gamma}
|f(t,x,z,p,\Gamma)|\leq  C(1+|z|+\norm{p}^q+ \norm{\Gamma}^r),    
\end{equation}
for all $t,x,z,p$ and $\Gamma$.
\end{itemize}

In the stationary setting \eqref{def_eq_el} our solutions will be uniformly bounded. 
To prove the existence of such solutions we need the discounting to be sufficiently large. We also need analogues of 
\mbox{\eqref{loc_lip_p_Gamma}--\eqref{cont_growth_p_Gamma}}
but they can be local in $z$ since $z$ will ultimately be restricted to a compact set. In fact, we assume that
\begin{itemize}
 \item  There exists $\eta>0$ such that for all $K\geq 0$, there exists $C_K^\ast>0$ satisfying 
\begin{equation} \label{est_f_el_neu}
 |f(x,z,p, \Gamma)| \leq (\lambda -\eta) |z|+ C_{K}^\ast,    
\end{equation}
for all $x\in \ol \Omega$, $z\in \R$, when $\norm{(p,\Gamma)} \leq K$; here $\lambda$ is the coefficient of $u$ in the equation \eqref{def_eq_el}.
\item  $f$ is locally Lipschitz in $p$ and $\Gamma$ in the sense that for any $K$ and $L$ we have
\begin{equation}  \label{loc_lip_p_Gamma_el}
|f(x,z,p,\Gamma)-f(x,z,p',\Gamma')|  \leq C_{K,L} \norm{(p,\Gamma)-(p',\Gamma')}, 
\end{equation}
for some constant $C_{K,L} \geq 0$, for all $x \in \ol  \Omega$, when $\norm{(p,\Gamma)}+\norm{(p',\Gamma')} \leq K$ and $|z|\leq L$.
\item $f$ has controlled growth with respect to $p$ and $\Gamma$, in the sense that for some constants $q, r\geq 1$ and for any $L$ we have
\begin{equation} \label{cont_growth_p_Gamma_el}
|f(x,z,p,\Gamma)|\leq  C_L(1+ \norm{p}^q+ \norm{\Gamma}^r),     
\end{equation}
for some constant $C_L \geq 0$, for all $x$, $p$ and $\Gamma$, and any  $|z|\leq L$.
\end{itemize}

\section{The games}
\label{games_presentation}

This section present our games. We begin by dealing with the linear heat equation. 
Section~\ref{def_eq_par} adresses the time-dependent problem depending non linearly on $u$; our main rigorous result for the time-dependent setting 
is stated here (Theorem~\ref{theo_cv_par_neu}). 
Section~\ref{def_eq_el} discusses the stationary setting and states our main rigorous result for that case (Theorem~\ref{theo_cv_el_neu}).

\subsection{The linear heat equation} 

This section offers a deterministic two-persons game approach to the linear heat equation in one space dimension. 
More precisely, let $a<c$ and $\Omega=]a,c[$.
 We consider the linear heat equation on $\Omega$ with continuous final time data $g$ 
 and Neumann boundary condition $h$  given by
\begin{equation}\label{heat_eq_1D}
 \begin{cases}
 u_t+u_{xx}=0 ,                            & \text{for }x \in \Omega \text{ and }  t<T,  \\ 
\dfrac{\partial u}{ \partial n}(x,t)=h(x), & \text{for }x \in \partial \Omega=\{a,c\} \text{ and }  t<T, \\
 u(x,T)=g(x)  ,                            & \text{for } x \in \ol \Omega \text{ and } t=T. 
\end{cases}
\end{equation}
Our goal is to capture, in the simplest possible setting, how a homogeneous Neumann condition can be retrieved through a repeated deterministic game. 
The game discussed here shares many features with the ones we will introduce in Sections \ref{def_jeu_parabolique}--\ref{rules_el_game}, 
though it is not a special case. In particular, it allows to understand the way we need to modify the rules of the pioneering games 
proposed by Kohn and Serfaty in \cite{kohns} in order to model the Neumann boundary condition.

There are two players, we call them Mark and Helen. A small parameter $\eps>0$ is fixed as are the final time $T$,  
``Helen's payoff'' (a continuous function $g$: $[a,c] \rightarrow \R$) and a ``coupon profile'' close to the boundary (a function $h$: $\{a,c\} \rightarrow \R)$.  
The state of the game is described by its ``spatial position'' $x\in \ol \Omega$ and ``Helen's score'' $y\in \R$. We suppose the game begins at time $t_0$. 
Since time steps are increments of $\eps^2$, it is convenient to assume that $T-t_0=K\eps^2$, for some $K$.
 
When the game begins, the position can have any value $x_0 \in \ol \Omega$; Helen's initial score is $y_0=0$. The rules are as follows: if, at time $t_j=t_0+j\eps^2$, 
the position is $x_j$  and Helen's score is $y_j$, then 
\begin{itemize}
 \item Helen chooses a real number $p_j$. 
 \item After seeing Helen's choice, Mark chooses $b_j=\pm 1$ which gives an intermediate position $\hat x_{j+1}=x_j+\Delta \hat x_j$ where 
\begin{equation*}
\Delta \hat x_j=\sqrt{2} \eps b_j \in \R.
\end{equation*}
 This position $\hat x_{j+1}$ determines the next position $x_{j+1}=x_j + \Delta x_j$ at time $t_{j+1}$ by the rule 
\begin{equation*}
x_{j+1}=\proj_{\overline \Omega}(\hat x_{j+1}) \in \overline \Omega , 
\end{equation*}
and Helen's score changes to 
\begin{equation}\label{Helen_score_heat} 
y_{j+1}=y_j+ p_j \Delta \hat x_j- \norm{ x_{j+1} - \hat x_{j+1}} h(x_j+\Delta x_j) . 
\end{equation}
 \item The clock moves forward to $t_{j+1}=t_j+\eps^2$ and the process repeats, stopping when $t_K=T$.
 \item At the final time $t_K=T$ a bonus $g(x_K)$ is added to Helen's score, where $x_K$ is the final-time position. 
\end{itemize}
\begin{remark} To give a sense to \eqref{Helen_score_heat} for all $\Delta x_j$, 
 the function $h$, which is defined only on $\{a,c\}$,  can be extended on $]a,c[$ by any function $\Omega \rightarrow \R$ since
 $\norm{ x_{j+1} - \hat x_{j+1}}$ is different from zero if and only if $\hat x_{j+1}\notin \ol \Omega$.
Moreover, by comparing the two moves $\Delta \hat x_j$ and $\Delta x_j$,  
 it is clear that $\norm{ x_{j+1} - \hat x_{j+1}}= \norm{\Delta x_j - \Delta \hat x_j }$.
\end{remark}
Helen's goal is to maximize her final score, while Mark's goal is to obstruct her. We are interested in Helen's ``value function'' 
$u^\eps(x_0,t_0)$, defined formally as her maximum worst-case final score starting from $x_0$ at time $t_0$. 
It is determined by the dynamic programming principle
\begin{equation}\label{dpp_heat_equation}
u^{\eps}(x,t_j)=\max_{p\in \R} \min_{b=\pm 1} \left[ u^{\eps}(x+\Delta x,t_{j+1})   - p \Delta \hat x + \norm{\Delta \hat x - \Delta x} h(x+\Delta x) \right],
\end{equation}
where $\Delta \hat x = \sqrt{2} \eps b$ and $\Delta x= \proj_{\overline \Omega}(x+\Delta \hat x) -x$, associated with the final-time condition 
\begin{equation*}
 u^\eps(x,T)=g(x).
\end{equation*}
Evidently, if $t_0=T - K\eps^2$ then
\begin{equation}\label{rec_dpp_heat}
u^\eps(x_0,T_0) = \max_{p_{0} \in \R}  \min_{b_{0}=\pm 1}  \cdots  \max_{p_{K-1} \in \R}  \min_{b_{K-1}=\pm 1}
\left\{g(x_K) + \sum_{j=0}^{K-1}  - \sqrt{2} \eps b_j p_j +    
 \norm{\Delta \hat x_j - \Delta x_j} h(x_j+\Delta x_j)  \right\}, 
\end{equation}
where $\Delta \hat x_j = \sqrt{2} \eps b_j$ and $\Delta x_j= \proj_{\overline \Omega}(x_j+\Delta \hat x_j) -x_j$.
In calling this Helen's value function, we are using an established convention from the theory of discrete-time, two person games (see e.g.~\cite{friedman}). 

By introducing the operator $L_\eps$ defined by
\begin{equation}\label{op_chaleur1d}
L_{\eps}  [x,\phi]  =  \max_{p\in \R}   \min_{b= \pm 1} 
 \left[\phi \left( x +\Delta x \right) - p   \Delta \hat{x} + \norm{\Delta \hat x - \Delta x}  h (x+\Delta x )\right],
\end{equation}
where $\Delta \hat x = \sqrt{2} \eps b$ and $\Delta x= \proj_{\overline \Omega}(x+\Delta \hat x) -x$, 
the dynamic programming principle \eqref{dpp_heat_equation} can be written in the form 
\begin{equation}\label{dpp_heat_eq_op}
 u^{\eps}(x,t)=L_{\eps}  [x, u^{\eps}(\cdot,t_{}+\eps^2)]. 
\end{equation}
We now formally argue that $u^\eps$ should converge as $\eps \rightarrow 0$ to the solution of the linear heat equation~\eqref{heat_eq_1D}. 
The procedure for formal passage from the dynamic programming principle to the associated PDE is familiar: 
 we suppress the dependence of $u^\eps$ on $\eps$ and we assume $u$ is smooth enough to use the Taylor expansion. The first step leads to
\begin{equation}\label{dpp_heat_eq_formal}
 u^{}(x,t)\approx L_{\eps}  [x, u^{}(\cdot,t_{}+\eps^2)].
\end{equation}
For the second step we need to compute $L^\eps$ for a $C^2$-function $\phi$.  By the Taylor expansion
\begin{align*}
\phi(x+\Delta x) & =\phi(x)+ \phi_x(x) \Delta x +\frac{1}{2} \phi_{xx}(x)  (\Delta x)^2 +O(\eps^{3}) \\
          & =\phi(x)+ \phi_x(x) \Delta \hat x + \norm{\Delta \hat x-\Delta x} \phi_x(x) n(\ol x)+\frac{1}{2} \phi_{xx}(x) (\Delta x)^2 +O(\eps^{3}), 
\end{align*}
where $\ol x= \proj_{\partial \Omega}(x)$, $\Delta \hat x - \Delta x=\norm{\Delta \hat x - \Delta x}  n(\ol x)$ with $n$ defined on $\partial \Omega $ by 
$n(x)=1$ if $x=c$ and $n(x)= - 1$ if $x=a$.  
Substituting this expression in \eqref{op_chaleur1d}, we deduce that for all $C^2$-function $\phi$, 
\begin{equation} \label{heat_op1}
L_{\eps}  [x,\phi]=\phi(x)    
 +\max_{p\in \R}   \min_{b= \pm 1} 
  \left[(\phi_x - p) \Delta \hat{x}  +\frac{1}{2} \phi_{xx} (\Delta x)^2  + \norm{\Delta \hat x - \Delta x} \big\{ 
 h(x+\Delta x )  -  n(\ol x)\phi_x \big\}  \right]+o(\eps^2).
 \end{equation}
It remains to compute the max min. 
If $d(x)>\sqrt{2} \eps$, we always have $\Delta x=\Delta \hat x=\sqrt{2}\eps b$, so that the boundary is never crossed and we retrieve 
the usual situation detailed in \cite[Section~2.1]{kohns}: Helen's optimal choice is $p=\phi_x$ and $L_{\eps}[x,\phi] = \phi(x) +\eps^2 \phi_{xx}(x)+o(\eps^2)$.
If $d(x)<\sqrt{2} \eps$, we still have  $\Delta \hat x=\sqrt{2} b \eps$ but there is a change: 
 if the boundary is crossed, $\Delta x= d(x)$ and $\norm{\Delta \hat x - \Delta x}=\sqrt{2} \eps-d(x)$.
Suppose that Helen has chosen $p\in \R$. Considering the min in \eqref{heat_op1}, Mark only has two possibilities $b\in \{\pm 1\}$. 
More precisely, suppose that $x$ is close to $c$ so that $\ol x=c$ and $n(\ol x)=1$; the case when $x$ is close to $a$ is strictly parallel. 
If Mark chooses $b=1$, the associated value is 
\begin{equation*}
V_{p,+}= \sqrt{2}(\phi_x - p) \eps +\frac{1}{2} \phi_{xx} d^2(x)  + (\sqrt{2}  \eps- d(x)) ( h(c) - \phi_x ),    
\end{equation*}
while if Mark chooses $b=-1$, the associated value is 
\begin{equation*}
V_{p,-} = - \sqrt{2}(\phi_x - p)  \eps + \phi_{xx} \eps^2.  
\end{equation*}
To determine his strategy, Mark compares $V_{p,-}$ to $ V_{p,+}$. He chooses $b=-1$ if $V_{p,-}< V_{p,+}$, i.e. if 
\begin{align*}
\sqrt{2}(\phi_x - p) \eps  +\frac{1}{2} \phi_{xx} d^2(x)  + (\sqrt{2} \eps- d(x)) ( h(c)-\phi_x)> -\sqrt{2}(\phi_x - p) \eps + \phi_{xx} \eps^2,
 \end{align*}
that we can rearrange into
\begin{equation*}
2 \sqrt{2}(\phi_x - p) \eps > \phi_{xx} \left(\eps^2- \frac{d^2(x)}{2}\right) - (\sqrt{2} \eps- d(x)) \left[ h(c)  - \phi_x \right].
\end{equation*}
This last inequality yields an explicit condition on the choice of $p$ previously made by Helen
\begin{equation}\label{def_p_opt_heat} 
p < p_\text{opt}  
:= \phi_x + \frac{1}{2} \left(1- \frac{d(x)}{\sqrt{2}\eps} \right) \left[ h(c)-\phi_x \right]+  \frac{1}{2\sqrt{2}} \phi_{xx} 
\left(1- \frac{d^2(x)}{2\eps^2}\right)\eps .
\end{equation}
Meanwhile Mark chooses $b=1$ if $V_{p,+}< V_{p,-}$, which leads to the reverse inequality $p > p_\text{opt}$. 
The situation when $V_{p,+}= V_{p,-}$ obviously corresponds to $p = p_\text{opt} $.
We deduce that
\begin{equation*}
L_\eps[x,\phi] = \max \left[ \max_{p\leq p_\text{opt}} V_{p,-}, V_{p_\text{opt},-},  \max_{p\geq p_\text{opt}} V_{p,+}  \right].
\end{equation*}
Helen wants to optimize her choice of $p$. 
The functions $V_{p,+}$ and  $V_{p,-}$ are both affine on $\phi_x-p$. 
The first one is decreasing  while the second is increasing with respect to $p$. 
As a result, we deduce that Helen's optimal choice is $p=p_\text{opt}$ as defined in \eqref{def_p_opt_heat} 
and $L_\eps[x,\phi] = V_{p_\text{opt},+} =  V_{p_\text{opt},-} $.
We notice that Helen behaves optimally by becoming indifferent to Mark's choice; our games will not always conserve this feature, 
 which was observed in \cite{kohns}. Finally, for all $C^2$-function~$\phi$, we have 
\begin{equation} \label{lem_cons_heat_neu}  
 L_\eps[x,\phi]=  \phi(x)  \\
+ \begin{cases}
 \ds   \dfrac{\eps}{\sqrt{2}} \left(1- \frac{d(x)}{\sqrt{2}\eps}\right) \left[ h(\ol x)  - n (\ol x) \phi_x(x)  \right] 
+ \dfrac{\eps^2}{2} \phi_{xx}(x) \left(1+ \frac{d^2(x)}{2\eps^2} \right) +o(\eps^2),  & \text{if } d(x)\leq \sqrt{2} \eps,  \\ 
 \ds   \eps^2  \phi_{xx}(x) +o(\eps^2),  & \text{if } d(x)\geq \sqrt{2} \eps.
\end{cases}
 \end{equation} 
Since $u$ is supposed to be smooth, the Taylor expansion on $t$ yields that $u(\cdot, t+\eps^2)=u(\cdot,t)+ u_t(\cdot,t) \eps^2+o(\eps^2)$ 
and we formally derive the PDE by plugging \eqref{lem_cons_heat_neu}  in \eqref{dpp_heat_eq_formal}. This gives
\begin{equation}\label{op_heat_order2}
0\approx \eps^2 u_t+
\begin{cases}
\ds  \dfrac{\eps}{\sqrt{2}} \left(1- \frac{d(x)}{\sqrt{2}\eps}\right) \left[ h(\ol x)  - n (\ol x) u_x   \right] 
+ \dfrac{\eps^2}{2} u_{xx} \left(1+ \frac{d^2(x)}{2\eps^2} \right) +o(\eps^2) , & \text{if } d(x)\leq \sqrt{2} \eps ,  \\ 
\ds  \eps^2  u_{xx} +o(\eps^2) ,   & \text{if } d(x)\geq \sqrt{2} \eps .
\end{cases}
\end{equation}
If $x \in \Omega$, for $\eps$ small enough, the second alternative in \eqref{op_heat_order2} is always valid
 so that we deduce from the $\eps^2$-order terms in \eqref{op_heat_order2} that $u_t+u_{xx}=0$.
If $x$ is on the boundary $\partial \Omega$, then  $d(x)=0$, $\ol x=x$ and the first possibility in~\eqref{op_heat_order2} is always satisfied. 
We observe that the $\eps$-order term is predominant since $\eps \gg \eps^2$. By dividing by $\eps$ and 
letting $\eps \rightarrow 0$, we obtain $h(x)-u_x(x) \cdot n(x)=0$.

Now we present a financial interpretation of this game. Helen plays the role of a hedger or an investor, 
while Mark represents the market. The position $x$ is a stock price 
which evolves in $\ol \Omega$ as a function of time $t$, starting at $x_0$ at time $t_0$ and the boundary $\partial \Omega$ plays 
the role of barriers which additionally determine a coupon when the stock price crosses $\partial \Omega$. 
The small parameter $\eps$ determines both the stock price increments $\Delta \hat x \leq \sqrt{2} \eps$ and the time step $\eps^2$. 
Helen's score keeps track of the profits and losses generated by her hedging activity.

Helen's situation is as follows: she holds an option that will pay her $g(x(T))$ at time $T$ ($g$ could be negative). 
Her goal is to hedge this position by buying or selling the stock at each time increment. She can borrow and lend money without paying or collecting any interest, 
and can take any (long or short) stock position she desires. 
At each step, Helen chooses a real number $p_j$ (depending on $x_j$ and $t_j$), then adjusts her portfolio so it contains  $-p_j$ units of stock 
(borrowing or lending to finance the transaction, so there is no change in her overall wealth). Mark sees Helen's choice. Taking it into account, 
he makes the stock go up or down (i.e. he chooses $b_j= \pm 1$), trying to degrade her outcome. The stock price changes from 
$x_j$ to $x_{j+1}=\proj_{\ol \Omega}(x_j+\Delta \hat x_j)$, and Helen's wealth changes by
$-\sqrt{2} \eps b_j p_j + \norm{\Delta \hat x_j - \Delta x_j} h(x_j+\Delta x_j)$ (she has a profit if it is positive, a loss if it is negative).
The term $\norm{\Delta \hat x_j - \Delta x_j} h(x_j+\Delta x_j)$ is a coupon that will be produced 
only if the special event $\Delta \hat x_j \notin \Omega$ happens.
The hedger must take into account the possibility of this new event. The hedging parameter $p_j$ is modified close to the boundary 
but the hedger's value function is still independent from the variations of the market. 
At the final time Helen collects her option payoff $g(x_K)$. 
If Helen and Mark both behave optimally at each stage, then we deduce by \eqref{rec_dpp_heat} that 
\begin{equation*}
 u^\eps(x_0,t_0)+ \sum_{j=0}^{K-1} \sqrt{2} \eps b_j p_j - \norm{\Delta \hat x_j - \Delta x_j} h(x_j+\Delta x_j)=g(x_K).
\end{equation*}
Helen's decisions are in fact identical to those of an investor hedging an option with payoff $g(x)$ and coupon $h(x)$ 
if the underlying asset crosses the barrier $\partial \Omega$ in a binomial-tree market with $\Delta \hat x =\sqrt{2} \eps$ at each timestep.

\subsection{General parabolic equations}

\label{def_jeu_parabolique}

This section explains what to do when $f$ depends on $Du$, $D^2u$ and also on $u$. 
We also permit dependence on $x$ and $t$, so we are now discussing a fully-nonlinear (degenerate) parabolic equation of the form
\begin{equation} \label{eq_neumann_evol}
\begin{cases}
  \partial_t u - f(t,x,u,Du,D^2u)=0,  &  \text{ for } x\in \Omega \text{ and }t<T , \\ 
  \left\langle D u(x,t), n(x) \right\rangle=h(x), &  \text{ for } x\in \partial \Omega \text{ and }t<T , \\ 
  u(x,T)  =g(x),   &\text{ for } x \in  \overline \Omega, 
\end{cases}
\end{equation}
where $\Omega$ is a $C^2$-domain satisfying both the uniform interior and exterior ball conditions and the boundary condition $h$ and the final-time data $g$ 
are uniformly bounded, continuous, depending only on $x$. 

There are two players, Helen and Mark; a small parameter $\eps$ is fixed.
Since the PDE is to be solved in $\Omega$, Helen's final-time bonus $g$ is now a function of $x\in \ol \Omega$ 
and Helen's coupon profile $h$ is a function of $x\in \partial \Omega$.
The state of the game is described by its spatial position $x\in \overline \Omega$ and Helen's debt $z\in \R$. 
Helen's goal is to minimize her final debt, while Mark's is to obstruct her.

The rules of the game depend on three new parameters, $\alpha, \beta, \gamma >0$ whose presence represents no loss of generality. Their role 
 will be clear in a moment. The requirements
\begin{equation} \label{condition_pas}
\alpha<1/3, 
\end{equation}
and
\begin{equation}\label{cd_coeff_gen}
\alpha+\beta<1,\qquad  2\alpha+\gamma<2, \qquad \max(\beta q, \beta r)<2 , 
\end{equation}
will be clear in the explanation of the game. However, the proof of 
convergence in Section \ref{convergence} and consistency in Section \ref{consistance}  needs more: there we will require
\begin{equation}\label{cd_coeff_classiq}
\gamma<1-\alpha, \quad \beta(q-1)<\alpha+1, \quad \gamma(r-1)<2 \alpha, 
\quad \gamma r<1+\alpha.
\end{equation}
These conditions do not restrict the class of PDEs we consider, since for any $q$ and $r$ 
there exist $\alpha$, $\beta$ and $\gamma$ with  the desired properties.

Using the language of our financial interpretation: 
\begin{enumerate}
 \item[a)] First we consider $U^\eps(x,z,t)$, Helen's optimal wealth at time $T$, if initially at time $t$ the stock price is $x$ 
and her wealth is $-z$. 
 \item[b)] Then we define $u^\eps(x,t)$ or $v^\eps(x,t)$ as, roughly speaking, the initial debt Helen should have at time $t$ 
to break even at time $T$.
\end{enumerate}
The  proper  definition of $U^\eps(x,z,t)$ involves a game similar to that of Section \ref{heat_eq_1D}. 
The rules are as follows: if at time $t_j=t_0+j \eps^2$, the position is $x_j$ and Helen's debt is $z_j$, then
\begin{enumerate}
 \item Helen chooses a vector $p_j \in \R^N$ and a 
 matrix $\Gamma_j \in \mathcal{S}^N$, restricted by
\begin{equation}
 \norm{p_j}\leq \eps^{-\beta}, \norm{\Gamma_j}\leq \eps^{-\gamma} \label{p_beta_gamma_new}.
\end{equation}
 \item Taking Helen's choice into account, Mark chooses the stock price $x_{j+1}$ so as to degrade Helen's outcome. 
Mark chooses an intermediate point $\hat x_{j+1}=x_j +\Delta \hat x_j \in \R^N$ such that 
\begin{equation}
\left\|\Delta \hat x_j\right\| \leq \eps^{1-\alpha}.
\label{moving1_new}
\end{equation}
This position $\hat x_{j+1}$ determines the new position $x_{j+1}=x_j+\Delta x_j \in \overline \Omega$ at time $t_{j+1}$ by the rule
\begin{equation}
x_{j+1}= \proj_{\overline \Omega} (\hat x_{j+1}).  
\label{moving2_new} 
\end{equation}
\item Helen's debt changes to 
\begin{equation}
z_{j+1}=z_j +  p_j\cdot \Delta \hat x_j +\frac{1}{2}  \left\langle \Gamma_j \Delta \hat x_j,\Delta \hat x_j \right\rangle
+\eps^2 f(t_j,x_j,z_j,p_j,\Gamma_j) - \norm{\Delta \hat x_j - \Delta x_j}  h(x_j+\Delta x_j ) .
\label{Helens_debt}
\end{equation}
\item The clock steps forward to $t_{j+1}=t_j+\eps^2$  and the process repeats, stopping when $t_K=T$. 
At the final time Helen receives $g(x_K)$ from the option.
\end{enumerate}

This game is well-posed for all $\eps>0$ small enough. 
As mentioned in the introduction, the uniform exterior ball condition holds automatically for a $C^2$-bounded domain.
In this case,  by compactness of $\partial \Omega$, there exists $\eps_\ast>0$ such that 
$\proj_{\overline \Omega} $ is well-defined for all $x\in \Omega$ such that $d(x) \leq \eps_\ast$. 
It can be noticed that an unbounded $C^2$-domain, even with bounded curvature, does not generally satisfy this condition.  
Since the domain $\Omega$ satisfy the uniform exterior ball condition given by Definition \ref{unif_int_ball_cd} 
for a certain $r$, the projection is well-defined on the tubular neighborhood $\{x\in \R^N  \backslash \Omega, d(x)< r/2\}$  of the boundary.

\begin{figure}[t]
\begin{center}
 \scalebox{0.7}{
\includegraphics{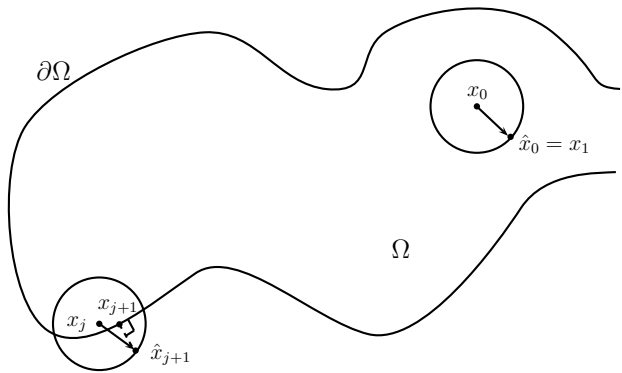}  
 }
\end{center}
\caption{Rules of the game, near the boundary and inside the domain.} 
\label{full_domain_rules_game} 
\end{figure}

\begin{remark}\label{extension_h}
 To give a sense to \eqref{Helens_debt} for all $\Delta x_j$, the function $h$ which is defined only on the boundary 
can be extended on $\ol \Omega $ by any function $\Omega \rightarrow \R$ since $\norm{x_{j+1} - \hat x_{j+1}}$ is different from zero if and only if 
$\hat x_{j+1} \notin \ol \Omega$. Moreover, by comparing $\Delta \hat x_j$ and $\Delta x_j$, one gets the relation 
\begin{equation*}
x_{j+1}=\hat x_{j+1} + \Delta x_j - \Delta \hat x_j. 
\end{equation*}
If $\hat x_{j+1}\in \Omega$, then $x_{j+1}= \hat x_{j+1}$ and the rules of the usual game \cite{kohns} are retrieved. 
Figure \ref{full_domain_rules_game} presents the two geometric situations for the choice for Mark: $B(x, \eps^{1-\alpha}) \subset \Omega$ or not. 
\end{remark}
Helen's goal is to maximize her worst-case score at time $T$, and Mark's is to work against her. Her  value function is 
\begin{equation} \label{}
 U^\eps(x_0,z_0,t_0) =\max_{\text{Helen's choices}} \left[ g(x_K) - z_K \right].
\end{equation}
It is characterized by the dynamic programming principle
\begin{equation}  \label{dpp_U}
U^\eps(x,z,t_j) = \max_{p,\Gamma} \min_{\Delta \hat x} U^{\eps} (x+\Delta x, z+\Delta z, t_{j+1}) 
\end{equation}
together with the final-time condition $U^\eps(x,z,T)= g(x) - z$. 
Here $\Delta \hat x$ is $\hat x_{j+1} - x_j$, $\Delta x$ is determined by
\begin{equation}\label{exp_delta_x}
 \Delta x=  x_{j+1} - x_j= \proj_{\ol \Omega} (x_j+\Delta \hat x_j) - x_j, 
\end{equation}
 and $\Delta z=z_{j+1} - z_j$ is given by \eqref{Helens_debt}, and the optimizations are constrained by \eqref{p_beta_gamma_new} and \eqref{moving1_new}. 
It is easy to see that the max/min is achieved and  is a continuous function of $x$ and $z$ at each discrete time
(the proof is by induction backward in time, like the argument sketched in \cite{kohns}).

When $f$ depends on $z$,  the function $z\mapsto U^\eps(x,z,t)$ can be nonmonotone, so we must distinguish between the minimal and maximal debt with 
which Helen breaks even at time $T$. 
\label{def_u_v_eps_par}
Thus, following \cite{cheridito_soner_touzi_victoir}, we define
\begin{equation}\label{def_u_eps_par}
u^\eps(x_0,t_0)=\sup\{z_0 :  U^\eps(x_0,z_0,t_0)\geq 0 \} 
\end{equation}
and
\begin{equation}\label{def_v_eps_par}
v^\eps(x_0,t_0)=\inf \{z_0 :  U^\eps(x_0,z_0,t_0)\leq 0 \}, 
\end{equation}
with the convention that the empty set has $\sup= - \infty$ and $\inf= \infty$. 
Clearly $v^\eps \leq u^\eps$, and $u^\eps(x,T) = v^\eps(x,T)=g(x)$. 
Since the definitions of $u^\eps$ and $v^\eps$ are implicit, these functions can not be characterized by a dynamic programming principle. 
However we still have two ``dynamic programming inequalities''.
\begin{prop} 
 If $u^\eps(x,t)$ is finite then
\begin{multline}\label{dyn_prog_ineq_sub_new}
u^\eps(x,t) \leq \sup_{p,\Gamma} \inf_{\Delta \hat x} \left[u^\eps(x+\Delta x ,t+\eps^2) \right.  \\ 
 - \left.  
  \left(p\cdot \Delta \hat x  +\frac{1}{2}  \left\langle \Gamma \Delta \hat x ,\Delta \hat x \right\rangle
+\eps^2 f(t,x, u^\eps(x,t),p ,\Gamma) -  \norm{\Delta \hat x - \Delta x} h(x+\Delta x)\right)  \right].
\end{multline}
Similarly, if $v^\eps(x,t)$ is finite then
\begin{multline}\label{dyn_prog_ineq_super_new}
v^\eps(x,t) \geq \sup_{p,\Gamma} \inf_{\Delta \hat x} \left[v^\eps(x+\Delta x ,t+\eps^2)\right.   \\
 - \left.    \left(p\cdot \Delta \hat x 
+\frac{1}{2}  \left\langle \Gamma \Delta \hat  x,\Delta \hat x \right\rangle
+\eps^2 f(t,x, v^\eps(x,t),p ,\Gamma) - \norm{\Delta \hat x - \Delta x} h(x+\Delta x)  \right) \right] .
\end{multline}
The sup and inf are constrained by \eqref{p_beta_gamma_new} and \eqref{moving1_new} 
 and $\Delta x$ is determined by \eqref{exp_delta_x}.
\end{prop}

\begin{proof}
The argument follows the same lines as the proof of the dynamic programming inequalities given in \cite[Proposition 2.1]{kohns}. 
For sake of completeness we give here the details. 
To prove~\eqref{dyn_prog_ineq_sub_new}, consider $z=u^\eps(x,t)$. By the definition of $u^\eps$ (and remembering that $U^\eps$ is continuous) we have
$U^\eps(x,z,t)=0$. Hence writing \eqref{dpp_U}, we have
\begin{equation*}
0=\max_{p,\Gamma} \min_{\Delta \hat x} U^{\eps} \left(x+\Delta x, z+  p\cdot \Delta \hat x +\frac{1}{2}  \langle \Gamma \Delta \hat x, \Delta \hat x \rangle
+\eps^2 f(t,x,z,p, \Gamma) - \norm{\Delta \hat x - \Delta x} h(x+\Delta x), t+\eps^2 \right). 
\end{equation*}
We conclude that there exist $p, \Gamma$ (constrained by \eqref{p_beta_gamma_new})  such that for all $\Delta \hat x$ constrained by \eqref{moving1_new},
determining $\Delta x$ by \eqref{exp_delta_x}, we have
\begin{equation*}
U^{\eps}\left(x+\Delta x, z+ p\cdot \Delta \hat x +\frac{1}{2}  \langle \Gamma \Delta \hat x, \Delta \hat x \rangle
+\eps^2 f(t,x,z,p, \Gamma) - \norm{\Delta \hat x - \Delta x} h(x+\Delta x)   , t+\eps^2\right) \geq 0.
\end{equation*}
By the definition of $u^\eps$ given by \eqref{def_u_eps_par}, this implies that
\begin{equation*}
z+ p\cdot \Delta \hat x +\frac{1}{2}  \langle \Gamma \Delta \hat x, \Delta \hat x \rangle
+\eps^2 f(t,x,z,p, \Gamma) - \norm{\Delta \hat x - \Delta x} h(x+\Delta x)
\leq u^{\eps}(x+\Delta x, t+\eps^2).
\end{equation*}
In other words, there exist $p, \Gamma$ such that for every $\Delta \hat x$, determining $\Delta x$ by \eqref{exp_delta_x},
\begin{equation*}
z \leq u^{\eps}(x+\Delta x, t+\eps^2) - \left( p\cdot \Delta \hat x +\frac{1}{2}  \langle \Gamma \Delta \hat x, \Delta \hat x \rangle
+\eps^2 f(t,x,z,p, \Gamma) - \norm{\Delta \hat x - \Delta x} h(x+\Delta x) \right).
\end{equation*}
Recalling that $z=u^\eps(x,t)$ and passing to the inf and sup, we get \eqref{dyn_prog_ineq_sub_new}. 
The proof of \eqref{dyn_prog_ineq_super_new} follows exactly the same lines. 
\end{proof}

To define viscosity subsolutions and supersolutions, we shall follow the Barles and Perthame procedure~\cite{barles_perthame1},
let us recall the upper and lower relaxed semi-limits defined for $(t,x)\in [0,T]\times \overline{\Omega}$ as
\begin{equation}\label{def_bp_subsup_visc}
 \bar{u}(x,t) =\limsup_{\substack{  y \rightarrow x,  y  \in \overline \Omega  \\ t_j \rightarrow t   \\  \eps \rightarrow 0}} u^\eps(y,t_j)   \quad
\text{ and }  \quad
 \underline v(x,t) =\liminf_{\substack{  y \rightarrow x,  y  \in \overline \Omega  \\ t_j \rightarrow t   \\  \eps \rightarrow 0}} v^\eps(y,t_j) ,
\end{equation}
where the discrete times are $t_j=T-j\eps^2$.
We shall show, under suitable hypotheses, that $\udl v$ and $\ol u$ are respectively viscosity super and subsolutions of \eqref{eq_neumann_evol}. 
Before stating our rigorous result in Section~\ref{rigorous_result_par},  
the next section presents the heuristic derivation of the PDE \eqref{eq_neumann_evol} through the optimal strategies of Helen and Mark.

\subsubsection{Heuristic derivation of the optimal player strategies}
\label{heuristic_derivation}
We now formally show that $u^\eps$ should converge as $\eps \rightarrow 0$ to the solution of \eqref{eq_neumann_evol}. Roughly speaking, 
the PDE~\eqref{eq_neumann_evol} is the formal Hamilton Jacobi Bellman equation associated to the two-persons game presented at the beginning of the 
present section.
The procedure for formal derivation from the dynamic programming principle to a corresponding PDE is classical: we assume $u^\eps$ and $v^\eps$ coincide and 
are smooth to use Taylor expansion, suppress the dependence of $u^\eps$ and $v^\eps$ on $\eps$ and finally make $\eps \rightarrow 0$.
That has already been done for $x$ far from the boundary in \cite[Section~2.2]{kohns} for $f$ depending only on $(Du,D^2u)$. 
We now suppose that $x$ is close enough of the boundary so that $\hat x$ can be nontrivial. 
By assuming $u^\eps= v^\eps$ as announced and suppressing the dependence of $u^\eps$ 
on $\eps$, the two dynamic programming inequalities \eqref{dyn_prog_ineq_sub_new} and 
\eqref{dyn_prog_ineq_super_new} give the programming equality 
\begin{multline}
 \label{dyn_prog_eq_formal}
u(x,t) \approx \sup_{p,\Gamma} \inf_{\Delta \hat x} \left[u(x+\Delta x ,t+\eps^2) \right. \\
- \left. \left(p\cdot \Delta \hat x 
+\frac{1}{2}  \left\langle \Gamma \Delta \hat x ,\Delta \hat x \right\rangle
+\eps^2 f(t,x, u(x,t),p ,\Gamma) -  \norm{\Delta \hat x - \Delta x} h(x+\Delta x) \right)  \right]. 
\end{multline}
Remembering that $\Delta \hat x$ is small, if $u$ is assumed to be smooth, we obtain 
\begin{align*}
u (  x +\Delta x,t+\eps^2)&  + \norm{\Delta \hat x - \Delta x}  h(x+\Delta x )  \\
&  \approx u(x,t)+\eps^2  u_t + Du\cdot \Delta x +\frac{1}{2} \left\langle D^2u \Delta x, \Delta x \right\rangle 
  + \norm{\Delta \hat x - \Delta x}  h(x+\Delta x ) \\
& \approx  u(x,t)+\eps^2  u_t + Du  \cdot \Delta \hat x + \norm{\Delta \hat x - \Delta x} \left\{ h(x+\Delta x )
 - Du \cdot n(x+\Delta x) \right\}  +\frac{1}{2} \left\langle D^2u \Delta x, \Delta x \right\rangle , 
\end{align*}
since the outer normal can be expressed by 
$\ds n(x+\Delta x)=  -  \frac{\Delta x - \Delta \hat x}{\norm{\Delta \hat x - \Delta x}}$ if $\hat x \notin \Omega$. 
Substituting this computation in \eqref{dyn_prog_eq_formal}, and rearranging the terms, we get 
\begin{multline} \label{pb_heur_gen}
0 \approx \eps^2  u_t+ \max_{p,\Gamma   }   \min_{\Delta \hat x} 
\left[ (Du - p)\cdot \Delta \hat x  + \norm{\Delta \hat x - \Delta x} \{ h(x+\Delta x ) - Du \cdot n(x+\Delta x) \}  \right. \\
\left.   +\frac{1}{2} \left\langle D^2u  \Delta x, \Delta x \right\rangle  - \frac{1}{2} \left\langle \Gamma \Delta \hat x, \Delta \hat  x \right\rangle 
  - \eps^2 f \left(t, x, u,  p,\Gamma \right) \right],
\end{multline}
where $u$, $Du$, $D^2u$ are evaluated at $(x,t)$. We have ignored the upper bounds in \eqref{p_beta_gamma_new} since they allow $p$, $\Gamma$
to be arbitrarily large in the limit $\eps \rightarrow 0$ (we shall of course be more careful in Section~\ref{consistance}).

If the domain $\Omega$ does not satisfy the uniform interior ball condition, $\Omega$ can present an infinity number of 
``neck pitchings'' of neck size arbitrarily small. 
To avoid this situation, the uniform interior ball condition is used to impose a strictly positive lower bound  on these necks. 
If $x$ is supposed to be extremely close to the $C^2$-boundary and \mbox{$\norm{\Delta \hat x} \leq \eps^{1-\alpha}$}, 
 the boundary looks like a hyperplane orthogonal to the outer normal vector $n(\bar x)$, where  
$\bar x$ is the projection of $x$ on the boundary $\partial \Omega$ (see Figure \ref{drawing_formal_derivation}). 
By Gram-Schmidt process, we can find some vectors $e_2, \cdots, e_N$ such that $(e_1=n(\bar x), e_2, \cdots, e_N)$ form 
 an orthonormal basis of $\R^N$. In this basis, denote
\begin{equation} \label{decomp_p_Gamma}
 p=p_1 n(\bar x) + \widetilde p \quad \text{ and } \quad 
\Gamma= \left(\left\langle \Gamma e_i,e_j \right\rangle \right)_{1\leq i,j\leq N} = 
\left(
\begin{array}{c|ccc}
\Gamma_{11} &  \cdots & (\Gamma_{1i})_{2 \leq i\leq N} &  \cdots   \\
    \hline  
  \vdots &  & &   \\
 (\Gamma_{i1})_{2 \leq i\leq N}  &   &  \widetilde \Gamma  &   \\
  \vdots  &  &&
\end{array}
\right),  
\end{equation}
 where $ p_1\in \R $,  $\widetilde p\in V^\perp=\text{span}(e_2, \cdots, e_N)$ 
  and $\widetilde \Gamma=\left(\left\langle \Gamma e_i,e_j \right\rangle \right)_{2 \leq i,j\leq N}\in \mathcal{S}^{N-1}$.

\begin{figure} [t] 
\begin{center}
 \scalebox{0.8}{
\includegraphics{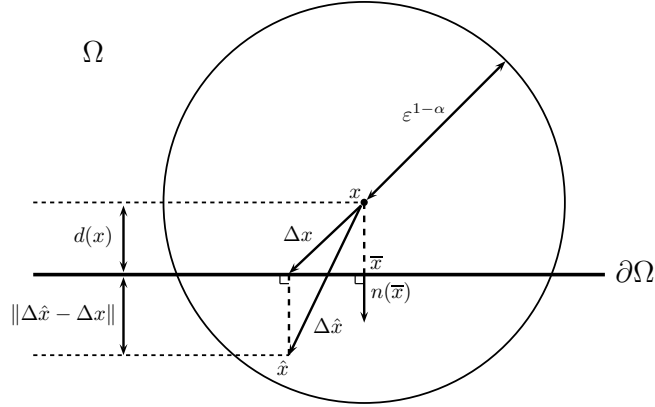}
 }
\end{center}
\caption{Formal derivation for $x$ near the boundary $\partial \Omega$, notation: $\bar x =\proj_{\partial \Omega}(x)$.} 
\label{drawing_formal_derivation}
\end{figure}

Let us focus on the Neumann penalization term in \eqref{pb_heur_gen} denoted by 
\begin{equation*}
 P(x)= \norm{\Delta \hat x - \Delta x} m(\Delta x) \quad \text{ with }  \quad
 m(\Delta x)= 
\begin{cases}
h(x+\Delta x) - Du(x)\cdot n(x+\Delta x),  & \text{ if } \hat x \notin \ol \Omega, \\
\tilde m(\Delta x) , & \text{ if } \hat x \in \ol \Omega, 
\end{cases}
\end{equation*}
where $m(\Delta x)$ is extended for $\hat x \in \ol \Omega$ by any function $\tilde m(\Delta x)$ (see Remark \ref{extension_h}). 
This contribution is favorable to Helen, $P(x)>0$, if $m(x)>0$, or to Mark, $P(x)<0$, if $m(x)<0$,  and its size depends on the magnitude of the vector 
$\Delta \hat x - \Delta x$. 
Our formal derivation is local and essentially geometric, in the sense that our target is to determine the optimal choices for Helen by considering all the 
moves $\Delta \hat x$ that Mark can choose.
By continuity of $h$ and smoothness of $u$, the function $m(\Delta x)$ 
is close to $m=h(\bar x) - Du(x)\cdot n(\bar x)$ if $\hat x \notin \ol \Omega$. 
We shall assume here that $m(\Delta x)$, which serves to model the Neumann boundary condition, is locally constant on the boundary and equal to $m$. 
This hypothesis corresponds in the game to assume that in a small neighborhood, crossing the boundary is always favorable to one player.
 In order to focus only on the geometric aspects, this approach seems formally appropriate 
since it freezes the dependence of $p(x)$ on $m(x)$ by eliminating the difficulties linked to the local variations of $m(x)$  like the change of sign.

Hence, it is sufficient to examine 
\begin{equation}\label{heuristic_opt_n_gene}
 \max_{p,\Gamma   }   \min_{\Delta \hat x} 
\left[  (Du - p)\cdot \Delta \hat x + \norm{\Delta \hat x - \Delta x} m  + \frac{1}{2} \left\langle D^2u  \Delta x, \Delta x \right\rangle 
 - \frac{1}{2} \left\langle \Gamma \Delta \hat x, \Delta \hat  x \right\rangle 
  - \eps^2 f \left(t, x, u ,  p,\Gamma \right) \right]. 
\end{equation}

The formal proof will be performed in three steps.

\textbf{Step 1:} To determine the optimal choice for Helen of $p$, 
we consider the $\eps$-order optimization problem $\mathcal{M}$ obtained from \eqref{heuristic_opt_n_gene} by neglecting the second $\eps$-order terms 
\begin{equation}\label{heuristic_opt1}
\mathcal{M}= \max_{p} \min_{\Delta \hat x}   \left[ (Du - p) \cdot \Delta \hat x +\norm{\Delta \hat x - \Delta x} m   \right].
\end{equation}
By writing $\Delta \hat x= (\Delta \hat x)_1 n(\ol x) + \widetilde{\Delta \hat x}$ with $\widetilde{\Delta \hat x}\in V^\perp$
 and observing that $\norm{\Delta \hat x - \Delta x} $ depends only on $(\Delta \hat x)_1$, 
we decompose the max min \eqref{heuristic_opt1} into 
  \begin{align*}
\mathcal{M} & =   \max_{p_1,\widetilde p  }   \min_{\Delta \hat x} \left[  (\widetilde {Du}  -  \widetilde p)\cdot \widetilde {\Delta \hat x} +
(Du_1-p_1)(\Delta \hat x)_1+\norm{\Delta \hat x - \Delta x}  \right]   \\
& =  \max_{p_1 }   \min_{|( \Delta \hat x)_1 |  \leq  \eps^{1-\alpha}} \left[  (Du_1-p_1)(\Delta \hat x)_1+  \norm{\Delta \hat x - \Delta x}  m 
+ \max_{\widetilde p } \min_{\norm{\widetilde{\Delta \hat x}} \leq \sqrt{ \eps^{2-2\alpha} - |(\Delta \hat x)_1|^2} } 
(\widetilde {Du}  -  \widetilde p)\cdot \widetilde {\Delta \hat x}  \right].
  \end{align*}
Noticing that the choices of $\widetilde p$ and $p_1$  are independent from each other, we can successively 
solve the optimization problems. First of all, in order to choose $\widetilde p$,  let us determine 
\begin{equation*}
\widetilde {\mathcal{M}}= \max_{\widetilde p } \min_{\norm{\widetilde{\Delta \hat x}} \leq \sqrt{ \eps^{2-2\alpha} - |(\Delta \hat x)_1|^2} } 
(\widetilde {Du}  -  \widetilde p)\cdot \widetilde {\Delta \hat x}.
\end{equation*}
If $\Delta \hat x = \pm \eps^{1-\alpha} n(\ol x)$, $\widetilde {\Delta \hat x}=0$ and the min is always zero: Helen's choice is irrelevant.
Otherwise, Helen should take $ \widetilde  p= \proj_{V^\perp} Du = \widetilde {Du} $, 
since otherwise Mark can make this max min strictly negative and minimal by choosing 
$\widetilde {\Delta \hat x} = -  
\sqrt{ \eps^{2-2\alpha} - |(\Delta \hat x)_1|^2} \frac{(Du - p)_{V^{\perp}}}{\norm{Du- p}}$ with $\Delta \hat x \neq \pm \eps^{1-\alpha} n(\ol x)$.
Thus Helen chooses $\widetilde p=\widetilde {Du}$, $\widetilde {\mathcal{M}}=0$ and  $\mathcal{M}$ reduces to
\begin{equation}\label{heuristic_opt1n}
\mathcal{M}= \max_{p_1} \min_{\Delta \hat x}   \left[ ((Du)_1 - p_1)  (\Delta \hat x)_1 +\norm{\Delta \hat x - \Delta x} m   \right].
\end{equation}
To determine the remaining coordinate $p_1=p\cdot n(\ol x)$ of $p$, we now consider 
the optimization problem \eqref{heuristic_opt1n} by restricting the possible choices made by Mark to the moves $\Delta \hat x$ which belong 
to the subspace $V=\R n(\bar x)$. 
Since $\norm{\Delta \hat x}\leq \eps^{1-\alpha}$ and $\Delta \hat x \in V$, 
we use the parametrization  $\Delta \hat x=\lambda \eps^{1-\alpha} n(\bar x)$, $\lambda \in [-1,1]$. 
If $\hat x \in \Omega$, the boundary is not crossed and $\norm{\Delta x-\Delta \hat x}=0$, 
while if $\hat x \notin \Omega$ the boundary is crossed and $\norm{\Delta x - \Delta \hat x}=\lambda \eps^{1-\alpha} - d(x)$. 
The intermediate point $\hat x = \bar x \in \partial \Omega$ separating the two regions corresponds to
 $\lambda_0= \frac{d(x)}{\eps^{1-\alpha}}$ and $\norm{\Delta x - \Delta \hat x}=0$. 
As a result, to compute the min in \eqref{heuristic_opt1n}, we  shall distinguish these two regions 
by decomposing the global minimization problem into two minimization problems respectively on each region
\begin{equation}\label{decompose_M}
\mathcal{M}= \max_{s_p} \kappa(s_p) \quad \text{ with } \quad  \kappa(s_p) =\min(\mathcal{M}_1(s_p), \mathcal{M}_2(s_p)), 
\end{equation}
where $s_p= (Du - p) \cdot n(\bar x)$ and 
\begin{align}
 & \mathcal{M}_1(s_p)= \min_{\lambda_0 \leq \lambda_1 \leq 1 }M_1 (\lambda_1)  \quad \text{ with } \quad M_1(\lambda_1)= 
(s_p+ m)\eps^{1-\alpha}  \lambda_1 - d(x) m, \label{heur_optimization_M1} \\
 & \mathcal{M}_2(s_p) =\min_{-1 \leq \lambda_2 \leq \lambda_0}M_2(\lambda_2) \quad \text{ with } \quad M_2(\lambda_2)=   s_p \eps^{1-\alpha} \lambda_2.
\label{heur_optimization_M2}
\end{align}
For fixed $p$, the functions defining $M_1$ and $M_2$ are affine and can easily be minimized separately:
\begin{itemize}
 \item If $s_p+m\geq 0$, $\mathcal{M}_1(s_p)$ is attained for $\lambda_1=\lambda_0$ and $\mathcal{M}_1(s_p)=d(x)s_p$.
 \item If $s_p+m <0$, $\mathcal{M}_1(s_p)$ is attained for $\lambda_1=1$ and $\mathcal{M}_1(s_p)=\eps^{1-\alpha} s_p +(\eps^{1-\alpha} - d(x))m$.  
 \item If $s_p \geq 0$, 
$ \mathcal{M}_2(s_p)$ is attained for $\lambda_2=- 1$ and  $ \mathcal{M}_2(s_p) = - \eps^{1-\alpha}s_p $.
 \item If $s_p< 0$ , $ \mathcal{M}_2(s_p)$ is attained for $\lambda_2= \lambda_0$ and  $ \mathcal{M}_2(s_p)=d(x)s_p$. 
\end{itemize}
Geometrically, $\lambda \in\{-1,1,\lambda_0\}$  corresponds to three particular moves: 
$\Delta \hat x=\pm \eps^{1-\alpha} n(\bar x)$ and $\Delta \hat x= d(x) n(\bar x)$. 
We are going to distinguish several cases to compute the max min according to the sign of $s_p$ and $m$. First of all, let us assume that 
$m$ is positive.  
\begin{enumerate}[label=\textup{ }{(C\arabic*)}\textup{ },ref={(C\arabic*)}] 
 \item \label{c1_heur} If $s_p \geq 0$ then $s_p + m \geq 0$  and the optimal choices are $(\lambda_1,\lambda_2)=(\lambda_0,-1) $. 
It remains to minimize between \eqref{heur_optimization_M1} and \eqref{heur_optimization_M2}. 
Taking into account that $d(x) \leq \eps^{1-\alpha}$ and $s_p \geq 0$, we get by the definition of $\kappa(s_p)$ given by \eqref{decompose_M} that
$\kappa(s_p) = \min \{ d(x) s_p, - \eps^{1-\alpha} s_p \} = -\eps^{1-\alpha} s_p$.
 \item \label{c2_heur} If $-m\leq s_p < 0$ then  $(\lambda_1,\lambda_2)=(\lambda_0,\lambda_0)$ and $\kappa(s_p) =\mathcal{M}_1(s_p)=\mathcal{M}_2(s_p)=d(x) s_p$.   
 \item \label{c3_heur} If $s_p<-m<0$ then $(\lambda_1,\lambda_2)=(1,\lambda_0)$ and 
$\mathcal{M}_1(s_p)= \eps^{1-\alpha} s_p +(\eps^{1-\alpha} - d(x)) m$  and $\mathcal{M}_2(s_p)=d(x) s_p$. 
By multiplying the inequality $s_p<-m<0$ by $(\eps^{1-\alpha} - d(x))$, we get 
\begin{equation*}
\kappa(s_p) = \min \{ \eps^{1-\alpha} s_p +(\eps^{1-\alpha} - d(x)) m  , d(x) s_p  \} = d(x) s_p.
\end{equation*}
\end{enumerate}
By combining cases \ref{c1_heur}--\ref{c3_heur}, we conclude that if $m>0$,
\begin{equation*}
\kappa(s_p)=
\begin{cases}
 \eps^{1-\alpha} s_p +(\eps^{1-\alpha}-d(x)) m, & \text{if }s_p \leq -m , \\ 
 d(x) s_p,                                      & \text{if } -m \leq s_p \leq 0,  \\
- \eps^{1-\alpha} s_p,                          & \text{if }s_p \geq 0  .
\end{cases}
\end{equation*}
The max of $\kappa$ is  zero and reached at the unique value $s_p=Du\cdot n (\ol x) -p_1=0$. 
Since $\tilde p=\widetilde {Du}$ by the previous analysis, we conclude in \eqref{decomp_p_Gamma} that if $m> 0$, Helen's optimal choice is $p=Du$.

Let us now suppose that $m$ is negative. 
\begin{enumerate}[resume*]
 \item \label{c4_heur} If $s_p < 0$ then $s_p+m < 0$ and the optimal choices are $(\lambda_1,\lambda_2)=(1,\lambda_0)$. 
By the definition of $\kappa(s_p)$ given by \eqref{decompose_M}, we obtain
 \begin{equation}  \label{min_1_m_neg}
\kappa(s_p)=\min \{ \eps^{1-\alpha} s_p+(\eps^{1-\alpha}-d(x))m, d(x)s_p \} = \eps^{1-\alpha} s_p  +(\eps^{1-\alpha}-d(x))m.   
\end{equation}
 \item \label{c5_heur} If $s_p \geq -m > 0$ then $(\lambda_1, \lambda_2)=(\lambda_0,-1)$  and $\mathcal{M}_1(s_p)=d(x) s_p$
  and $\mathcal{M}_2(s_p)=-\eps^{1-\alpha} s_p$. By the definition of $\kappa(s_p)$ given by \eqref{decompose_M}, we obtain
\mbox{$\kappa(s_p) = \min \left\{  d(x) s_p ,  - \eps^{1-\alpha} s_p \right\}  = - \eps^{1-\alpha} s_p$}.
 \item \label{heur_case_very_neg} If $0< s_p   < -m$, then $(\lambda_1, \lambda_2)=(1,-1)$ and $\mathcal{M}_1(s_p)=\eps^{1-\alpha}s_p +(\eps^{1-\alpha}-d(x)) m$
and $\mathcal{M}_2(s_p)=-\eps^{1-\alpha} s_p$. By the definition of $\kappa(s_p)$ given by \eqref{decompose_M}, we obtain
\begin{equation*} 
\kappa(s_p)= \min \left\{ \eps^{1-\alpha}s_p +(\eps^{1-\alpha}-d(x)) m , -\eps^{1-\alpha}s_p \right\} .
\end{equation*}
The target for Helen is to maximize this minimum with respect to $s_p$.
Both functions intervening in the minimum are affine: the first one is affine, strictly increasing and is equal to $(\eps^{1-\alpha}-d(x)) m<0$ 
for $s_p=0$  and to $d(x)m>0$   for  $s_p=-m$  whereas 
the second function is linear and strictly decreasing and is equal to $m\eps^{1-\alpha}<0$ for  $s_p=-m$. 
As a result, there is a unique $s^\ast$ such that these two functions are equal and this value precisely
 realizes the max of $\kappa$ on $[0, -m]$. Thus, the best that Helen can hope corresponds to  
$ \eps^{1-\alpha} s^\ast +(\eps^{1-\alpha}-d(x)) m =  - \eps^{1-\alpha}   s^\ast $. This gives
\begin{equation*}
s^\ast= (Du - p) \cdot n(\bar x )= - \frac{1}{2}\Big(1 -  \frac{d(x)}{\eps^{1-\alpha}} \Big) m.
\end{equation*}
We immediately check that $s^\ast \in \left[0,-\frac{m}{2}\right]$, which implies the condition $s^\ast+m \leq \frac{1}{2}m< 0$. 
Thus,  $\ds \max_{s_p \in[0,-m]} \kappa (s_p)= \frac{1}{2} (\eps^{1-\alpha}-d(x)) m$ is greater than the minimum obtained in \eqref{min_1_m_neg}.
\end{enumerate}
By combining cases \ref{c4_heur}--\ref{heur_case_very_neg}, we conclude that if $m \leq 0$,
\begin{equation*}
\kappa(s_p)=
\begin{cases}
 \eps^{1-\alpha} s_p +(\eps^{1-\alpha} - d(x)) m,  & \text{if }s_p < s^\ast , \\ 
- \eps^{1-\alpha} s_p,                             & \text{if }s_p \geq s^\ast.
\end{cases}
\end{equation*}
The max of $\kappa$ is equal to $\kappa(s^\ast)$  and reached for $s_p=Du\cdot n (\ol x) -p_1=s^\ast$.

Let us give an intermediate conclusion: if $m>0$, Helen chooses $p=Du$ whereas if $m \leq  0$, she chooses   
\begin{equation} \label{opt_choice_p1}
 p  = Du   +  \frac{m}{2} \left( 1-\frac{d(x)}{\eps^{1-\alpha}} \right)  n(\bar x ).
\end{equation}

\textbf{Step 2:} We are now going to take into account the second order terms in $\eps$ in the optimization problem. 
If $m \geq 0$, once Helen has chosen $p=Du$, the optimization problem \eqref{heuristic_opt_n_gene} reduces to computing
\begin{equation} \label{heuristic_opt_n_m+1}
 \max_{\Gamma   }   \min_{\Delta \hat x} 
\left[ \norm{\Delta \hat x - \Delta x} m   + \frac{1}{2} \left\langle D^2u  \Delta x, \Delta x \right\rangle 
 - \frac{1}{2} \left\langle \Gamma \Delta \hat x, \Delta \hat  x \right\rangle   - \eps^2 f \left(t, x, u ,  Du,\Gamma \right) \right]  .
\end{equation}
Mark is going to choose $\Delta \hat x\cdot n(\ol x) \leq 0$, because otherwise
the first $\eps$-order quantity $\norm{\Delta \hat x - \Delta x} m  $ will be favorable to Helen. 
 Then considering $\Delta \hat x\cdot n(\ol x) \leq 0$, we have $\Delta \hat x=\Delta x$ and 
by symmetry of the quadratic form associated to $ (D^2u - \Gamma )$, the optimization problem \eqref{heuristic_opt_n_gene} reduces to  
\begin{multline}\label{heuristic_opt_n_m+2}
\max_{\Gamma   }   \min_{\Delta \hat x\cdot n(\ol x) \leq 0} 
\left[ \frac{1}{2} \left\langle (D^2u - \Gamma )\Delta \hat x, \Delta \hat x \right\rangle  - \eps^2 f \left(t, x, u , Du,\Gamma \right) \right]  
\\
=  \eps^2  \max_{\Gamma   }   \min_{\Delta \hat x} 
\left[ \frac{1}{2}  \eps^{-2}  \left\langle (D^2u - \Gamma )\Delta \hat x, \Delta \hat x \right\rangle  -f \left(t, x, u,  Du,\Gamma \right) \right].
\end{multline}
Helen should choose $\Gamma\leq D^2u$, since otherwise Mark can drive $\eps^{-2}\langle (D^2u -\Gamma) \Delta \hat x, \Delta \hat x\rangle$
to $-\infty$ by a suitable choice of $\Delta \hat x$. Thus, the min attainable by Mark is zero and is at least realized for the choice $\Delta \hat x=0$. 
Helen's maximization reduces to 
\begin{equation*}
\max_{\Gamma\leq D^2u}[u_t-f(t,x,u,Du,\Gamma)].
\end{equation*} 
Since the PDE is parabolic, i.e. since $f$ satisfies \eqref{ellipticity_f}, Helen's optimal choice is $\Gamma=D^2u$ and \eqref{pb_heur_gen} reduces formally to 
$u_t- f(t,x,u,Du, D^2u)=0$.
 
If $m<0$, Helen must now choose $\Gamma$. In fact, 
we are going to see that the choice of $p_1=p\cdot n(\ol x)$ obtained at \eqref{opt_choice_p1} can be slightly improved
by taking into account the additional terms containing $D^2u$ and $\Gamma$. 
Suppose Helen chooses $p$ such that $(p-Du)_{|V^\perp}=0$ (notice that our first order  computation \eqref{opt_choice_p1} fulfills this condition)
and Mark chooses a move $\Delta \hat x^\ast$ realizing the minimum on $\Delta \hat x$ in \eqref{heuristic_opt_n_gene}. 
We consider two cases depending on~$\Delta \hat x^\ast$. 

\textbf{Case a:} if $\Delta \hat x^\ast \in V^\perp$, we can restrain the minimization problem to the moves $\Delta \hat x$ 
which belong to $V^\perp$, $\Delta \hat x=\Delta x $. 
Thus, the optimization problem \eqref{heuristic_opt_n_gene} reduces to computing
\begin{equation*} 
\mathcal{M}_{V^\perp}= \eps^2   \max_{\Gamma} \min_{\substack{\Delta \hat x \in V^\perp  \\ \norm{\Delta \hat x} \leq \eps^{1-\alpha}}} \left[
\frac{1}{2} \eps^{-2} \langle (D^2u - \widetilde \Gamma) \Delta  x, \Delta  x \rangle - f \left(t, x, u,  p,\Gamma \right) \right], 
\end{equation*}
where $\widetilde \Gamma= \Gamma_{|V^\perp}$. Helen should choose $\widetilde \Gamma \leq \widetilde {D^2u}$, since otherwise 
Mark can drive $\eps^{-2}\langle (D^2u - \widetilde \Gamma) \Delta  x, \Delta  x \rangle$ to $-\infty$ 
by a suitable choice of $\Delta \hat x$. By repeating the same argument of ellipticity of $f$ already used for $m> 0$, 
Helen's optimal choice is $\widetilde \Gamma = \widetilde {D^2u}$.

\textbf{Case b:} if $\Delta \hat x^\ast \notin V^\perp$, there exists an unit vector $v$ orthogonal to $n(\bar x)$ such that
 $\Delta \hat x^\ast \in \text{span} (n (\bar x), v)$. Thus, we restrain 
the minimization problem on $\Delta \hat x$ given by \eqref{heuristic_opt_n_gene} to the moves $\Delta \hat x$ which belong to the disk
$D=\text{span} (n (\bar x), v) \cap B(\eps^{1-\alpha})$. 
This gives the optimization problem $\mathcal{M}_D$ given by 
\begin{equation}\label{heuristic_opt_n_m-4} 
\mathcal{M}_D=  \max_{p_1, \Gamma   }   \min_{\Delta \hat x \in D}  \Big[(Du-p)_1 (\Delta \hat x)_1 +\norm{\Delta \hat x - \Delta x}m   
+ \frac{1}{2} \left\langle D^2u \Delta  x, \Delta  x \right\rangle
- \frac{1}{2} \left\langle  \Gamma \Delta \hat x, \Delta \hat x \right\rangle  - \eps^2 f \left(t, x, u,  p ,\Gamma \right) \Big]
\end{equation}
by taking into account that $\widetilde p = \widetilde {Du}$ and $\widetilde \Gamma = \widetilde {D^2u}$.  
Neglecting $-\eps^2 f(t,x,z,p, \Gamma)$, we want to compute
$\ds  \max_{s_p,\Gamma   }   \min_{\Delta \hat x} \mathcal{N}(s_p, \Gamma, \Delta \hat x)$ with 
\begin{equation}\label{heuristic_opt_n_2}
\mathcal{N}(s_p, \Gamma, \Delta \hat x) =  s_p ( \Delta \hat x)_1
+ \norm{\Delta \hat x - \Delta x} m  + \frac{1}{2} \left\langle D^2u  \Delta x, \Delta x \right\rangle 
 - \frac{1}{2} \left\langle \Gamma \Delta \hat x, \Delta \hat  x \right\rangle .
\end{equation}
Since  $\norm{\Delta  \hat x }\leq \eps^{1-\alpha}$, we parametrize the disk $D$ by 
$\Delta \hat x= \lambda \eps^{1-\alpha} n(\bar{x})+ \mu \eps^{1-\alpha} v$ with $\lambda^2+\mu^2 \leq 1 $. 
Notice that the calculation of $\Gamma_{1v}:= \langle \Gamma n(\bar x), v\rangle$  for 
 $v$ orthogonal to $n(\bar x)$ implies the computation  of $(\Gamma n(\bar x))_{|V^\perp}$.
If Mark chooses $\Delta \hat x$ such that $\lambda\geq \lambda_0$ for which the boundary is crossed, 
\begin{multline}\label{formal_ord2_lambda_neg_mu}
\mathcal{N}(s_p, \Gamma, \Delta \hat x) = (s_p +m)\eps^{1-\alpha} \lambda  - d(x) m  
 +\frac{1}{2} d^2(x)  (D^2u)_{11}  - \frac{1}{2} \lambda^2 \eps^{2-2\alpha} \Gamma_{11} \\
-  \mu  \left(\frac{d(x)}{\eps^{1-\alpha}} (D^2u)_{1v}  - \lambda \Gamma_{1v} \right) \eps^{2-2\alpha}, 
\end{multline}
whereas for $\Delta \hat x$ such that $\lambda\leq \lambda_0$ for which the boundary is not crossed, 
 \begin{equation}\label{formal_ord2_lambda_pos_mu}
  \mathcal{N}(s_p, \Gamma, \Delta \hat x) =   s_p \eps^{1-\alpha} \lambda   + \frac{1}{2} \lambda^2 ((D^2u)_{11}
    - \Gamma_{11})  \eps^{2-2\alpha}  -  \lambda \mu  ((D^2u)_{1v}  - \Gamma_{1v} ) \eps^{2-2\alpha}. 
  \end{equation}
For fixed $\lambda $, Mark will always choose $\mu$ so that the last term is negative and maximal which leads to 
\begin{equation*}
\mu= 
\begin{cases}
 \text{sgn}(\frac{d(x)}{\eps^{1-\alpha}} (D^2u)_{1v}- \lambda \Gamma_{1v}) \sqrt{1-\lambda^2}, & \text{if } \lambda_0 \leq \lambda \leq 1, \\
 \text{sgn}(\lambda ((D^2u)_{1v}  - \Gamma_{1v} )) \sqrt{1-\lambda^2},                          & \text{if } - 1 \leq \lambda < \lambda_0.
\end{cases}
\end{equation*}
The min of $\mathcal{N}(s_p, \Gamma, \Delta \hat x)$ 
on $\mu$ depends only on $\lambda = \Delta \hat x \cdot n(\bar x)$ and will be denoted below by $\mathcal{N}(s_p, \Gamma, \lambda)$. 
By virtue of \eqref{formal_ord2_lambda_neg_mu} and \eqref{formal_ord2_lambda_pos_mu} it corresponds, for $\lambda \geq \lambda_0$, to 
\begin{equation*}\label{formal_ord2_lambda_neg}
\mathcal{N}(s_p, \Gamma, \lambda) =  (s_p +m)\eps^{1-\alpha} \lambda  
- d(x) m  +\frac{1}{2} d^2(x)  (D^2u)_{11}  
 - \frac{1}{2} \lambda^2 \eps^{2-2\alpha} \Gamma_{11}  
 - \sqrt{1-\lambda^2}  \Big| \frac{d(x)}{\eps^{1-\alpha}} (D^2u)_{1v}  - \lambda \Gamma_{1v} \Big| \eps^{2-2\alpha} \notag
\end{equation*}
and, for  $\lambda \leq \lambda_0$, to
\begin{equation*}\label{formal_ord2_lambda_pos}
\mathcal{N}(s_p, \Gamma, \lambda) =
 s_p \eps^{1-\alpha} \lambda  + \frac{1}{2} \lambda^2 ((D^2u)_{11}  - \Gamma_{11})  \eps^{2-2\alpha} 
-  |\lambda| \sqrt{1-\lambda^2} \Big|(D^2u)_{1v}  - \Gamma_{1v} \Big| \eps^{2-2\alpha} .
\end{equation*}
The second order terms containing $D^2u$ and $\Gamma$ being a little perturbation for $\eps>0$ small enough compared to $(\eps^{1-\alpha} - d(x))m$ 
for $d(x)\ll \eps^{1-\alpha}$, it is sufficient to consider the case \ref{heur_case_very_neg}
which led to \eqref{opt_choice_p1} and $(\lambda_1, \lambda_2)=(1,-1)$ corresponding 
to $\Delta \hat x= \pm \eps^{1-\alpha} n(\ol x)$.
Therefore, we are going to compare the moves close to the optimal choices $\Delta \hat x= \pm \eps^{1-\alpha} n(\ol x)$ previously obtained by 
considering only the first terms in the Taylor expansion. 
More precisely, we may assume $\lambda \approx \pm 1$, which leads to making the change of variables
$\lambda_1= 1-\rho_1$, $\lambda_2= -1+ \rho_2$  and take $\ds \rho_i \sublim_{\eps\rightarrow 0} 0$  for  $i=1,2$.
After some computations, we get a Taylor expansion in $\rho_i$, $i=1,2$, in the form
\begin{multline}  \label{formal_dev_rho1}
\mathcal{N}  (s_p, \Gamma, 1- \rho_1) =
(s_p +m)\eps^{1-\alpha}   - d(x) m +\frac{1}{2} d^2(x)  (D^2u)_{11}   - \frac{1}{2} \eps^{2-2\alpha} \Gamma_{11} \\  
- \sqrt{2\rho_1} \Big| \frac{d(x)}{\eps^{1-\alpha}} (D^2u)_{1v}  -  \Gamma_{1v} \Big| \eps^{2-2\alpha}  
  +\rho_1 \eps^{1-\alpha} \Big[- (s_p +m)  +  \eps^{1-\alpha} \Gamma_{11}   \Big] +O(\eps^{2-2\alpha} \rho_1^{3/2}),
\end{multline}
and
\begin{multline}\label{formal_dev_rho2}
\mathcal{N}  (s_p, \Gamma, -1+ \rho_2) = -s_p \eps^{1-\alpha} - \frac{1}{2} ((D^2u)_{11}- \Gamma_{11})  \eps^{2-2\alpha} \\
- \sqrt{2\rho_2} \Big|(D^2u)_{1v}  - \Gamma_{1v} \Big| \eps^{2-2\alpha}
 + \rho_2 \eps^{1-\alpha} (  s_p   - ((D^2u)_{11}  - \Gamma_{11})  \eps^{1-\alpha} )  +O(\eps^{2-2\alpha} \rho_2^{3/2}).  
\end{multline}
%
First of all, we are now going to focus on the 0-order terms on the $\rho_1,\rho_2$ variables.
 Dropping the next terms corresponds to the two particular moves $\Delta \hat x= \pm \eps^{1-\alpha} n(\bar x)$ ($\rho_1=\rho_2=0 $).   
For fixed $\Gamma$, since these terms containing $\Gamma$ have the same contributions, we can omit the dependence of
$\mathcal{N}(\cdot, \Gamma,1)$ and $\mathcal{N}(\cdot, \Gamma, - 1)$ on $\Gamma$.
Then, by repeating the same arguments already used,  there exists  a unique $s_2^\ast$ 
 realizing the max of $\ds \min_{}(\mathcal{N}(\cdot, \Gamma,1),\mathcal{N}(\cdot, \Gamma,-1) )$ for which both functions are equal.  
After some calculations, we find
\begin{equation}\label{formal_choice_s}
s_2^\ast= -\frac{1}{2}\left(1-\frac{ d(x)}{\eps^{1-\alpha}}\right) m 
 + \frac{1}{4} \left( \eps^{1-\alpha}- \frac{d^2(x)}{\eps^{1-\alpha}} \right) (D^2u)_{11}.
\end{equation}
If $m< 0$, Helen will finally choose 
\begin{equation}\label{formal_choice_popt}
p_\text{opt}(x) = Du  + \left[ \frac{1}{2}\left(1-\frac{ d(x)}{\eps^{1-\alpha}}\right) m 
 - \frac{1}{4} \left( \eps^{1-\alpha}- \frac{d^2(x)}{\eps^{1-\alpha}} \right) (D^2u)_{11}
\right] n(\bar x).
\end{equation}
To complete the analysis on the 0-order terms on the $\rho_1,\rho_2$ variables, we are now going to see how Helen must choose 
$\Gamma_{11}$. By conserving only the 0-order terms, we obtain 
\begin{equation*}
\mathcal{M}_D  \approx \frac{1}{2}(\eps^{1-\alpha} - d(x) ) m 
+  \max_{\Gamma_{11}}  
 \left[    \frac{1}{4} (\eps^{2-2\alpha}+ d^2(x)) (D^2u)_{11} -\frac{1}{2} \eps^{2-2\alpha}  \Gamma_{11}
 - \eps^2 f \left(t, x,  u,  p_{\text{opt}},\Gamma \right) \right]. 
\end{equation*}
Since $\Gamma_{11}$ cannot counterbalance the first order term, the $\Gamma_{11}$-term and the second order terms are gathered.
Helen wants to make the best choice, so she is going to choose $ \Gamma_{11} $ such that
\begin{equation*}
\frac{1}{4} (\eps^{2-2\alpha}+ d^2(x)) (D^2u)_{11} -\frac{1}{2} \eps^{2-2\alpha}  \Gamma_{11}  \geq 0 .
\end{equation*}
By ellipticity of $f$, Helen will choose $\Gamma_{11} $  such that this upper bound on $\Gamma_{11}$ is attained. She takes    
\begin{equation}\label{Gamma_opt_nn}
\Gamma_{11} = \frac{1}{2} \left(1+ \frac{d^2(x)}{\eps^{2-2\alpha}}\right) (D^2u)_{11}.
\end{equation}
It remains to determine $\Gamma_{1v}$. By plugging the optimal choices $s_2^\ast$, corresponding to $p_\text{opt}$,  and $\Gamma_{11}$, 
respectively given by \eqref{formal_choice_s} and \eqref{Gamma_opt_nn} in \eqref{formal_dev_rho1}--\eqref{formal_dev_rho2}, we have
\begin{align*}  
\mathcal{N}  (s_2^\ast, \Gamma, 1- \rho_1) 
& 
= \frac{1}{2} (\eps^{1-\alpha} - d(x) )m  - \sqrt{2\rho_1} \Big| \frac{d(x)}{\eps^{1-\alpha}} (D^2u)_{1v} - \Gamma_{1v} \Big| \eps^{2-2\alpha}   
   -(s_2^\ast +m) \rho_1 \eps^{1-\alpha}  +O(\eps^{2-2\alpha} \rho_1), \\
\mathcal{N} (s_2^\ast, \Gamma, -1 + \rho_2) 
&
=   \frac{1}{2} (\eps^{1-\alpha} - d(x) )m - \sqrt{2\rho_2} \left|(D^2u)_{1v} -\Gamma_{1v} \right| \eps^{2-2\alpha} 
 + s_2^\ast \rho_2 \eps^{1-\alpha}  +O(\eps^{2-2\alpha} \rho_2).
\end{align*}
Dropping the $O(\eps^{2-2\alpha} \rho_i)$ terms and noticing that $s_2^\ast>0$ and $ - (s_2^\ast +m)>0 $ for $\eps$ small enough, 
the two minimization problems $\ds \min_{\rho_i} \mathcal{N}  (s_2^\ast, \Gamma, 1- \rho_i)$, $i\in \{1,2\}$ for Mark reduce to find
\begin{align*}
\min_{0< \rho \leq 1}  f(\rho), \quad \text{ where }  f(\rho) = a \sqrt{\rho}+b\rho,    
\end{align*}
with $a<0<b$. Differentiating $f$, the minimum of $f$ is attained at $\ds \sqrt{\rho^\ast} =  - \frac{a}{2b}$. 
We can notice that this computation is equivalent to formally differentiating the Taylor expansion of 
\eqref{formal_ord2_lambda_neg_mu}--\eqref{formal_ord2_lambda_pos_mu}. 
Conserving the predominant terms and dropping the next terms, the minimum of $\mathcal{N}  (s_2^\ast, \Gamma, 1- \rho_1)$ 
and $\mathcal{N}  (s_2^\ast, \Gamma, -1+ \rho_2)$ are respectively attained at  
\begin{equation*}
\sqrt{\rho_{1}^\ast } \simeq \frac{1}{\sqrt{2}}  \frac{ |\frac{d}{\eps^{1-\alpha}}(D^2u)_{1v}-\Gamma_{1v} | }{|s_2^\ast+m|} \eps^{1-\alpha}
\quad \text{ and } \quad  
\sqrt{\rho_2^\ast} \simeq \frac{1}{\sqrt{2}}\frac{|(D^2u)_{1v}-\Gamma_{1v}|}{|s_2^\ast|} \eps^{1-\alpha}  . 
\end{equation*}
Assuming formally that these approximations are in fact equalities, we obtain 
\begin{align}
\mathcal{N} (s_2^\ast, \Gamma, 1- \rho_{1}^\ast) & = \frac{1}{2} (\eps^{1-\alpha} - d(x) )m  
  -\frac{1}{2}  \left|\frac{d(x)}{\eps^{1-\alpha}}(D^2u)_{1v}  - \Gamma_{1v} \right|^2   \frac{\eps^{3-3\alpha}}{|s_2^\ast+m|}  +O(\eps^{4-4\alpha}),
\label{gamma1v1}  \\
\mathcal{N} (s_2^\ast, \Gamma, -1+\rho_2^\ast) & =  \frac{1}{2} (\eps^{1-\alpha} - d(x) )m  
         -\frac{1}{2}  \left|(D^2u)_{1v}  - \Gamma_{1v} \right|^2 \frac{\eps^{3-3\alpha}}{|s_2^\ast|} +O(\eps^{4-4\alpha}). 
         \label{gamma1v2}
\end{align}
Helen now has to choose $\Gamma_{1v}$ such that 
$\min \{ \mathcal{N} (s_2^\ast, \Gamma, 1 - \rho_{1}^\ast),\mathcal{N}(s_2^\ast,\Gamma, -1+\rho_{2}^\ast) \}$ is maximal.
We could compute the optimal value of $\Gamma_{1v}$ on the $\eps^{3-3\alpha}$-terms. However, it is not very useful. 
Since $m$ is a constant and  $\eps^{3-3\alpha} \ll \eps^2$ by \eqref{condition_pas}, the $\eps^{3-3\alpha}$-terms are negligible 
compared to $- \eps^2 f(t,x,u, p_\text{opt}, \Gamma)$ that we have omitted until now. 
For instance Helen can fix $\Gamma_{1v}$ such that one of the two terms depending on $\Gamma_{1v}$ in \eqref{gamma1v1} and \eqref{gamma1v2} is equal to zero: 
$\Gamma_{1v} = (D^2u)_{1v}$ or $ \Gamma_{1v} = \frac{d(x)}{\eps^{1-\alpha}}(D^2u)_{1v}$. 
The two choices are equivalent because Mark can reverse his move $\Delta \hat x$. For sake of simplicity, we assume Helen chooses  $\Gamma_{1v} = (D^2u)_{1v}$.
It is worth noticing that this expansion holds if $m$ is far from zero and we shall modify our arguments very carefully in Section \ref{consistance}
 when $m$ is negative but small with respect to a certain power of $\eps$.

Thus, if $m<0$, Helen will choose
\begin{equation} \label{formal_gamma_opt_n}
\Gamma_\text{opt}(x) = D^2 u + \left[ \frac{1}{2}   \left(-1+\frac{d^2 (x) }{\eps^{2-2\alpha}} \right)  (D^2 u)_{11} \right] E_{11}.
\end{equation}
 Unlike the usual game \cite{kohns}, when  Helen chooses $p$ and $\Gamma$ optimally, 
she does not become indifferent to Mark's choice of $\Delta \hat x$. More precisely, 
it depends  on the projection of $\Delta \hat x $ with respect to $n(\bar x)$.  Our games always have this feature.

\textbf{Step 3:} Now let us go back to the original optimization problem \eqref{pb_heur_gen}. If $m=0$, by letting $\eps \rightarrow 0$, we get 
$h(x) - Du(x)\cdot n(x)=0$.  Otherwise, \eqref{pb_heur_gen} formally reduces  to
\begin{equation} \label{eq_formal_final}
0\approx \eps^2 u_t+
\begin{cases}
\ds  \frac{1}{2}(\eps^{1-\alpha} - d(x) ) m  -  \eps^2 f (t, x,  u,  p_{\text{opt}}(x),\Gamma_\text{opt}(x))+o(\eps^2), 
& \text{if } d(x)\leq \eps^{1-\alpha} \text{ and } m < 0, \\ 
\ds  - \eps^2 f (t, x,u,Du,D^2u),  & \text{if } d(x)\geq \eps^{1-\alpha} \text{ or } m > 0, 
\end{cases}
\end{equation}
with $p_\text{opt}$ and $\Gamma_\text{opt}$ respectively defined by \eqref{formal_choice_popt} and \eqref{formal_gamma_opt_n}. 
If $x\in \Omega$, for $\eps$ small enough, the second relation in \eqref{eq_formal_final} is always valid
 so that we deduce from the $\eps^2$-order terms in \eqref{eq_formal_final} that $u_t- f(t,x,u,Du,D^2u)=0$.
If $x\in \partial \Omega$, $d(x)=0$ and we distinguish the cases $m>0 $ and $m< 0$. 
If $m>0$, one more time the second relation in \eqref{eq_formal_final} is always valid
 so that  $u_t- f(t,x,u,Du,D^2u)=0$. 
Otherwise, if $m< 0$, the first relation in \eqref{eq_formal_final} is always satisfied. 
We observe that the $\eps$-order term is predominant since $\eps^{1-\alpha} \gg \eps^2$. By dividing by $\eps^{1-\alpha}$ and 
letting $\eps \rightarrow 0$, we obtain $ m = 0$ that leads to a contradiction since we assumed $m< 0$. 
Therefore,  we have formally shown that on the boundary $h(x) - Du(x)\cdot n(x)=0$ or $u_t- f(t,x,u,Du,D^2u)=0$.

\subsubsection{Main parabolic result}

\label{rigorous_result_par}

We shall show, under suitable hypotheses, that $\ol u$ and $\udl v$ are respectively viscosity sub and supersolutions. 
A natural question is to compare $\overline u$ and $\underline v$. 
This is a global question, which we can answer only if the PDE has a comparison principle. Such a principle asserts that 
if $u$ is a subsolution and $v$ is a supersolution then $u\leq v$. 
If the PDE has such a principle then it follows that $\overline u \leq \underline v$.  
The opposite inequality is immediate from the definitions, so it follows that $\overline u= \underline v$, and we get a viscosity solution of the PDE. 
It is in fact the unique viscosity solution, since the comparison principle implies uniqueness.

\begin{theorem}\label{theo_cv_par_neu}
Consider the final-value problem \eqref{eq_neumann_evol} where $f$ satisfies \eqref{ellipticity_f}--\eqref{cont_growth_p_Gamma}, 
 $g$  and $h$ are continuous, uniformly bounded, and $\Omega$ is a $C^2$-domain satisfying both the uniform interior and exterior ball conditions. 
Assume the parameters $\alpha$, $\beta$, $\gamma$  satisfy  \eqref{condition_pas}-\eqref{cd_coeff_classiq}. 
Then $\overline u$ and $\underline v$ are uniformly bounded on $\ol \Omega \times [t_\ast, T]$ 
for any $t_\ast<T$, and they are respectively a viscosity subsolution and a viscosity supersolution of \eqref{eq_neumann_evol}. 
If the PDE has a comparison principle (for uniformly bounded solutions),  then it follows that $u^\eps$ and $v^\eps$ 
converge locally uniformly to the unique viscosity solution of \eqref{eq_neumann_evol}. 
\end{theorem}
This theorem is an immediate consequence of Propositions \ref{cv_par} and \ref{stability_para_prop}. 

In this theorem, we require the domain $\Omega$ to be $C^2$.
This assumption is crucial for the proof of Proposition~\ref{cv_par} case~\ref{cas_cv_par2} 
corresponding to the convergence at the final-time in the viscosity sense (see Remark~\ref{hyp_mini_domain}). 
It can also be noticed that it is this part of  Proposition~\ref{cv_par} which allows to use a comparison principle for the  parabolic PDE. 
On the other hand, since the game already requires the uniform interior and exterior ball conditions, the domain $\Omega$ is in fact at least $C^{1,1}$. 
It remains an open question to overcome the analysis in this case.

As mentioned in \cite{kohns}, some sufficient conditions for the PDE to have a comparison result can be found in Section~4.3 of \cite{cheridito_soner_touzi_victoir}.
In our framework, we can emphasise on the comparison principle obtained by Sato \cite[Theorem 2.1]{sato_interface} 
for a fully nonlinear parabolic equation with a homogeneous condition. 
The reader is also referred to the introduction for other references about comparison and existence results.  
Note that most comparison results require $f(t,x,z,p, \Gamma)$ to be nondecreasing in $z$.

We close this section with the observation that if $U^\eps(x,z,t)$ is a strictly decreasing function of $z$ then $v^\eps(x,t)=u^\eps(x,t)$. 
A sufficient condition for this to hold is that $f$ be nondecreasing in $z$:

\begin{lemma}\label{lem_dec_U}
Suppose $f$ is non-decreasing in $z$ in the sense that 
\begin{equation*} 
f(t,x,z_1,p, \Gamma)\geq f(t,x,z_0,p, \Gamma) \quad \text{ whenever }z_1>z_0.
\end{equation*}
Then $U^\eps$ satisfies 
\begin{equation*} 
U^\eps(x,z_1,t_j)\leq U^\eps(x,z_0,t_j) - (z_1-z_0)  \quad \text{ whenever }z_1>z_0,
\end{equation*}
at each discrete time $t_j=T - j\eps^2$. In particular, $U^\eps$ is strictly decreasing in $z$ and $v^\eps= u^\eps$.
\end{lemma}
\begin{proof} The whole space case is provided in \cite[Lemma 2.4]{kohns}. 
 For our game, it suffices to add  $- \norm{\Delta \hat x - \Delta x} h(x+\Delta x)$ in 
the expressions of $\Delta z_0$  and $\Delta z_1$  defined in the proof of \cite[Lemma~2.4]{kohns}.  The rest of the proof remains unchanged.
\end{proof}

\subsection{Nonlinear elliptic equations}
\label{rules_el_game}

This section explains how our game can be used to solve stationary problems with Neumann boundary conditions. The framework is similar to the parabolic case, 
but one new issue arises: we must introduce discounting as in \cite{kohns}, to be sure Helen's value function is finite. Therefore we focus on 
\begin{equation}\label{eq_neumann_el}
\begin{cases}
   f(x, u,Du,D^2u)+ \lambda u =0,  &  \text{ in } \Omega, \\ 
  \left\langle D u, n \right\rangle=h,  &  \text{ on } \partial \Omega ,  
\end{cases}
\end{equation}
where $\Omega$ is a domain with $C^2$-boundary and satisfies both the uniform interior and exterior ball condition presented in the introduction.
The constant $\lambda$ (which plays the role of an interest rate) must be positive, and large enough so that \eqref{est_f_el_neu} holds. 
Notice that if $f$ is independent of $z$ then any $\lambda$ will do.

We now present the game. The main difference with Section \ref{def_jeu_parabolique} is the presence of discounting. 
The boundary condition $h$ is assumed to be a bounded continuous function on $\partial \Omega$. 
Besides the parameters $\alpha$, $\beta$, $\gamma$ introduced previously,   
in the stationary case we need two new parameters, $m$ and $M$, and a $C_b^2(\ol \Omega)$-function $\psi$ such that
\begin{equation}\label{def_psi_h_el}
 \frac{\partial \psi}{\partial n}=\norm{h}_\infty +1 \quad \text{ on } \partial \Omega.
\end{equation}
It suffices to construct $\psi_1$ such that it is $C_b^2(\ol \Omega)$  and satisfies $\frac{\partial \psi_1}{\partial n}=1$ on the boundary.
Then we can define $\psi$ by  $\psi=( \norm{h}_\infty +1)\psi_1$.
 The existence and construction of such a function $\psi_1$ for a $C^2$-domain $\Omega$ satisfying the uniform interior ball condition 
is discussed at the end of this section.  

 From $m$ and $\psi$ we construct a function $\chi$ defined by 
\begin{equation}\label{def_chi_el}
 \chi(x)=m+ \norm{\psi}_{L^\infty(\ol \Omega)} +\psi(x).
\end{equation}
Both $m$ and $M$ are positive constants, which also yield that $\chi$ is positive. 
$M$ serves to cap the score, and the function $\chi$ determines what happens when the cap is reached.
We shall in due course choose $m$ such that $m + 2 \norm{\psi}_{L^\infty} =M-1$ and require that $M$ is sufficiently large. 
Like  the choices of $\alpha$, $\beta$, $\gamma$, the parameters $M$, $m$ and the function $\psi$ 
are used to define the game but they do not influence the resulting PDE.
As in Section~\ref{def_jeu_parabolique}, we proceed in two steps:
\begin{enumerate}[label=$\bullet$]
 \item First we introduce $U^\eps(x,z)$, the optimal worst-case present value of Helen's wealth if the initial 
stock price is $x$ and her initial wealth is $-z$.
 \item Then we define $u^\eps(x)$ and $v^\eps(x)$ as the maximal and minimal initial debt Helen should have at time $t$ to break even upon exit.
\end{enumerate}
The definition of $U^\eps(x,z)$ for $x\in \ol \Omega$ involves a game similar to that of the last section:
\begin{enumerate}
 \item Initially, at time $t_0=0$, the stock price is $x_0=x$ and Helen's debt is $z_0=z$. 
 \item Suppose, at time $t_j=j\eps^2$, the stock price is $x_j$ and Helen's debt is $z_j$ with $|z_j|<M$. Then Helen chooses a vector $p_j \in \R^N$ and 
a matrix $\Gamma_j \in \mathcal{S}^N$, 
restricted in magnitude by \eqref{p_beta_gamma_new}. Knowing these choices, Mark determines the next stock price
$ x_{j+1}=x_j+\Delta x$ so as to degrade Helen's outcome. 
The increment $\Delta x$ allows to model the reflection exactly as in the previous subsections.
Mark chooses an intermediate point $\hat x_{j+1}=x_j +\Delta \hat x_j \in \R^N$ such that 
\begin{equation*} 
\left\|\Delta \hat x_j\right\| \leq \eps^{1-\alpha}.
\end{equation*}
This position $\hat x_{j+1}$ determines the new position $x_{j+1}=x_j +\Delta x_j$ at time $t_{j+1}$ by 
\begin{equation*}  
x_{j+1}= \proj_{\overline \Omega} (\hat x_{j+1}). 
\end{equation*}
Helen experiences a loss at time $t_j$ of 
\begin{equation} \label{Helen_loss_el}
\delta_j =  p_j\cdot \Delta \hat x_j +\frac{1}{2}  \left\langle \Gamma_j \Delta \hat x_j,\Delta \hat x_j \right\rangle
+\eps^2 f(x_j,z_j,p_j,\Gamma_j) - \norm{\Delta \hat x_j - \Delta x_j}  h(x_j+\Delta x_j ) .
\end{equation}
As a consequence, her time $t_{j+1}=t_j+\eps^2$ debt becomes 
\begin{equation*}
z_{j+1}=e^{\lambda \eps^2} (z_j+\delta_j),  
\end{equation*}
where the factor $e^{\lambda \eps^2}$ takes into account her interest payments.
\item  If $z_{j+1} \geq M$, then the game terminates, and Helen pays a ``termination-by-large-debt penalty'' worth  
$e^{\lambda \eps^2}(   \chi(x_j) -\delta_j)$ at time $t_{j+1}$. 
Similarly, if  $z_{j+1} \leq  - M$, the the game terminates, and Helen receives a ``termination-by-large-wealth bonus'' worth 
$e^{\lambda \eps^2}(\chi(x_j) +\delta_j)$ at time $t_{j+1}$. If the game  stops this way we call $t_{j+1}$ the ``ending index'' $t_K$. 
\item If the game has not terminated then Helen and Mark repeat this procedure at time $t_{j+1}=t_j+\eps^2$. 
If the game never stops the ``ending index'' $t_K$ is $+\infty$.
\end{enumerate}
Helen's goal is a bit different from before, due to the presence  of discounting: she seeks 
to maximize the minimum present value of her future income, using the discount factor of $e^{-j\lambda \eps^2}$  for income received at time 
$t_j$. 
 If the game ends by capping at time $t_K$ with $z_K \geq M$, then the present value of her income is
\begin{align*}
U^\eps(x_0,z_0) & =  - z_0 -\delta_0 - e^{-\lambda \eps^2} \delta_1 - \cdots - e^{-(K-1) \lambda \eps^2} \delta_{K-1} 
 - e^{-(K-1) \lambda \eps^2} ( \chi(x_{K-1}) -\delta_{K-1})       \\
 &  = e^{-(K-1) \lambda \eps^2} (-z_{K-1}- \chi(x_{K-1})).
\end{align*}
Similarly, if the game ends by capping at time $t_K$ with $z_K \leq  - M$,  then the present value of her income is
\begin{align*}
U^\eps(x_0,z_0) & =  - z_0 -\delta_0 - e^{-\lambda \eps^2} \delta_1 - \cdots - e^{-(K-1) \lambda \eps^2} \delta_{K-1} 
 + e^{-(K-1) \lambda \eps^2} (\chi(x_{K-1}) + \delta_{K-1})       \\
 &  = e^{-(K-1) \lambda \eps^2} (-z_{K-1}+ \chi(x_{K-1}) ). 
\end{align*}
If the game never ends (since $z_j$ and $\chi(x_j)$ are uniformly bounded),  
we can take $K=\infty$ in the preceding formula to see that the present value of her income is $0$. 

To get a dynamic programming   characterization of $U^\eps$, we observe that if $|z_0|<M$ then 
\begin{equation*}
U^{\eps}(x_0,z_0)=\sup_{p, \Gamma} \inf_{\Delta \hat x }
 \begin{cases}
 e^{-\lambda \eps^2} U^{\eps}(x_1,z_1),   & \text{if } |z_1|<  M  , \\ 
 -z_0 - \chi(x_0) ,                       & \text{if } z_1 \geq  M, \\
 -z_0 + \chi(x_0) ,                       & \text{if } z_1 \leq -M.
\end{cases}
\end{equation*}
Since the game is stationary (nothing distinguishes time 0), the associated dynamic programming principle is that for $|z| <M$,
\begin{equation}\label{dpp_U_el}
U^{\eps}(x,z)=\sup_{p, \Gamma} \inf_{\Delta \hat x }
 \begin{cases}
 e^{-\lambda \eps^2} U^{\eps}(x',z'),    & \text{if } |z'|<  M  , \\ 
 -z - \chi(x),                           & \text{if } z' \geq  M, \\
 -z + \chi(x),                           & \text{if } z' \leq -M,
\end{cases}
\end{equation}
where  $x'=\proj_{\ol \Omega}(x+\Delta \hat x)$ and $z'=e^{\lambda \eps^2}(z+\delta)$, with $\delta$ defined as in \eqref{Helen_loss_el}. 
Here $p$, $\Gamma$ and $\Delta \hat x$ are constrained as usual by \eqref{p_beta_gamma_new}--\eqref{moving1_new}, 
and we write $\sup/\inf$ rather than $\max/\min$ since it is no longer 
clear that the optima are achieved (since the right-hand side is now a discontinuous function of $p$, $\Gamma$ and $\Delta \hat x$).
The preceding discussion defines $U^\eps$ only for $|z|<M$; it is natural to extend the definition to all $z$ by
\begin{equation*}
  U^{\eps}(x,z) =
\begin{cases}
 -z - \chi(x), & \text{ for } z\geq M,\\ 
 -z + \chi(x),  & \text{ for } z\leq  - M,
\end{cases}
\end{equation*}
which corresponds to play being  ``capped'' immediately. Notice that when extended this way, 
$U^\eps$~is strictly negative for $z\geq M$ and strictly positive for $z\leq -  M$.

The definitions of $u^\eps$ and $v^\eps$ are slightly different from those in Section \ref{def_u_v_eps_par}: 
\begin{align}
u^\eps(x_0) & =\sup   \{z_0\, :\, U^\eps(x_0,z_0)> 0 \} \label{def_subelr} ,  \\
v^\eps(x_0) & =\inf   \{z_0\, :\, U^\eps(x_0,z_0)< 0 \}  \label{def_supelr} .
\end{align}
The change from Section \ref{def_u_v_eps_par} is that the inequalities in \eqref{def_u_eps_par}--\eqref{def_v_eps_par} are strict.

\begin{prop}\label{ineq_dyn_prog_el}
 Let $m_1$, $M$ be two constants such that $0<m_1<M$. 
Then whenever $x\in \ol \Omega$ and $-m_1 \leq u^{\eps}(x) <M$ we have 
\begin{multline} \label{dyn_prog_ineq_sub_new_el}
u^\eps(x) \leq \sup_{p,\Gamma} \inf_{\Delta \hat x} \left[e^{ - \lambda \eps^2} u^\eps(x+\Delta x) \right.  \\
\left.     - \left(p\cdot \Delta \hat x +\frac{1}{2}  \left\langle \Gamma \Delta \hat x ,\Delta \hat x \right\rangle
+\eps^2 f(x, u^\eps(x),p ,\Gamma) -  \norm{\Delta \hat x - \Delta x} h(x+\Delta x) \right)  \right] , 
\end{multline}
for $\eps$ small enough (depending on $m_1$ and the parameters of the game but not on $x$). 
Similarly, if $x\in \ol \Omega$ and $-M<v^\eps(x)<m_1$ then for $\eps$ small enough
\begin{multline} \label{dyn_prog_ineq_super_new_el}
v^\eps(x) \geq \sup_{p,\Gamma} \inf_{\Delta \hat x} \left[e^{ - \lambda \eps^2} v^\eps(x+\Delta x)   \right. \\
\left.  - \left(p\cdot \Delta \hat x +\frac{1}{2}  \left\langle \Gamma \Delta \hat  x,\Delta \hat x \right\rangle
+\eps^2 f(x, v^\eps(x),p ,\Gamma) - \norm{\Delta \hat x - \Delta x} h(x+\Delta x)  \right) \right] .
\end{multline}
As usual, the sup and inf are constrained by \eqref{p_beta_gamma_new} and \eqref{moving1_new}  
and $\Delta x$ is determined by \eqref{exp_delta_x}.
\end{prop}

\begin{proof} 
We shall focus on \eqref{dyn_prog_ineq_sub_new_el}; the proof for \eqref{dyn_prog_ineq_super_new_el} follows exactly the same lines.  
Since $-m_1 \leq u^{\eps}(x)<M$, there is a sequence $z^k \uparrow u^{\eps}(x) $ such that $U^\eps(x,z^k)>0$.
Since $u^{\eps}(x)$ is bounded away from $-M$, we may suppose that $z^k$ also remains bounded away from $-M$. 
Dropping the index $k$ for simplicity of notation, consider any such $z=z^k$. 
 The fact that $U^\eps(x,z)>0$ tells us that the right-hand side of the dynamic programming principle~\eqref{dpp_U_el}
 is positive. So there exist $p$, $\Gamma$ constrained by \eqref{p_beta_gamma_new} such that for any $\Delta \hat x$ satisfying \eqref{moving1_new},
\begin{equation} \label{dyn_ineq_0}
0 < \sup_{p, \Gamma}\inf_{\Delta \hat x }
 \begin{cases}
 e^{-\lambda \eps^2} U^{\eps}(x',z'),  & \text{if } |z'|<  M  , \\ 
 -z- \chi(x),                          & \text{if } z' \geq  M, \\
 -z+ \chi(x),                          & \text{if } z' \leq -M,
\end{cases}
\end{equation}
where  $x'=\proj_{\overline \Omega}(x+\Delta \hat x)$ and $z'=e^{\lambda \eps^2}(z+\delta)$. 
Capping above, the alternative $z'\geq M$, cannot happen, since otherwise we compute
\begin{equation*}
-z- \chi(x) = - e^{-\lambda \eps^2}z' - \delta - \chi(x)
            \leq - M   e^{-\lambda \eps^2} - \delta - m \leq - \delta - m< 0,  
\end{equation*}
for $\eps$ small enough because $\delta $ is bounded by a positive power of $\eps$. This sign is a contradiction to our assumption~\eqref{dyn_ineq_0}.
If $\eps$ is sufficiently small, capping below (the alternative $z'\leq - M$) cannot occur either, because 
$z$ is bounded away from $-M$ and $\delta $ is bounded by a positive power of $\eps$. Therefore only the first case can take place 
\begin{equation*}
0<U^{\eps}(x+\Delta x, e^{\lambda \eps^2}(z+\delta)),  
\end{equation*}
whence by the definition of $u^\eps$ given by \eqref{def_subelr}, we deduce that
\begin{equation*}
u^\eps(x+\Delta x) \geq e^{\lambda \eps^2} (z+\delta). 
\end{equation*}
Thus, we have shown the existence of $p$, $\Gamma$ such that  for every $\Delta \hat x$,
\begin{equation} \label{eq_z_el}
z\leq  e^{ - \lambda \eps^2} u^\eps(x+\Delta x)  
- \left( p \cdot \Delta \hat x +\frac{1}{2} \langle \Gamma \Delta \hat x, \Delta \hat x \rangle
+\eps^2 f(x,z,p, \Gamma) - \norm{\Delta \hat x - \Delta x} h(x+\Delta x) \right).
\end{equation}
Recalling that $z=z^k \uparrow u^\eps(x)$, we pass to the limit on both sides of \eqref{eq_z_el}, with $p$, $\Gamma$ held fixed, to see that
\begin{equation*}
u^\eps(x) \leq  e^{-\lambda \eps^2} u^\eps(x+\Delta x)  - \left(  p\cdot \Delta \hat x +\frac{1}{2} \langle  \Gamma \Delta \hat x, \Delta \hat x \rangle 
+\eps^2 f(x,u^{\eps}(x),p, \Gamma) - \norm{\Delta \hat x - \Delta x} h(x+\Delta x) \right). 
\end{equation*}
Since this is true for some $p$, $\Gamma$ and for every $\Delta \hat x$, we have established \eqref{dyn_prog_ineq_sub_new_el}. 
\end{proof}

The PDE \eqref{eq_neumann_el} is the formal Hamilton-Jacobi-Bellman equation associated with the dynamic programming inequalities 
\eqref{dyn_prog_ineq_sub_new_el}--\eqref{dyn_prog_ineq_super_new_el}, by the usual Taylor expansion, if one accepts $-M<v^\eps \approx  u^\eps <M$.
Rather than giving that heuristic argument which is quite similar to the one proposed in the parabolic setting, we now state our 
main result in the stationary setting, which follows from the results in Sections \ref{consistance} and \ref{stability}.  
It concerns the upper and lower relaxed semi-limits, defined  for any $x \in \overline \Omega$,  by
\begin{equation} \label{def_bp_subsup_visc_el}
  \overline u(x)= \limsup_{\substack{y\rightarrow x \\ \eps \rightarrow 0}} u^{\eps}(y)   \quad \text{ and }    \quad  
  \underline v(x)= \liminf_{\substack{y\rightarrow x \\ \eps \rightarrow 0}} v^{\eps}(y), 
\end{equation}
with the convention that $y$ approaches $x$ from $\ol \Omega$ (since $u^\eps$ and $v^\eps$ are defined on $\ol \Omega$).

\begin{theorem}\label{theo_cv_el_neu}
 Consider the stationary boundary value problem \eqref{eq_neumann_el} where $f$ satisfies \eqref{ellipticity_f} and 
\eqref{est_f_el_neu}--\eqref{cont_growth_p_Gamma_el}, $g$ and  $h$ are continuous, uniformly bounded, 
 and $\Omega$ is a $C^2$-domain satisfying both the uniform interior and exterior ball conditions. 
Assume the parameters of the game $\alpha$, $\beta$, $\gamma$ fulfill \eqref{condition_pas}--\eqref{cd_coeff_classiq}, 
$\psi \in C^2(\ol \Omega)$ satisfies~\eqref{def_psi_h_el}, $\chi \in C^2(\ol \Omega)$ is defined by \eqref{def_chi_el}, 
 $M$ is sufficiently large  and  $m=M-1 -2 \norm{\psi}_{L^\infty(\ol \Omega)}$. 
Then $u^\eps$ and $v^\eps$ are well-defined when $\eps$ is sufficiently small, and they satisfy 
$ |u^\eps | \leq \chi$ and  $ |v^\eps| \leq \chi$.  
Their relaxed semi-limits $\overline u$ and $\underline v$ are respectively a viscosity subsolution and a viscosity supersolution of \eqref{eq_neumann_el}. 
If in addition we have $\underline v \leq \overline u$ and the PDE has a comparison principle, then it follows that $u^\eps$ and $v^\eps$ 
converge locally uniformly in $\ol \Omega$ to the unique viscosity solution of \eqref{eq_neumann_el}.
\end{theorem}

This is an immediate consequence of Propositions~\ref{cv_el} and \ref{stability_prop_el}. 
A sufficient condition for $\udl v \leq \bar u$ is that $f$ is nondecreasing in $z$. 
As mentioned in \cite{kohns}, 
sufficient conditions for the PDE to have a comparison principle can be found for example in Section 5 of 
\cite{user_s_guide}, and (for more results) in~\cite{barles_busca}--\cite{barles_rouy}.

\medskip
Let us now go back to the existence and the construction of $\psi_1\in C_b^2(\ol \Omega)$ such that $\frac{\partial \psi_1}{\partial n}= 1$ on~$\partial \Omega$, that 
we will need at various points of the paper.
If $\Omega$ is of class $C^2$ and satisfies the uniform interior ball condition of Definition \ref{unif_int_ball_cd} for a certain $r$, 
 $d$ is $C^2$ on $\Omega(3r/4)$ and an explicit suitable function is 
\begin{equation}\label{def_fonction_bord}
\psi_1(x) =\begin{cases}   \exp\left[  - \frac{d(x)}{1 - \frac{d(x)}{r/2} }\right],  &  \text{ if } d(x)<r/2, \\
                            0 ,                                                      & \text{ if } d(x)\geq r/2.
            \end{cases}
\end{equation}
It is clear that $\text{supp } \psi_1 \subset \Omega(r/2)$, $\psi_1(\ol \Omega) \subset [0,1] $ and  $\psi_1$ is $C^2$ on $\Omega(r/2)$.
Then,  for all $x$ such that $d(x)=\frac{r}{2}$, $D\psi_1$ and $D^2\psi_1$ are continuous at $x$.
Thus $\psi_1$ is  $C^2$ on $\ol \Omega$. 
It is easy to check that the two first derivatives of $\psi_1$ are also bounded and that $\frac{\partial \psi_1}{\partial n} = 1$ on the boundary.
Hence, the function $\psi_1$ defined by \eqref{def_fonction_bord} has all the desired properties.

\begin{remark}\label{est_schauder}
If $\Omega$ is a domain with  $C^{2,\alpha}$-boundary where $\alpha>0$, the Schauder theory \cite[Theorem~6.31]{gilbarg_trudinger_reprint} 
 ensures the solution $\psi$ of the elliptic problem 
\begin{equation*}
\begin{cases}
 \Delta \psi  - \psi =0 , &  \text{ in } {\Omega}   ,          \\
\dfrac{\partial \psi}{\partial n}=   \norm{h}_{L^\infty} +1,  & \text{ on } \partial \Omega,   \\ 
\end{cases}
 \end{equation*}
is $C^{2,\alpha}(\ol \Omega)$. 
In addition, the estimate  $\norm{\psi}_{C^{2,\alpha}(\ol \Omega)} \leq C_\Omega (1 + \norm{h}_{L^\infty})$ holds  
for a certain constant $ C_\Omega$ depending only on the domain.
\end{remark}

\section{Convergence}
\label{convergence}

This section presents our main convergence results, linking the limiting behavior of $v^\eps$ and $u^\eps$ as $\eps \rightarrow 0$ to the PDE. 
The argument uses the framework of \cite{barles_souga_cv_schemes}  and is basically a special case of the theorem proved there.

Convergence is a local issue: in the time-dependent setting, Proposition \ref{cv_par} shows that in any region where the lower and upper 
semi-relaxed limits $\udl v$ and $\bar u$ are finite they are in fact viscosity super and subsolutions respectively. 
The analogous statement for the stationary case is more subtle. In fact, 
 we will need global hypotheses on $f$ at Section~\ref{s_stability_ell} to ensure that $u^{\eps}$ and $v^{\eps}$ are well-defined
and satisfy the dynamic programming inequalities \eqref{dyn_prog_ineq_sub_new_el}--\eqref{dyn_prog_ineq_super_new_el}. Thus, we cannot 
discuss about $\udl v$ or $\bar u$ without global assumptions on $f$. 

\subsection{Viscosity solutions with Neumann condition}
\label{conv_def_visco}
Our PDEs can be degenerate parabolic, degenerate elliptic, or even first order equations. 
Therefore, we cannot expect a classical solution, and we cannot always impose boundary data in 
the classical sense. 
 The theory of viscosity solutions provides the proper framework for handling these issues. 
We review the basic definitions in our setting for the reader's convenience.
We refer to \cite{Barles_book}, \cite{user_s_guide} and \cite{giga_surf_evol_equations} for further details about the general theory.
Consider first the final-value problem \eqref{eq_neumann_evol} in $\Omega$,
\begin{equation*}
\begin{cases}
 - u_t + f(t,x,u, Du, D^2 u ) =0,  & \text{ for } x \in \Omega \text{ and }t<T, \\
\langle Du(x,t),n(x) \rangle = h(x), & \text{ for } x \in \partial \Omega \text{ and }t<T, \\
u(x,T)=g(x),  & \text{ for }  x\in \ol \Omega. 
\end{cases}
\end{equation*} 
where $f(t,x,z,p, \Gamma)$ is continuous in all its variables and satisfies the monotonicity 
condition \eqref{ellipticity_f} in its last variable.
We must be careful to impose the boundary condition in the viscosity sense.

\begin{Def}\label{def_visco_evol}
 A real-valued lower-semicontinuous function $u(x,t)$ defined for $x\in \ol \Omega$  and $t_\ast \leq t \leq T$ 
is a \textit{viscosity supersolution}  
 of the final-value problem \eqref{eq_neumann_evol} if 
\begin{enumerate}[label=\textup{(}{P\arabic*}\textup{)},ref=\textup{(}{P\arabic*}\textup{)}] 
 \item \label{def_visc_sup_int} for any  $(x_0,t_0)$ with $x_0\in \Omega$ and $t_\ast \leq t_0 <T$
and any smooth $\phi(x,t)$ such that $u-\phi$ has a local minimum at $(x_0,t_0)$, we have 
\begin{equation*}
\partial_t \phi(x_0,t_0) -  f(t_0,x_0,u(x_0,t_0),D\phi(x_0,t_0),D^2\phi(x_0,t_0))\leq 0  ,
\end{equation*}
 \item \label{def_visc_sup_bd} 
for any  $(x_0,t_0)$ with $x_0\in \partial \Omega$ and $t_\ast \leq t_0 <T$ 
and any smooth $\phi(x,t)$ such that $u-\phi$ has a local minimum  on $\overline{\Omega}$ at $(x_0,t_0)$, we have 
\begin{align*}
\max \{
-(  \partial_t \phi(x_0,t_0) - &  f(t_0,x_0,u(x_0,t_0),D\phi(x_0,t_0),D^2\phi(x_0,t_0)) ) ,
   \left\langle D\phi(x_0,t_0), n(x_0)\right\rangle - h(x_0)  
\}
\geq 0 ,
\end{align*}
 \item \label{def_visc_sup_finalT} $u\geq g$ at the final time $t=T$.
\end{enumerate}
Similarly, a real-valued upper-semicontinuous function $u(x,t)$ defined for $x\in \ol \Omega$  and $t_\ast \leq t \leq T$  
is a \textit{viscosity subsolution} of the final-value problem \eqref{eq_neumann_evol} if 
\begin{enumerate}[label=\textup{(}{P\arabic*}\textup{)},ref=\textup{(}{P\arabic*}\textup{)}]
 \item \label{def_visc_sub_int} for any  $(x_0,t_0)$ with $x_0\in \Omega$ and $t_\ast \leq t_0 <T$
and any smooth $\phi(x,t)$ such that $u-\phi$ has a local maximum at $(x_0,t_0)$, we have 
\begin{equation*}
\partial_t \phi(x_0,t_0)  -   f(t_0,x_0,u(x_0,t_0),D\phi(x_0,t_0),D^2\phi(x_0,t_0))\geq 0  ,
\end{equation*}
 \item \label{def_visc_sub_bd} 
for any  $(x_0,t_0)$ with $x_0\in \partial \Omega$ and $t_\ast \leq t_0 <T$ 
and any smooth $\phi(x,t)$ such that $u-\phi$ has a local maximum  on $\overline{\Omega}$ at $(x_0,t_0)$, we have 
\begin{align*}
\min \{ -( \partial_t \phi(x_0,t_0)  -&  f(t_0,x_0,u(x_0,t_0),D\phi(x_0,t_0),D^2\phi(x_0,t_0))   ),
   \left\langle D\phi(x_0,t_0), n(x_0) \right\rangle - h(x_0) \}  \leq 0   ,
\end{align*} 
  \item \label{def_visc_sub_finalT} $u\leq g$ at the final time $t=T$.
\end{enumerate}
A viscosity solution of \eqref{eq_neumann_evol} is a continuous function $u$ that is 
both a viscosity subsolution and a viscosity supersolution of \eqref{eq_neumann_evol}.  
\end{Def}

In the stationary problem \eqref{eq_neumann_el}, the definitions are similar to the time-dependent setting.

\begin{Def}\label{def_visco_el}
 A real-valued lower-semicontinuous function $u(x)$ defined on $\overline{\Omega}$ 
 is a \textit{viscosity supersolution} of the stationary problem \eqref{eq_neumann_el} if 
\begin{enumerate}[label=\textup{(}{E\arabic*}\textup{)},ref=\textup{(}{E\arabic*}\textup{)}] 
 \item \label{def_visc_sup_int_el} for any  $x_0\in \Omega$ and any smooth $\phi(x)$ such that $u-\phi$ has a local minimum at $x_0$, we have 
\begin{equation*}
  f(x_0,u(x_0),D\phi(x_0),D^2\phi(x_0))  +\lambda u(x_0) \geq 0,
\end{equation*}
 \item \label{def_visc_sup_bd_el} for any $x_0 \in \partial \Omega$ and any smooth $\phi(x)$ such that $u-\phi$ has a local minimum
on $\overline{\Omega}$ at $x_0$, we have 
\begin{align*}
\max \{  f(x_0,u(x_0),D\phi(x_0),D^2\phi(x_0))  +\lambda u(x_0),   \left\langle D\phi(x_0), n(x_0)\right\rangle - h(x_0)  \}\geq 0 .
\end{align*}
\end{enumerate}
Similarly, a real-valued upper-semicontinuous function $u(x)$ defined on $\overline{\Omega}$   
is a \textit{viscosity subsolution}   of the stationary problem \eqref{eq_neumann_el} if 
\begin{enumerate}[label=\textup{(}{E\arabic*}\textup{)},ref=\textup{(}{E\arabic*}\textup{)}]
 \item \label{def_visc_sub_int_el} for any  $x_0\in \Omega$ and any smooth $\phi(x)$ such that $u-\phi$ has a local maximum at $x_0$, we have 
\begin{equation*}
  f(x_0,u(x_0),D\phi(x_0),D^2\phi(x_0))  +\lambda u(x_0) \leq 0, 
\end{equation*}
 \item \label{def_visc_sub_bd_el} for any $x_0 \in \partial \Omega$ and any smooth $\phi(x)$ such that $u-\phi$ has a local maximum 
on $\overline{\Omega}$ at $x_0$, we have 
\begin{equation*}
\min \{  f(x_0,u(x_0),D\phi(x_0),D^2\phi(x_0))  +\lambda u(x_0), \left\langle D\phi(x_0), n(x_0) \right\rangle - h(x_0) \} \leq 0 .
\end{equation*}
\end{enumerate}
A viscosity solution of \eqref{eq_neumann_el} is a continuous function $u$ 
that is both a viscosity subsolution and a viscosity supersolution of \eqref{eq_neumann_el}. 
\end{Def}

In stating these definitions, we have assumed that the final-time data $g$ and the boundary 
Neumann condition $h$ are continuous. In Definition \ref{def_visco_evol}, 
the pointwise inequality in part \ref{def_visc_sup_finalT} can be replaced by an apparently 
weaker condition analogous to part \ref{def_visc_sup_bd}. 
The equivalence of such a definition  with the one stated above is standard, 
the argument uses barriers of the form $ \phi(x,t)=\norm{x-x_0}^2/\delta + (T-t)/\mu + K d(x)$ with $\delta$ and $\mu$ sufficiently small, 
and is contained in our proof of Proposition \ref{cv_par}~\ref{cas_cv_par2}.
We shall be focusing on the lower and upper semi-relaxed limits of $v^{\eps}$ and $u^{\eps}$, 
defined by \eqref{def_bp_subsup_visc} in the time-dependent setting and \eqref{def_bp_subsup_visc_el} in the stationary case.

\medskip

We now provide a key definition to deal with the Neumann boundary condition within viscosity solutions framework which will 
be essential all along the article.
We  introduce some applications which give bounds on the Neumann penalization term for a smooth function and $x$ close to the boundary. 
This approach is well-suited to the viscosity solutions framework. More precisely,   
we define the applications $ m_\eps$ and $M_\eps$, for all $x\in \Omega(\eps^{1- \alpha})$ and $\phi \in C^1 (\overline \Omega)$,  by
\begin{align}
m_\eps^{x}[\phi]   & 
:=\inf_{\substack{x+\Delta \hat x \notin \Omega  \\ \Delta \hat x}}
\left\{ h(x+\Delta x)  - D \phi(x)\cdot n(x+\Delta x)  \right\}, \label{m_eps_par} \\
M_\eps^{x}[\phi] &  
:=\sup_{\substack{x+\Delta \hat x \notin \Omega  \\ \Delta \hat x }}
 \left\{ h(x+\Delta x)  - D \phi(x)\cdot n(x+\Delta x)  \right\}, \label{M_eps_par}
\end{align}
where  $\Delta \hat x$ is constrained by \eqref{moving1_new} and determines $\Delta x$  by \eqref{exp_delta_x}. 
Notice that the functions $ m_\eps^{\cdot}[\phi]$ and $M_\eps^{\cdot}[\phi]$ are 
bounded by $\norm{h}_{L^\infty}+ \norm{D\phi}_{L^\infty(\ol \Omega)}$. 
Since $h$ is supposed to be continuous, the following property clearly holds.

\begin{lemma}\label{lemma_cv_m_M} Let $x\in \partial \Omega$ and $\phi \in C^1(\overline \Omega)$. 
Suppose there exists a sequence $(\eps_k, x_k)_{k\in \N}$ in $\R_+^\ast \times \overline \Omega$ 
convergent to $(0,x)$ such that for all $k$ large enough, $x_k\in \Omega(\eps_k^{1-\alpha})$. Then 
\begin{equation*}
\lim_{k\rightarrow +\infty}  m_{\eps_k}^{x_k}[\phi] =\lim_{k\rightarrow +\infty}  M_{\eps_k}^{x_k}[\phi] =h(x) - D\phi(x)\cdot n(x).
\end{equation*}
Similarly, let $\phi \in C^1(\overline \Omega \times [0,T])$.
Suppose there exists a sequence $(\eps_k, x_k,t_k)_{k\in \N}$ in $\R_+^\ast  \times \overline \Omega \times [0,T]$ 
convergent to $(0,x,t)$ such that for all $k$ large enough, $x_k\in \Omega(\eps_k^{1-\alpha})$. Then 
\begin{equation*}
\lim_{k\rightarrow +\infty}  m_{\eps_k}^{x_k}[\phi(\cdot, t_k)] =\lim_{k\rightarrow +\infty}  M_{\eps_k}^{x_k}[\phi(\cdot, t_k)] 
=h(x) - D\phi(x,t)\cdot n(x) .
\end{equation*}
\end{lemma}

\subsection{The parabolic case}
\label{cv_par_case}
We are ready to state our main convergence result in the time-dependent setting. 
At first sight, the proof seems  to use the monotonicity condition \eqref{ellipticity_f}. 
The proof relies on the consistency of the numerical scheme, Propositions \ref{cons_lower_bound}, \ref{cons_new_sub} and \ref{est_stab_sub_ord2}, 
which are proved in Section \ref{consistance}. Proposition \ref{est_stab_sub_ord2} is necessary to deal with the degeneration of the consistency estimates due to 
the Neumann boundary condition.
So we also require that $f(t,x,z,p, \Gamma)$ satisfy \eqref{loc_lip_p_Gamma}--\eqref{cont_growth_p_Gamma}, 
and that the parameters $\alpha$, $\beta$, $\gamma$ satisfy \eqref{condition_pas}--\eqref{cd_coeff_classiq}. 

\begin{prop} \label{cv_par}
Suppose $f$ and $\alpha$, $\beta$, $\gamma$ satisfy the hypotheses just listed. Assume furthermore that
$\bar{u}$ and $\underline{v}$ are finite for all $x$ near $x_0$ and all $t\leq T$ near $t_0$. Then
 \begin{enumerate}[label=\textup{ }{\roman*.}\textup{ },ref=({\roman*})]
   \item \label{cas_cv_par1} if $t_0<T$ and $x_0\in \Omega$, 
then $\bar{u}$ is a viscosity subsolution at $x_0$ and $\underline{v}$ is a supersolution at $x_0$ 
(i.e. each one satisfies part \ref{def_visc_sup_int} of the relevant half of Definition \ref{def_visco_evol} at $x_0$).
   \item \label{cas_cv_par1bd} if $t_0<T$ and $x_0\in \partial \Omega$, 
then $\bar{u}$ is a viscosity subsolution at $x_0$ and $\underline{v}$ is a supersolution at $x_0$ 
(i.e. each one satisfies part \ref{def_visc_sup_bd} of the relevant half of Definition~\ref{def_visco_evol} at $x_0$).
   \item \label{cas_cv_par2} if $t_0=T$, then  $\bar{u}(x_0,T)=g(x_0)$ and $\underline{v}(x_0,T)=g(x_0)$
(in particular, each one satisfies part~\ref{def_visc_sup_finalT} of the relevant half of Definition \ref{def_visco_evol} at $x_0$).
  \end{enumerate}
In particular, if $\bar{u}$ and $\underline{v}$ are finite for all $x\in \ol \Omega$ and $t_\ast<t\leq T$, then they are
respectively a viscosity subsolution and a viscosity supersolution of \eqref{eq_neumann_evol} on this time interval.
\end{prop}

\begin{proof} When $x_0\in \Omega$, since we can find in $\Omega$ a $\delta$-neighborhood  of $x_0$, the proof follows from \cite[Proposition~3.3]{kohns}. 
Therefore we shall focus on the case when $x_0\in \partial \Omega$.  We give the proof for $\overline u$. The argument for $\udl v$ is entirely parallel, 
relying on Proposition \ref{cons_lower_bound}.  
Consider a smooth function $\phi$ such that $\bar{u}- \phi$ has a local maximum at $(x_0,t_0)$. Adding a 
constant, we can assume $\overline{u}(x_0,t_0)=\phi(x_0,t_0)$.  
Replacing $\phi$ by $\phi+ \norm{x-x_0}^4 + |t-t_0|^2$ if necessary, we can assume that the local maximum is 
strict, i.e.  that
\begin{equation}\label{loc_max_strict}
 \bar{u}(x,t)<\phi(x,t) \quad \text{ for } 0< \norm{(x,t)-(x_0,t_0)}\leq r, 
\end{equation}
for some $r>0$. By the definition of $\bar{u}$, there exist sequences $\eps_k$, $\tilde{y}_k$, $\tilde{t}_k=T-\tilde{N}_k \eps^2_k$ such that
\begin{equation*}
\tilde{y}_k\rightarrow x_0, \quad \tilde{t}_k\rightarrow t_0, \quad u^{\eps_k}(\tilde{y}_k,\tilde{t}_k) \rightarrow \bar{u}(x_0,t_0). 
\end{equation*}
Let $y_k$ and $t_k=T-N_k \eps^2_k$ satisfying
\begin{equation*}
(u^{\eps_k}-\phi)(y_k,t_k) \geq \sup_{\norm{(x,t)-(x_0,t_0)}\leq r}  (u^{\eps_k}-\phi)(x,t) - \eps_k^3.
\end{equation*}
Notice that since $u^{\eps_k}$ is defined only at discrete times, the sup is taken only over such times. Evidently,
\begin{equation*}
(u^{\eps_k} - \phi)(y_k,t_k)\geq (u^{\eps_k}-\phi)(\tilde{y}_k,\tilde{t}_k) -\eps_k^3
\end{equation*}
and the right-hand side tends to 0 as $k\rightarrow +\infty$. It follows using \eqref{loc_max_strict} that 
\begin{equation*}
(y_k,t_k) \rightarrow (y_0,t_0) \quad \text{ and } \quad u^{\eps_k}(y_k,t_k) \rightarrow \bar{u}(x_0,t_0) ,
\end{equation*}
as $k\rightarrow +\infty$. Setting $\xi_k=u^{\eps_k}(y_k,t_k) -\phi(y_k,t_k)$, we also have by construction that
\begin{equation}
\xi_k \rightarrow 0 \text{ and }  u^{\eps_k}(x,t) \leq \phi(x,t)+\xi_k +\eps_k^3  \quad \text{ whenever }t=T-n_k\eps_k^2 \text{ and }\norm{(x,t)-(x_0,t_0)}\leq r.
\label{eq_xi_u_phi_sub}
 \end{equation}
Now we use the dynamic programming inequality \eqref{dyn_prog_ineq_sub_new} at $(y_k,t_k)$, which can be written  concisely as
\begin{equation*}
u^{\eps_k}(y_k,t_k)\leq \sup_{p,\Gamma} \inf_{\Delta \hat x} \left\{ u^{\eps_k}(y_k +\Delta x, t_k +\eps_k^2)-\Delta z \right\}, 
\end{equation*}
with the convention 
\begin{equation*}
\Delta z= p\cdot \Delta \hat x +\frac{1}{2} \langle \Gamma \Delta \hat x, \Delta \hat x \rangle
 +\eps_k^2 f(t_k,y_k, u^{\eps_k}(y_k,t_k), p, \Gamma) - \norm{\Delta \hat x - \Delta x} h(y_k+\Delta x)  .
\end{equation*}
Using the definition of $\xi_k$, \eqref{eq_xi_u_phi_sub}, and the fact that 
$\Delta x$ is bounded by a positive power of $\eps$, we conclude that
\begin{equation} \label{inegalite_sub_phi_xi} 
\phi(y_k,t_k) +\xi_k \leq \sup_{p,\Gamma} \inf_{\Delta \hat  x} \left\{  \phi(y_k+\Delta x, t_k+\eps^2_k) +\xi_k+\eps^3_k - \Delta z \right\}, 
\end{equation}
when $k$ is sufficiently large. Dropping $\xi_k$ from both sides of  \eqref{inegalite_sub_phi_xi}, 
we deduce, by introducing the operator $S_\eps$ defined by \eqref{def_new_op}, that
\begin{equation}\label{ineq_sub_phi_op1} 
 \phi(y_k ,t_k)  \leq S_\eps[y_k,t_k,u^{\eps_k}(y_k,t_k), \phi(\cdot, t_k+\eps^2_k)] +o(\eps_k^2).
\end{equation}
According to the consistency estimates provided by Proposition \ref{cons_new_sub}, 
we shall introduce four sets  $(A_i)_{1\leq i \leq 4}$ respectively defined by
\begin{align*}
A_1 &    :=\left\{k\in \N: d(y_k)\leq \eps_k^{1-\alpha} \text{ and } M_{\eps_k}^{y_k}[\phi(\cdot, t_k+\eps_k^2)]
 \geq \frac{4}{3} \norm{D^2\phi(y_k,t_k+\eps_k^2)} \eps_k^{1-\alpha} \right\},  \\
A_2 & :=\left\{k\in \N: \eps_k^{1-\alpha}- \eps_k^{\rho} \leq d(y_k) \leq \eps_k^{1-\alpha}\text{ and } M_{\eps_k}^{y_k}[\phi(\cdot, t_k+\eps_k^2)] 
\leq \frac{4}{3} \norm{D^2\phi(y_k,t_k+\eps_k^2)} \eps_k^{1-\alpha}  \right\} \\
  & \phantom{=============} 
\bigcup  \bigg\{k\in \N :  d(y_k)\geq \eps_k^{1-\alpha}  \bigg\}   ,   \\
A_3 & :=\left\{k\in \N: d(y_k) \leq \eps_k^{1-\alpha} - \eps_k^{ \rho}    \text{ and }  
-\eps_k^{1-\alpha - \kappa} \leq   M_{\eps_k}^{y_k}[\phi(\cdot, t_k+\eps_k^2)] \leq \frac{4}{3}    \norm{D^2\phi(y_k,t_k+\eps_k^2)} \eps_k^{1-\alpha} \right\} ,  \\
A_4 & :=\left\{k\in \N: d(y_k) \leq \eps_k^{1-\alpha} - \eps_k^{ \rho} \text{ and }
  M_{\eps_k}^{y_k}[\phi(\cdot, t_k+\eps_k^2)] \leq -\eps_k^{1-\alpha - \kappa} \right\}, 
\end{align*}
where $\rho$  and $\kappa$  are defined in Section \ref{consistance_estimates_para} by \eqref{def_nul} and \eqref{relation_gamma_nul_tilde} and satisfy 
$0 < \kappa <1-\alpha < \rho <1  $. Since $\cup_{1\leq i \leq 4} A_i=\N$, at least one set among $A_{1}$, $A_{2}$, $A_{3}$ and $A_4$ is necessarily unbounded. 
We shall consider these four cases. 
\begin{enumerate}[label=$\bullet$, align=right, leftmargin=*, noitemsep]
\item If $A_1$ is unbounded, up to a subsequence, we may assume that $A_1=\N$. 
Taking the limit $k\rightarrow +\infty$, we deduce that $\ds \liminf_{k\rightarrow +\infty} M_{\eps_k}^{y_k}[\phi(\cdot, t_k+\eps_k^2)] \geq 0$.
Since $M_{\eps_k}^{y_k}[\phi(\cdot, t_k+\eps_k^2)]\rightarrow  h(x_0) - D \phi(x_0, t_0)\cdot n(x_0) $ by Lemma~\ref{lemma_cv_m_M}, 
it follows in  the limit $k\rightarrow \infty$ that
\begin{equation}\label{cond_neumann_sub}
  D \phi(x_0, t_0)\cdot n(x_0) - h(x_0) \leq  0. 
\end{equation}
We can notice this result also holds through \eqref{ineq_sub_phi_op1}.
We can apply the second alternative given by \eqref{estsup_par_g}  in Proposition~\ref{est_stab_sub_ord2}    
to evaluate the right-hand side of \eqref{ineq_sub_phi_op1}. This gives
\begin{equation*}
\phi(y_k,t_k) - \phi(y_k, t_k+\eps_k^2) 
 \leq 3 \eps_k^{1-\alpha} M_{\eps_k}^{y_k}[\phi(\cdot, t_k+\eps_k^2)] + C \eps_k^2 (1+|u^{\eps_k}(y_k,t_k)|)+o(\eps_k^2), 
\end{equation*}
where $C$ depends only on $\norm{h}_{L^\infty}$ and $\norm{D\phi(\cdot,t_k+\eps_k^2)}_{C^1_b(\ol \Omega \cap B(y_k,\eps_k^{1-\alpha}))}$. 
Since for $k$ large enough, 
\begin{equation*} 
\norm{D\phi(\cdot,t_k+\eps_k^2)}_{C^1_b(\ol \Omega \cap B(y_k,\eps_k^{1-\alpha}))}
 \leq \sup_{|t-t_0|\leq r}\norm{D\phi(\cdot,t)}_{C^1_b(\ol \Omega \cap B(x_0,r))},   
\end{equation*}
we can suppose that $C$ depends only on $\norm{h}_{L^\infty}$ and this sup, which is finite (since $\phi$ is smooth) and independent of $k$. 
By smoothness of $\phi$ we have
\begin{equation*}
-\eps_k^2 \partial_t \phi(y_k,t_k)+o(\eps^2_k)-C(1+|u^{\eps_k}(y_k,t_k)|)\eps_k^2  \leq  3 \eps_k^{1-\alpha} 
M_{\eps_k}^{y_k}[\phi(\cdot, t_k+\eps_k^2)]. 
\end{equation*}
By dividing by $\eps_k^{1-\alpha}$ we obtain
\begin{equation*}
-\eps_k^{1+\alpha} \Big(\partial_t \phi(y_k,t_k)-C(1+|u^{\eps_k}(y_k,t_k)|)\Big) + o(\eps^{1+\alpha}_k) 
 \leq 3  M_{\eps_k}^{y_k}[\phi(\cdot,t_k+\eps_k^2)].
\end{equation*}
The sequences $(u^{\eps_k}(y_k,t_k))_{k\in \N}$  and  $(\partial_t\phi(y_k,t_k))_{k\in \N}$ are respectively bounded by 
definition of $\overline u(x_0,t_0)$ and smoothness of $\phi$.
By passing to the limit on $k$, $\ds  \liminf_{k\rightarrow +\infty} M_{\eps_k}^{y_k}[\phi(\cdot,t_k+\eps_k^2)] \geq 0$.
By Lemma~\ref{lemma_cv_m_M}, we know that $M_{\eps_k}^{y_k} [\phi(\cdot,t_k+\eps_k^2)]\rightarrow  h(x_0) - D \phi(x_0, t_0)\cdot n(x_0) $
and \eqref{cond_neumann_sub} is retrieved.

 \item If $A_{2}$ is unbounded, up to a subsequence, we may assume that $A_2=\N$.
We can apply Proposition~\ref{cons_new_sub} case~\ref{cons_sub_cas2} to evaluate the right-hand side of \eqref{ineq_sub_phi_op1}. This gives
\begin{align*}
 \phi(y_k ,t_k) 
 & \leq \phi(y_k, t_k+\eps_k^2) - \eps_k^2 f(t_k,y_k,u^{\eps_k}(y_k,t_k), D\phi(y_k, t_k+\eps_k^2), D^2\phi(y_k, t_k+\eps_k^2)) +o(\eps^2_k).
\end{align*}
By smoothness of $\phi$ and Lipschitz continuity of $f$ with respect to $p$ and $\Gamma$, we obtain 
\begin{equation*}
\phi(y_k,t_k) - \phi(y_k, t_k+\eps_k^2) \leq - \eps_k^2 f(t_k,y_k,u^{\eps_k}(y_k,t_k), D\phi(y_k, t_k), D^2\phi(y_k, t_k)) +o(\eps_k^2).
\end{equation*}
It follows in  the limit $k\rightarrow \infty$ that
\begin{equation}\label{cond_edp_sub}
\partial_t\phi(x_0,t_0) - f(t_0,x_0, \bar u (x_0,t_0), D\phi(x_0,t_0), D^2\phi(x_0,t_0)) \geq 0. 
\end{equation}
\item If $A_3$ is unbounded, up to a subsequence, we may assume that $A_3=\N$. 
By passing to the limit on $k$, we have that $M_{\eps_k}^{y_k,t_k+\eps_k^2}[\phi]$ tends to zero when $\eps_k$
tends to zero. Since $M_{\eps_k}^{y_k}[\phi(\cdot,t_k+\eps_k^2)]\rightarrow  h(x_0) - D \phi(x_0, t_0)\cdot n(x_0)$ by Lemma~\ref{lemma_cv_m_M}, 
it follows in  the limit $k\rightarrow \infty$ that $D \phi(x_0, t_0)\cdot n(x_0) - h(x_0) = 0$.
\item If $A_{4}$ is unbounded, up to a subsequence, we may assume that $A_4=\N$.
Hence, taking the limit $k \rightarrow +\infty$, we have 
\begin{equation} \label{ineq_limsup}
\ds \limsup_{k\rightarrow +\infty} M_{\eps_k}^{y_k}[\phi(\cdot,t_k+\eps_k^2)] \leq 0. 
\end{equation}
Moreover, by applying the fourth alternative given by \eqref{estsup_par_g}  in Proposition~\ref{est_stab_sub_ord2} 
to evaluate the right-hand side of \eqref{ineq_sub_phi_op1}, 
we obtain
\begin{equation*}
 \phi(y_k ,t_k) 
  \leq \phi(y_k, t_k+\eps_k^2) + \dfrac{1}{4} \eps_k^{ \rho} M_{\eps_k}^{y_k} [\phi(\cdot,t_k+\eps_k^2)] +C \eps_k^2 (1+|u^{\eps_k}(y_k,t_k)|)+o(\eps_k^2), 
\end{equation*}
where $C$ depends only on $\norm{h}_{L^\infty}$ and $\ds \sup_{|t-t_0|\leq r}\norm{D\phi(\cdot,t)}_{C^1_b(\ol \Omega \cap B(x_0,r))}$ 
by the same arguments used above for $A_1$. 
By smoothness of $\phi$ we have 
\begin{equation*}
-\eps_k^2 \partial_t \phi(y_k,t_k)+o(\eps^2_k) - C (1+|u^{\eps_k}(y_k,t_k)|)\eps_k^2 
\leq  \frac{1}{4} \eps_k^{ \rho}  M_{\eps_k}^{y_k}[\phi(\cdot,t_k+\eps_k^2)].
\end{equation*}
By dividing by $\eps_k^{\rho}$ we get
\begin{align*}
-\eps_k^{2-\rho} \Big(\partial_t \phi(y_k,t_k) - C(1+|u^{\eps_k}(y_k,t_k)|) \Big) +  o(\eps_k^{2- \rho}) & 
 \leq \frac{1}{4} M_{\eps_k}^{y_k}[ \phi(\cdot,t_k+\eps_k^2)] .
\end{align*}
The sequences $(u^{\eps_k}(y_k,t_k))_{k \in \N}$  and  $(\partial_t\phi(y_k,t_k))_{k \in \N}$ are respectively bounded by 
definition of $\overline u(x_0,t_0)$ and smoothness of $\phi$.
By passing to the limit as $k \rightarrow +\infty$,  we have 
\begin{equation*}
 \liminf_{k\rightarrow +\infty} M_{\eps_k}^{y_k}[\phi(\cdot,t_k+\eps_k^2)] \geq 0.
\end{equation*}
By comparing this inequality with \eqref{ineq_limsup} and using Lemma \ref{lemma_cv_m_M}, we deduce that 
\begin{equation*}
  D \phi(x_0, t_0)\cdot n(x_0) - h(x_0)  = 0.
\end{equation*}
Moreover, we can also apply Lemma \ref{lem_elem_M_neg} since $\eps^{1-\alpha}_k \ll \eps_k^{1-\alpha -\kappa}$. By the same manipulations as those done 
for the set $A_2$, the inequality \eqref{cond_edp_sub} holds also true.
\end{enumerate}
Thus $\bar{u}$ is a viscosity subsolution at $(x_0,t_0)$.

We turn now to case \ref{cas_cv_par2}, i.e. the case $t_0=T$. 
If $x_0\in \Omega$, the analysis led in \cite[Proposition~3.3]{kohns} gives the result. It remains to study  $\overline u$ on the boundary.    
We want to show that $\overline u(\cdot, T)=g$ is also satisfied on $\partial \Omega$. 
By the definition of $\ol u$ given by \eqref{def_bp_subsup_visc} and considering a particular sequence $(\eps_k, x_k, t_k=T)_{k \in \N}$ 
which converges to $(0, x_0, T)$, it is clear that $\overline u(\cdot, T)\geq g$ on $\partial \Omega$
(using the continuity of $g$ and the fact that each $u^\eps$ has final value $g$).
 If this sequence realizes the sup, we have in fact the equality. 
 The preceding argument can still be used provided $t_k<T$ for all sufficiently large $k$.  Thus, considering the different possibilities according to 
$t_k<T$ or $t_k=T$ and also on $x_k\in \Omega$ or $x_k\in \partial \Omega$, we know that
 for any smooth $\phi$ such that $\bar{u}-\phi$ has a local maximum at $(x_0,T)$, 
\begin{multline} \label{alternative_T_1} 
\text{either }   \bar{u}(x_0,T) =  g(x_0) \text{ or else } \\
 \max \left( \partial_t\phi(x_0,T) - f(t_0,x_0, \bar{u}(x_0,T), D\phi(x_0,T), D^2\phi(x_0,T)), 
             h(x_0) - D\phi(x_0, T)\cdot n(x_0) \right)\geq 0. 
\end{multline}
Moreover this statement applies not only at the given point $x_0$, but also at any point nearby.

Now consider the functions
\begin{equation*}
 \psi(x,t) =\bar{u}(x,t) - \frac{\norm{x-x_0}^2}{\eta}- \frac{T-t}{\mu} + K d(x) 
\end{equation*}
and 
\begin{equation} \label{f_test_bord}
 \phi(x,t)  = \frac{\norm{x-x_0}^2}{\eta}+ \frac{T-t}{\mu} - K d(x),   
\end{equation}
where the parameters $\eta$, $\mu$ are small and positive and $K=\norm{h}_{L^\infty}+1$. 
Suppose $\ol u$ is uniformly bounded on the closed half-ball $\{\norm{(x,t)-(x_0,T)}\leq r, t\leq T  \}$ and let $\psi$ attain its maximum on this 
half-ball at $(x_{\eta, \mu}, t_{\eta,\mu})$. We assume $r$ is small enough such that $d$ is $C^2$ on this half-ball so that $\phi$ can be taken
 as a test function. 
We clearly have
\begin{equation} \label{rel_u_bar}
\bar{u}(x_{\eta,\mu}, t_{\eta, \mu}) + K d(x_{\eta,\mu}) \geq \psi(x_{\eta,\mu}, t_{\eta, \mu}) \geq \psi(x_0,T)=\ol u (x_0,T) .
\end{equation}
By plugging the expression of $\psi(x_{\eta,\mu}, t_{\eta, \mu})$ in the right-hand side of inequality \eqref{rel_u_bar}, we obtain
\begin{equation}\label{encadr_xt}
0\leq  \frac{\norm{x_{\eta, \mu} - x_0}^2}{\eta }+\frac{T-t_{\eta, \mu}}{\mu} \leq 
\bar{u}(x_{\eta,\mu}, t_{\eta, \mu}) - \ol u (x_0,T)+K d(x_{\eta, \mu}).  
\end{equation}
Since $\ol u$ is bounded on the half-ball and $x_{\eta,\mu}$ belongs to the half ball for all $\eta$ and $\mu$, the right-hand side
of \eqref{encadr_xt} is bounded independently of $\eta$, $\mu$, which yields 
\begin{equation} \label{conv_xt_eta_mu}
 (x_{\eta, \mu}, t_{\eta,\mu} ) \rightarrow (x_0,T) \quad \text{ as } \eta, \mu \rightarrow 0.
\end{equation}
By using the upper semicontinuity of $\ol{u}$ and taking the limit \eqref{conv_xt_eta_mu} in \eqref{rel_u_bar}, we get
\begin{equation}\label{conv_u_bar}
\ol u (x_{\eta,\mu}, t_{\eta, \mu})  \rightarrow  \ol u (x_0,T) \quad \text{ as } \eta, \mu \rightarrow 0.
\end{equation}
By combining \eqref{conv_xt_eta_mu} and \eqref{conv_u_bar}, we finally obtain by \eqref{encadr_xt} that   
\begin{equation}\label{vitesse_convergence}
\frac{\norm{x_{\eta, \mu} - x_0}^2}{\eta }+\frac{T-t_{\eta, \mu}}{\mu}\rightarrow 0  \quad \text{ as } \eta, \mu \rightarrow 0.
\end{equation}
If $t_{\eta, \mu}<T$ and $x_{\eta, \mu}\in \Omega$ then part \ref{cas_cv_par1} of Proposition \ref{cv_par} applied to $\phi$ defined by \eqref{f_test_bord}
assures us  that
\begin{equation}\label{ineq_usual_int}
 -\frac{1}{\mu} - f(t_{\eta, \mu}, x_{\eta, \mu}, \ol u (x_{\eta,\mu}, t_{\eta, \mu}) , 
2 \frac{x_{\eta,\mu}-x_0}{\eta}-KDd(x_{\eta,\mu})   , \frac{2}{\eta}I -KD^2d(x_{\eta,\mu})) \geq 0.
\end{equation}
Since $f$ is continuous, for any $\eta >0$ there exists $\mu>0$ such that \eqref{ineq_usual_int} cannot happen. 
Restricting our attention to such choices of $\eta$ and $\mu$, it remains to examine two situations:  
on the one hand $t_{\eta, \mu}<T$ and $x_{\eta, \mu}\in \partial \Omega$ and on the other hand $t_{\eta, \mu}=T$.
Arguing by contradiction, let us assume that $t_{\eta, \mu}<T$ and $x_{\eta, \mu}\in \partial \Omega$. 
 By the Taylor expansion on the distance function close to $x_0$, we have 
\begin{equation*}
d(x)=d(x_0)+Dd(x_0)\cdot (x-x_0) +O(\norm{x-x_0}^2).
\end{equation*}
By using that $x_0$ and $x_{\eta, \mu}$ are on the boundary $ \partial \Omega$, $d(x_0)= d(x_{\eta, \mu})=0$ and $Dd(x_0)=-n(x_0)$, this relation reduces to 
\begin{equation}\label{miracle_distance}
 n(x_0)\cdot (x_{\eta, \mu}-x_0)= O(\norm{x_{\eta, \mu}-x_0}^2).
\end{equation}
By combining \eqref{vitesse_convergence} and \eqref{miracle_distance}, we compute
\begin{align*}
D \phi(x_{\eta, \mu}, t_{\eta, \mu})\cdot n(x_0) &  = \frac{2}{  \eta} (x_{\eta, \mu}-x_0)\cdot n(x_0) 
- K Dd(x_{\eta, \mu}) \cdot n(x_0) \\
& =O\left(\dfrac{\norm{x_{\eta, \mu} - x_0}^2}{\eta}\right) +K n(x_{\eta, \mu}) \cdot n(x_0) \rightarrow K,  \text{ as } \eta, \mu \rightarrow 0.
\end{align*}
By smoothness of $\phi$ and continuity of $n$ on $\partial \Omega$, we deduce  that
$D \phi(x_{\eta, \mu}, T)\cdot n(x_{\eta, \mu})\rightarrow \norm{h}_{L^\infty}+1>h(x_{\eta, \mu})$ which denies the second alternative 
proposed at \eqref{alternative_T_1}. As a result, the only remaining possibility for \eqref{alternative_T_1} is 
$\ol u (x_{\eta, \mu},T)=g(x_{\eta, \mu})$. By continuity of $g$, it follows in the limit $\eta, \mu \rightarrow 0$
that $\ol u (x_0,T)=g(x_0)$, as asserted. 
\end{proof}

\begin{remark} \label{hyp_mini_domain}
In the proof of the convergence at the final-time in Theorem \ref{cv_par}, we needed in a essential way that the domain was assumed to be at least $C^2$. 
More precisely, in this case, since the distance function $d$ is $C^2$ in a neighborhood of the boundary, 
it allows us to take  $\phi$ given by \eqref{f_test_bord}  as a test function. 
\end{remark}

\subsection{The elliptic case}

We turn now to the stationary setting discussed  in Section \ref{rules_el_game}. As in the time-dependent setting, our convergence result depends on the 
fundamental consistency result Proposition \ref{cons_new_sub_el}. So we require that the parameters $\alpha$, $\beta$, $\gamma$ 
satisfy \eqref{condition_pas}--\eqref{cd_coeff_classiq}, and that $f(x,z,p,\Gamma)$ satisfy not only 
the monotonicity condition \eqref{ellipticity_f} but also the Lipschitz continuity 
and growth conditions~\eqref{loc_lip_p_Gamma_el}--\eqref{cont_growth_p_Gamma_el}. 
To prove that $U^\eps$ is well defined, we require that the interest rate $\lambda$ be large enough, condition \eqref{est_f_el_neu}, 
and that $h$ be uniformly bounded. Finally, concerning the parameters $m$ and $M$ and the function $\psi$ associated to the termination of the game, we assume that 
$\psi \in C^2 (\ol \Omega)$ fulfills $\frac{\partial \psi}{\partial n}=\norm{h}_{L^\infty}+1$ on~$\partial \Omega$, $m=M-1-2\norm{\psi}_{L^\infty(\ol \Omega)}$, 
$\chi=m+\norm{\psi}_{L^\infty(\ol \Omega)}+\psi$ 
 and $M$ is sufficiently large. 
These hypotheses ensure us the availability of the dynamic programming inequalities stated in Proposition~\ref{ineq_dyn_prog_el}.

\begin{prop}\label{cv_el}
 Suppose $f$, $g$, $\lambda$ and $\alpha$, $\beta$, $\gamma$, $m$, $M$, $\psi$ satisfy the hypotheses just listed (from which it follows that 
$\underline v$ and $\overline u$ are bounded away from $\pm M$ for all $x\in \overline \Omega$). 
Then $\overline u$ is a viscosity subsolution and $\underline v$ is a viscosity supersolution of \eqref{eq_neumann_el} in $ \overline{\Omega}$. More specifically:
\begin{itemize}
 \item if $x_0\in \Omega$          then each of $\overline u$ and $\underline v$ satisfies part \ref{def_visc_sub_int_el} of relevant 
half of Definition \ref{def_visco_el} at $x_0$, and 
 \item if $x_0\in \partial \Omega$ then each of $\overline u$ and $\underline v$ satisfies part \ref{def_visc_sub_bd_el} of relevant 
half of Definition \ref{def_visco_el} at $x_0$.
\end{itemize}
\end{prop}
\begin{proof} 
 When $x_0 \in \Omega$, the proof is similar to that of Theorem \ref{cv_par}. 
Therefore we shall focus on the case when $x_0\in \partial \Omega$.
We give the proof for $\bar u$, the arguments for $\udl v$ being similar and even easier due to fewer cases to distinguish. 
Consider a smooth function  $\phi$ such that $\overline u - \phi$ has local maximum on $\overline \Omega$ at $x_0 \in \partial \Omega$.
 We may assume that  $\langle D\phi(x_0), n(x_0) \rangle > h(x_0)$  since otherwise the assertion is trivial. 
Adjusting $\phi$ if necessary, we can assume that $\overline u (x_0)=\phi(x_0)$ and that 
the local maximum is strict, i.e. 
\begin{equation} \label{eq_sub_el_hyp_visc}
\overline u (x) <\phi(x) \quad  \text{ for } x\in \overline \Omega \cap \{0<\norm{x-x_0}\leq r \} ,
\end{equation}
for some $r>0$. By the definition of $\overline u$ given by \eqref{def_bp_subsup_visc_el}, 
there exist sequences $\eps_k>0$ and $\tilde y_k \in \ol \Omega$ such that
\begin{equation*}
 \tilde y_k  \rightarrow x_0, \quad u^{\eps_k}(\tilde y_k) \rightarrow \overline u (x_0).
\end{equation*}
We may choose $y_k \in \ol \Omega$ such that 
$\ds  (u^{\eps_k} - \phi) (y_k) \geq \sup_{\ol \Omega \cap \{ \norm{x-x_0} \leq r \}}  (u^{\eps_k} - \phi) (x) -\eps_k^3$.   
Evidently
\begin{equation*}
 (u^{\eps_k} - \phi) (y_k) \geq (u^{\eps_k} - \phi) (\tilde y_k)  - \eps_k^3 
\end{equation*}
and the right-hand side tends to 0 as $k \rightarrow \infty$. It follows using \eqref{eq_sub_el_hyp_visc} that 
\begin{equation*}
y_k \rightarrow x_0 \quad \text{ and } \quad  u^{\eps_k}(y_k) \rightarrow \bar u(x_0), 
\end{equation*}
as $k \rightarrow \infty$. Setting $\xi_k = (u^{\eps_k} - \phi)(y_k)$, we also have by construction that
\begin{equation} \label{min_veps_k_sub}
 \xi_k \rightarrow 0  \quad \text{ and } \quad  u^{\eps_k}(x)\leq \phi(x) +\xi_k - \eps^3_k \quad
\text{ whenever } x\in \ol \Omega \text{ with } \norm{x-x_0}\leq r.
\end{equation}
We now use the dynamic programming inequality \eqref{dyn_prog_ineq_super_new_el} at $y_k$, which can be written concisely as
\begin{equation*}
u^{\eps_k}(y_k) \leq \sup_{p, \Gamma} \inf_{\Delta \hat x} \left\{e^{-\lambda \eps_k^2} u^{\eps_k}(y_k+\Delta x) - \delta_k \right\}, 
\end{equation*}
with the convention
\begin{equation*}
\delta_k= p\cdot \Delta \hat x +\frac{1}{2} \langle \Gamma \Delta \hat x, \Delta \hat x\rangle +
\eps^2_k f(x, u^{\eps_k}(x), p, \Gamma) - \norm{\Delta \hat x - \Delta x} h(x+\Delta x).
\end{equation*}
By the rule \eqref{moving2_new} of the game, for every move $\Delta \hat x$ decided by Mark, the point $y_k+\Delta x$ belongs to $\ol  \Omega$.
Combining this observation with \eqref{min_veps_k_sub} and the definition of $\xi_k$  we conclude that 
\begin{equation*}
\phi(y_k)+\xi_k \leq \sup_{p, \Gamma} \inf_{\Delta \hat x} \left\{e^{-\lambda \eps^2_k} \left[ \phi(y_k+\Delta x) +\xi_k - \eps^3_k \right] - \delta_k \right\}. 
\end{equation*}
We may replace $e^{-\lambda \eps_k^2}$ by $1-\lambda \eps_k^2$ and   
$e^{-\lambda \eps_k^2} \xi_k$ by $\xi_k$ while making an error which is only $o(\eps^2)$ using the fact that $\xi_k \rightarrow 0$. 
Dropping $\xi_k$ from both sides, 
we conclude that
\begin{equation*}
\phi(y_k)\leq \sup_{p, \Gamma} \inf_{\Delta \hat x} \left( e^{-\lambda \eps_k^2} \phi(y_k+\Delta x)   - \delta_k \right) +o(\eps_k^2). 
\end{equation*}
We can evaluate the right-hand side using Proposition \ref{cons_new_sub_el} case \ref{cas_el_far} for $k$ large enough. 
We introduce $\rho$ and $\kappa$ defined in Section \ref{consistance_estimates_para} by \eqref{def_nul} and \eqref{relation_gamma_nul_tilde} and satisfying 
$0 < \kappa <1-\alpha < \rho <1  $.
If we may assume, up to a subsequence, that for all $k$ large enough, on the one hand $d(y_k)\geq \eps_k^{1-\alpha}$ or on the one hand
$\eps_k^{1-\alpha} - \eps_k^{\rho} \leq d(y_k)\leq \eps_k^{1-\alpha}$ 
 and $ M_{\eps_k}^{y_k}[\phi]  \leq \frac{4}{3} \norm{D^2\phi(y_k)}\eps_k^{1-\alpha}$, 
we can apply Proposition~\ref{cons_new_sub_el} case \ref{cas_el_far} to evaluate the right-hand side
\begin{equation*}
0\leq - \eps_k^2  f(y_k,  u^{\eps_k}(y_k), D\phi(y_k), D^2\phi(y_k)) - \eps_k^2\lambda u^{\eps_k}(y_k) +o(\eps_k^2). 
\end{equation*}
By passing to the limit $k\rightarrow +\infty$, we get the required inequality 
in the viscosity sense.  
 Otherwise, recall that $\langle D\phi(x_0), n(x_0) \rangle > h(x_0)$. By Lemma \ref{lemma_cv_m_M}, we have 
\begin{equation}\label{hyp_absurde}
 M_{\eps_k}^{y_k}[\phi]  \rightarrow  h(x_0) - \langle D\phi(x_0), n(x_0) \rangle < 0 , 
\end{equation}
and the condition 
\begin{equation} \label{cd_bord_cv_el_cont}
 M_{\eps_k}^{y_k}[\phi] \leq - \eps_k^{1-\alpha -\kappa} 
\end{equation}
 is satisfied for all $k$ sufficiently large. 
Therefore, up to a subsequence, it remains to consider a sequence $(y_k, \eps_k)_{k\in \N}$ satisfying both 
 $d(y_k) \leq \eps_k^{1-\alpha} - \eps_k^{\rho}$ and \eqref{cd_bord_cv_el_cont}. 
The last part of Proposition~\ref{estimate_el_psi_u} 
can be applied and we get by \eqref{est_phi_ell_cas_neg_glolocal} that there exists a constant 
$C$ depending only on $M$, $\norm{D\phi}_{C_b^1(\ol \Omega)\cap B(y_k,\eps_{k}^{1-\alpha})}$ and $\norm{h}_{L^\infty}$ such that
\begin{equation*}
0 \leq \frac{1}{4} \eps_k^{\rho}  M_{\eps_k}^{y_k}[\phi] +C \eps_k^2  - \lambda \eps_k^2 \phi(y_k) +o(\eps_k^2), 
\end{equation*}
recalling that  $\left( \eps^{1-\alpha} - d(y_k) \right)\geq \eps_k^{\rho}$ and $M_{\eps_k}^{y_k}[\phi] <0$. 
By dividing by $\eps_k^\rho$, it follows that 
\begin{equation*}
  - \eps_k^{2-\rho} \left(C  - \lambda \phi(y_k)\right)  +o(\eps_k^{2-\rho}) \leq \frac{1}{4}  M_{\eps_k}^{y_k}[\phi].
\end{equation*}
The sequence $(\phi(y_k))_{k\in \N}$ is bounded  by  smoothness of $\phi$. 
Since  $\norm{D\phi}_{C_b^1(\ol \Omega)\cap B(y_k,\eps_{k}^{1-\alpha})} \leq \norm{D\phi}_{C_b^1(\ol \Omega)\cap B(x_0,r)}$
holds for $k$ large enough,
we can assume that $C$ is independent of $k$ depending only on  $\norm{D\phi}_{C_b^1(\ol \Omega)\cap B(x_0,r)}$, $M$ and $\norm{h}_{L^\infty}$. 
Taking the limit as $k \rightarrow +\infty$, we deduce that 
\begin{equation*}
\liminf_{k \rightarrow \infty }M_{\eps_k}^{y_k}[\phi] \geq 0,   
\end{equation*}
which is a contradiction with \eqref{hyp_absurde}. Thus $\overline u $ is a viscosity subsolution at $x_0$.
\end{proof}

\section{Consistency}
\label{consistance}

A numerical scheme is said to be consistent if every smooth solution of the
PDE satisfies it modulo an error that tends to zero with the step size. 
It is the idea of the argument used in \cite{kohns}. In our case, we must understand how the Neumann condition interferes with the estimates proposed 
in \cite[Section 4]{kohns}. The essence of our formal argument in Section \ref{heuristic_derivation}
 was that the Neumann condition term is predominant compared to the PDE term at the boundary and produces a degeneracy in the consistency estimate. 
The present section clarifies the connection between our formal argument and the consistency of the game, by discussing consistency in more conventional terms.
The main point is presented in Propositions \ref{cons_lower_bound} and \ref{cons_new_sub}. 
In order to explain very precisely how the consistency estimate of \cite[Section 4]{kohns} degenerates, we establish the consistency of
our game as a numerical scheme by focusing on different cases according to the values of the quantities $m^{x}_\eps[\phi]$ and $M^{x}_\eps[\phi]$ 
defined by \eqref{m_eps_par}--\eqref{M_eps_par} and the distance $d(x)$ to the boundary~$\partial \Omega$.

\subsection{The parabolic case}
Consider the game discussed in Section \ref{def_jeu_parabolique} for solving $-u_t +f(t,x,u,Du,D^2u)=0$
in $ \Omega$ with final-time data $u(x,T)=g(x)$  for $x\in \overline \Omega$ and 
boundary condition $\frac{ \partial u}{\partial n}(x,t)=h(x)  $ for $x\in \partial \Omega, t\in (0,T)$.  
The dynamic programming principles \eqref{dyn_prog_ineq_sub_new}--\eqref{dyn_prog_ineq_super_new} can be written as
\begin{align*}
u^{\eps}(x,t) & \leq S_{\eps}  \left[x,t,u^{\eps}(x,t), u^{\eps}(\cdot ,t+\eps^2) \right], \\
v^{\eps}(x,t) & \geq S_{\eps}  \left[x,t,v^{\eps}(x,t), v^{\eps}(\cdot, t+\eps^2) \right], 
\end{align*}
where $S_{\eps}  \left[x,t,z, \phi \right]$ is defined for any $x\in \overline \Omega$, $z\in \R$ and $t\leq T$ and any 
continuous function \mbox{$\phi$: $\ol \Omega \rightarrow \R$} by
\begin{multline}\label{def_new_op}
S_{\eps} \left[x,t,z, \phi \right] = \max_{p,\Gamma} \min_{\Delta \hat x} \left[\phi \left( x +\Delta x \right) \right.  \\
\left. - \left( p \cdot \Delta \hat{x} +\frac{1}{2} \left\langle \Gamma \Delta \hat x, \Delta \hat x \right\rangle 
  +\eps^2 f \left(t,x,z,p,\Gamma\right) - \norm{\Delta \hat x -\Delta x} h(x+\Delta x) \right) \right], 
\end{multline}
subject to the usual constraints $\norm{p}\leq \eps^{-\beta}$, $\norm{\Gamma}\leq \eps^{-\gamma}$, $\norm{\Delta \hat x} \leq \eps^{1-\alpha}$ and 
 $\Delta x=\proj_{\ol \Omega} (x+\Delta \hat x) - x$.  The operator $S_\eps$ clearly satisfies the three following properties: 
\begin{itemize}
 \item For all  $\phi \in C(\overline \Omega)$, $S_{0}  \left[x,t,z, \phi \right]=\phi(x)$.
 \item $S_{\eps}$ is monotone, i.e. if $\phi_1 \leq \phi_2$, then 
$S_{\eps}  \left[x,t,z,  \phi_1 \right]\leq S_{\eps}  \left[x,t,z, \phi_2 \right]$.
 \item For all  $\phi \in C(\overline \Omega)$ and $c\in \R$, 
\begin{equation}\label{action_on_cte}
S_{\eps}  \left[x,t,z, c + \phi \right]=c + S_{\eps}\left[x,t,z, \phi \right] .
\end{equation}
\end{itemize}
Fixing $x,t,z$ and a smooth function $\phi$, a Taylor expansion shows that for any $\norm{\Delta x}\leq \eps^{1-\alpha}$,  
\begin{equation*}
\phi(x+\Delta x)=\phi(x)+ D\phi(x)\cdot \Delta x  +\frac{1}{2} \langle D^2\phi(x) \Delta x, \Delta x \rangle +O(\eps^{3-3 \alpha}). 
\end{equation*}
Since $\alpha<1/3$ by hypothesis, $\eps^{3-3\alpha}=o(\eps^{2})$. By rearranging the terms, we compute 
\begin{multline*}
\phi  \left( x +\Delta x \right)  - \left( p \cdot  \Delta \hat{x} 
+\frac{1}{2} \left\langle \Gamma \Delta \hat x, \Delta \hat x \right\rangle   +\eps^2 f \left(t,x, z,  p,\Gamma \right)  
- \norm{\Delta \hat x - \Delta x}  h(x+\Delta x ) \right) \\
  =  \phi(x) + (D\phi(x) - p)\cdot \Delta \hat x + \norm{\Delta \hat x - \Delta x} \left[ h(x+\Delta x ) - D\phi(x) \cdot n(x+\Delta x) \right]   \\
   +\frac{1}{2} \left\langle D^2\phi(x)  \Delta x, \Delta x \right\rangle 
 - \frac{1}{2} \left\langle \Gamma \Delta \hat x, \Delta \hat  x \right\rangle   - \eps^2 f(t,x, z,  p ,\Gamma) +o(\eps^2), 
\end{multline*}
since the outward normal  can be expressed by 
$\ds n(x+\Delta x)= -   \frac{\Delta x - \Delta \hat x}{\norm{\Delta \hat x - \Delta x}}$ if $x+\Delta \hat x \notin \Omega$ 
and the move $\Delta x$ can be decomposed as $\Delta x=\Delta \hat x + ( \Delta x - \Delta \hat x )$. 
Thus, we shall examine  
\begin{multline}\label{operator_fonction_C2}
S_{\eps}  \left[x,t,z, \phi \right] - \phi(x)=
 \max_{p,\Gamma   }   \min_{\Delta \hat x} 
\left[  (D\phi(x) - p)\cdot \Delta \hat x  +\frac{1}{2} \left\langle D^2\phi(x)  \Delta x, \Delta x \right\rangle  \right. \\
\left. + \norm{\Delta \hat x - \Delta x} \{ h(x+\Delta x ) - D\phi(x) \cdot n(x+\Delta x) \} 
- \frac{1}{2} \left\langle \Gamma \Delta \hat x, \Delta \hat  x \right\rangle 
  - \eps^2 f( t,x, z,  p,\Gamma) 
\right] +o(\eps^2).   
\end{multline}
\subsubsection{Preliminary geometric lemmas}
\label{prel_geo_lem}
This subsection is devoted to some geometric properties of the game that will be useful to show consistency in Section \ref{consistance_estimates_para}. 
We start by some estimates, involving the geometric conditions on the domain, about the moves $\Delta \hat x$ decided by Mark. 
\begin{lemma}
 Suppose that $\Omega$ is a $C^2$-domain satisfying the uniform exterior ball condition for a certain $r>0$. 
Then, for all $0<\eps<r^{\frac{1}{1-\alpha}}$ 
and for all $\Delta \hat x$ constrained by \eqref{moving1_new}, determining $\Delta x$  by \eqref{exp_delta_x}, we have
\begin{equation}\label{eq_11}
\norm{\Delta \hat x - \Delta x}\leq \eps^{1-\alpha} - d(x) \quad \text{ and }  \quad    \norm{\Delta x}\leq 2 \eps^{1-\alpha} - d(x).
\end{equation}
\end{lemma}

\begin{proof} Let us prove the first inequality, the second following immediately by the triangle inequality.
If the point $\hat x= x+\Delta \hat x$ belongs to $ \overline \Omega$, $\Delta x= \Delta \hat x$ and the result is obvious.
 Supposing now $\hat x$ does not belong to $ \overline\Omega$, the set $S=[x,\hat x] \cap \partial \Omega$ is not empty and we can consider 
a point $x_{\partial} \in S$.  
By the rule of the game, we have $\norm{ x-\hat x} = \norm{\Delta \hat x}\leq \eps^{1-\alpha}$. Since  $x_{\partial}\in \partial \Omega$ by construction, it is clear 
that $\norm{x  - x_\partial}\geq d(x)$. We deduce that 
\begin{equation*}
\norm{x_\partial -  \hat x } = \norm{x  - \hat x } - \norm{x_\partial  - x } \leq \eps^{1-\alpha} - d(x) .
\end{equation*}
By the uniform exterior ball condition, the orthogonal projection on $\ol \Omega$ is well-defined on $\Omega(\eps^{1-\alpha}) \subset \Omega(r)$. 
By property of the orthogonal projection and since $\hat x\notin \overline\Omega$, we can write
\begin{equation*}
\norm{\Delta \hat x - \Delta x}	=\inf_{y\in \overline \Omega} \norm{y-\hat x} 
				=\inf_{y\in \partial  \Omega} \norm{y-\hat x} \leq \norm{x_\partial  - \hat x } , 
\end{equation*}
which gives directly the first estimate in \eqref{eq_11}.
\end{proof}

The following lemma uses the crucial geometric fact that $\Omega$ satisfies the interior ball condition introduced in Definition \ref{unif_int_ball_cd} 
for which there is no neck pitching for $\eps$ sufficiently small.
\begin{lemma}\label{boundary_bounce1}
Let $\sigma > 1-\alpha$ and $B>0$.
 Suppose that $\Omega$ is a domain with $C^2$-boundary $\partial \Omega$ and satisfies the uniform interior ball condition. 
Then, for all possible moves $\norm{\Delta \hat x}\leq \eps^{1-\alpha}$ such that $\norm{\Delta \hat x + \eps^{1-\alpha} n(\bar x)} \leq B\eps^\sigma$
the intermediate point $\hat x$ belongs to $\Omega$ for all $\eps$ sufficiently small.
Moreover, for all possible moves $\norm{\Delta \hat x}\leq \eps^{1-\alpha}$ such that $\norm{\Delta \hat x - \eps^{1-\alpha} n(\bar x)} \leq B\eps^\sigma$
and $\Delta x$ determined by \eqref{exp_delta_x}, we have
\begin{equation}\label{rebond_dist_boundary_geo}
\norm{\Delta \hat x - \Delta x}\geq \eps^{1-\alpha} - d(x)- B\eps^\sigma+  O(\eps^{2-2\alpha}) . 
\end{equation}
Furthermore, if in addition we assume  $d(x)\geq \eps^{1-\alpha} - \eps^\eta$ with $1-\alpha< \eta<\sigma$, the intermediate point $\hat x$ is outside $\Omega$ 
 for all $\eps$ sufficiently small. 
\end{lemma}
\begin{proof}   
 For the first assertion, since $\Omega$ satisfies the uniform interior ball condition (there is no neck pitching for $\eps$ sufficiently small), 
we observe that the set $\partial \Omega\cap B(x,2\eps^{1-\alpha})$ 
is below a paraboloid $P_1$ of opening $A$ and above a paraboloid $P_2$ of opening $ - A$ touching $\partial \Omega$ at $\bar x$. 
By the Taylor expansion, if $T_{\bar{x}} \partial \Omega$ denotes the tangent space to $\partial \Omega$ at $\bar x$, 
we get that for all $y\in \partial \Omega \cap B(x, 2\eps^{1-\alpha})$, 
\begin{equation*} 
|(y-\bar x) \cdot n(\bar x)|=d(y, T_{\bar{x}} \partial \Omega) \leq \frac{1}{2}A (2\eps^{1-\alpha})^2, 
\end{equation*}
Since $(x+\Delta \hat x-\bar x) \cdot n(\bar x) \leq  - \eps^{1-\alpha} - d(x) + B\eps^{\sigma}$, 
we deduce that for all $\eps$ sufficiently small, 
\begin{equation*}
(x+\Delta \hat x-\bar x) \cdot n(\bar x)< \inf_{y\in \partial \Omega \cap B(x, 2\eps^{1-\alpha})}  (y-\bar x) \cdot n(\bar x), 
\end{equation*}
which yields that $x+\Delta \hat x$ belongs to $\Omega$. 

For the second claim, we denote by $(\kappa_i(x))_{1\leq i\leq N-1}$ the principal curvatures at $x$ on $\partial \Omega$ 
 and by $(e_1, \cdots, e_N)$ an  orthonormal frame centered in $\bar x$ with first vector $e_1=n(\bar x)$. 
Since $\Omega$ is a $C^2$-domain, $(e_2, \cdots, e_N)$ form a basis of the tangent space $T_{\bar{x}} \partial \Omega$. We compute
\begin{equation*}
\eps^{1-\alpha} - B\eps^{\sigma} \leq \Delta \hat x \cdot n(\bar x) = (\Delta \hat x -\eps^{1-\alpha}n(\bar x)) \cdot n(\bar x) + \eps^{1-\alpha}.
\end{equation*}
Thus $\hat x$ is contained in the half-space $H_1$ determined by $(y-\bar x)\cdot e_1 \geq \eps^{1-\alpha} - d(x) - B\eps^{\sigma}$ and 
$d(\hat x , T_{\overline x} \partial \Omega)  \geq \eps^{1-\alpha} - d(x) - B\eps^{\sigma}$.  
Moreover, we deduce from \eqref{eq_11} and the triangle inequality that for each move $\Delta \hat x$ we have
 $x+\Delta x \in B(\bar x, 2\eps^{1-\alpha})$. Assume $x_1= p(x_2,\cdots, x_N)$ is a local $C^2$-parametrization of $\partial \Omega$ around $x$.
  By a Taylor argument and by continuity of the principal curvatures on $\partial \Omega$, it follows that, 
for $\eps>0$ small enough,  
\begin{equation}\label{taylor_bord}
d(x+\Delta x, T_{\bar{x}} \partial \Omega) \leq \frac{1}{2}C_1 (2\eps^{1-\alpha})^2  =2 C_1 \eps^{2-2\alpha} , 
\end{equation}
 where  $\ds C_1:=2 \max \left\{| \kappa_i(\overline x)| : 1\leq i \leq N-1  \right\}$. 
By the triangle inequality, we deduce that
\begin{align*}
 \norm{x+\Delta x-\hat x} 
 & \geq \norm{ \proj_{T_{\bar{x}} \partial \Omega}( x+\Delta x)  -\hat x} -  \norm{x+\Delta x - \proj_{T_{\bar{x}} \partial \Omega}( x+\Delta x) }   \\
 & \geq d(\hat x , T_{\overline x} \partial \Omega)  -  d(x+\Delta x, T_{\overline x} \partial \Omega)  \\
 & \geq \eps^{1-\alpha} - d(x) - B\eps^{\sigma}  - 2 C_1 \eps^{2-2\alpha}.
\end{align*}
In particular, if $d(x)\geq \eps^{1-\alpha} - \eps^\eta$ with $1-\alpha< \eta<\sigma$ 
 the right-hand side is strictly positive for $\eps$ sufficiently small and $\hat x\notin \Omega$. 
\end{proof}

The next lemmas gather some estimates which will be useful to establish our consistency estimates.

\begin{lemma}\label{lemma_key_bound_geo} Under the hypothesis of Lemma \ref{boundary_bounce1}, 
for all moves $\Delta \hat x$ constrained by \eqref{moving1_new}, determining $\Delta x$ by \eqref{exp_delta_x}, we have
\begin{equation}\label{ineq_key}
-\frac{1}{2} (\eps^{1-\alpha} - d(x ))\leq  - \frac{1}{2} \left( 1-\frac{ d(x)}{\eps^{1-\alpha}}\right) (\Delta \hat x)\cdot n(\bar x)
    + \norm{\Delta \hat x - \Delta x} \leq \frac{3}{2}(\eps^{1-\alpha} - d(x)).
\end{equation}
\end{lemma}
\begin{proof} 
The left-hand side of \eqref{ineq_key} can be written in the form
\begin{equation*}  
 - \frac{1}{2} \left( 1-\frac{ d(x)}{\eps^{1-\alpha}}\right) (\Delta \hat x)\cdot n(\bar x)+ \norm{\Delta \hat x - \Delta x}
 =(\eps^{1-\alpha} - d(x)) \left[-\frac{1}{2}\frac{(\Delta \hat x)\cdot n(\bar x)}{\eps^{1-\alpha}} 
 + \frac{\norm{\Delta \hat x - \Delta x}}{\eps^{1-\alpha} - d(x)}\right], 
\end{equation*} 
which directly gives the desired estimates by using \eqref{moving1_new} and the first inequality given by \eqref{eq_11}.
\end{proof}

\begin{lemma} \label{lem_sign_min_BC_pos}
Let $A\in \mathcal{M}^N(\R)$, $k\in C_b(\partial \Omega)$ extended by some function  $k: \ol \Omega \rightarrow \R$, and  $x\in \overline \Omega$. 
Suppose in addition that
\begin{equation}\label{hyp_m_big}
(3\eps^{1-\alpha}-d(x)) \norm{A} \leq \inf_{\substack{x+\Delta \hat x \notin \Omega \\ \Delta \hat x}} k(x+\Delta x), 
\end{equation}
with $\Delta \hat x$ constrained by \eqref{moving1_new} and $\Delta x$ determined by \eqref{exp_delta_x}. Then 
\begin{equation} \label{sign_min_BC_pos}
 \min_{\Delta \hat x} \left\{ \left\langle A  \Delta x, \Delta x \right\rangle  - \left\langle A  \Delta \hat x, \Delta \hat  x \right\rangle 
+ \norm{\Delta \hat x - \Delta x}  k(x+\Delta x) \right\} = 0, 
\end{equation}
where $\Delta \hat x$ is constrained by \eqref{moving1_new} and determines $\Delta x$  by \eqref{exp_delta_x}.
\end{lemma}
\begin{proof} 
If $\hat x =x+\Delta \hat x \in \overline \Omega$, the function is equal to zero. 
We now consider the moves for which $\hat x \notin \overline \Omega$. Then 
\begin{equation}  \label{decoup_geo1} 
 \left\langle A \Delta x, \Delta x \right\rangle - \left\langle A  \Delta \hat x, \Delta \hat x \right\rangle 
= \left\langle A (\Delta \hat x - \Delta x), \Delta \hat x - \Delta x \right\rangle + 2 \left\langle A \Delta \hat x, \Delta \hat x - \Delta x \right\rangle .    
\end{equation}
By the Cauchy-Schwarz inequality, we obtain
\begin{equation}
   | \left\langle A  \Delta x, \Delta x \right\rangle  - \left\langle A  \Delta \hat x , \Delta \hat x \right\rangle  | 
 \leq \norm{A} \norm{ \Delta \hat x  - \Delta x  } \left(  \norm{ \Delta \hat x  - \Delta x  } + 2 \norm{ \Delta \hat x } \right). \label{diff_x_hat_x}
\end{equation}
By using \eqref{eq_11} and $\norm{\Delta \hat x}\leq \eps^{1-\alpha}$, we get
\begin{equation} \label{identite_quadrilatere}
  | \left\langle A  \Delta x, \Delta x \right\rangle  - \left\langle A  \Delta \hat x , \Delta \hat x \right\rangle  |
 \leq \norm{A}  \norm{ \Delta \hat x  - \Delta x  } \left( 3 \eps^{1-\alpha} - d(x) \right).   
\end{equation}
Thus
\begin{equation*}
 \left\langle A  \Delta x, \Delta x \right\rangle  - \left\langle A  \Delta \hat x, \Delta \hat  x \right\rangle 
+ \norm{\Delta \hat x - \Delta x}  k(x+\Delta x)   
  \geq \norm{ \Delta \hat x  - \Delta x  }
 \left\{ \inf_{\substack{x+\Delta \hat x \notin \Omega \\ \Delta \hat x}} k(x+\Delta x) -\norm{A} (3\eps^{1-\alpha} - d(x))  \right\}. 
\end{equation*}
The right-hand side of this last inequality is strictly positive by the assumption~\eqref{hyp_m_big}. 
\end{proof}

 \subsubsection{Consistency estimates}
\label{consistance_estimates_para}
In this subsection we state our consistency estimates. They explain precisely the conditions under which the usual estimate proposed in \cite{kohns} 
holds for $x$ near the boundary and $\phi \in C^2 (\ol \Omega)$. If it does not hold, there is a degeneration of the estimates 
respecting the final discussion of formal derivation of the PDE at Section \ref{heuristic_derivation}.
For fixed $x\in \Omega(\eps^{1-\alpha})$, these estimates take into account the size and the sign of the boundary condition in
the small ball $B(x, \eps^{1-\alpha})$ and the distance $d(x)$ to the boundary.

In the heuristic derivation presented in Section  \ref{heuristic_derivation}, we assumed that 
 $\Delta \hat x\mapsto h(x+\Delta x) - D\phi(x) \cdot n(x+\Delta x)$, with $\Delta x$ determined by \eqref{exp_delta_x},
 was locally constant in a $\delta$-neighborhood of the boundary near $x$.
In the general case, this hypothesis must be relaxed. To do this, we observe that, for all $\Delta \hat x$ constrained by~\eqref{moving1_new} 
satisfying $x+\Delta \hat x \notin \Omega$ and determining $\Delta x$ by \eqref{exp_delta_x},
\begin{equation*}
m_\eps^x[\phi] \leq  h(x+\Delta x) - D\phi(x) \cdot n(x+\Delta x) \leq M_\eps^x[\phi], 
\end{equation*}
where $m_\eps^{x}[\phi]$ and $M_\eps^{x}[\phi]$ are defined by  \eqref{m_eps_par}--\eqref{M_eps_par}.
 Therefore we are going to specify some strategies for Helen which are associated to the two extreme situations $m_\eps^{x}[\phi]$ and $M_\eps^{x}[\phi]$
by following the optimal choices \eqref{formal_choice_popt} and \eqref{formal_gamma_opt_n} 
obtained in the formal derivation at Section \ref{heuristic_derivation}. 
More precisely, for all $x\in \Omega(\eps^{1-\alpha})$, 
we define the strategies $p_\text{opt}^m(x)$, $p_\text{opt}^M(x)$ and $\Gamma_\text{opt}(x)$  in an orthonormal 
basis $\mathcal{B}=(e_1=n(\bar x), e_2,\cdots, e_N)$ respectively by
\begin{align}
p_\text{opt}^m(x) &   = D\phi(x)  +\left[\frac{1}{2}\left(1-\frac{ d(x)}{\eps^{1-\alpha}}\right) m_\eps^{x}[\phi] 
 - \frac{\eps^{1-\alpha}}{4} \left(1 - \frac{d^2(x)}{\eps^{2-2\alpha}} \right) (D^2\phi(x))_{11} \right]  n(\bar{x}),     \label{p_opt_m}  \\
p_\text{opt}^M (x)  & = D\phi(x)  +\left[\frac{1}{2}\left(1-\frac{ d(x)}{\eps^{1-\alpha}}\right) M_\eps^{x}[\phi]
- \frac{\eps^{1-\alpha}}{4} \left( 1 - \frac{d^2(x)}{\eps^{2- 2\alpha}} \right) (D^2\phi(x))_{11} \right] n(\bar{x}),      \label{p_opt_M}
\end{align}
and
\begin{equation}\label{Gamma_opt}
 \Gamma_\text{opt}(x) = D^2\phi(x) +\left[ \frac{1}{2}\left( - 1+\frac{d^2(x)}{\eps^{2-2\alpha}} \right) (D^2\phi(x))_{11}  \right] E_{11}, 
\end{equation}
where $E_{11}$ denotes the unit-matrix $(1,1)$ in the basis $\mathcal{B}$. 
These strategies depend on the local behavior of $\phi$ (size and amplitude) around the boundary and on the geometry of the boundary itself.

Since there is a degeneration of the usual estimates, there is no hope for one simple estimate. We are going to separate the study in two steps: 
Proposition~\ref{cons_lower_bound} provides the estimates for the lower bound and Proposition~\ref{cons_new_sub} deals with the upper bound.  
Moreover, Section \ref{tech_proof} is devoted to the technical proof of the upper bound distinguishing several cases 
according to the size of $M_\eps^{x}[\phi]$ and $d(x)$.

\begin{prop} \label{cons_lower_bound}   
Let $f$ satisfy \eqref{ellipticity_f} and \eqref{loc_lip_p_Gamma}--\eqref{cont_growth_p_Gamma} and assume $\alpha$, $\beta$, $\gamma$
 satisfy \eqref{condition_pas}--\eqref{cd_coeff_classiq}. 
Let $p_\text{opt}^m$ and $\Gamma_\text{opt}$ be respectively defined in the orthonormal basis $(e_1=n(\bar x), e_2,\cdots, e_N)$ by 
\eqref{p_opt_m} and \eqref{Gamma_opt}. 
For any $x$, $t$, $z$ and any smooth function $\phi$ defined near $x$, $S_\eps[x,t,z,\phi]$ being defined by \eqref{def_new_op},
 we distinguish two cases:
\begin{enumerate}[label=\textup{ }{\roman*.}\textup{ },ref=({\roman*})]
 \item \label{case1_lower} Big bonus: if $ d(x)\geq  \eps^{1-\alpha}$ 
or $m_\eps^{x}[\phi] >\frac{1}{2} (3\eps^{1-\alpha}-d(x)) \norm{D^2\phi(x)}$, then 
\begin{equation*}
  S_\eps[x,t,z, \phi] - \phi(x) \geq - \eps^2 f(t,x,z, D\phi(x), D^2\phi(x)). 
 \end{equation*}
\item \label{case2_lower} Penalty or small bonus: if $d(x)\leq  \eps^{1-\alpha}$ 
and $ m_\eps^{x}[\phi] \leq \frac{1}{2} (3\eps^{1-\alpha}-d(x)) \norm{D^2\phi(x)}$, then 
\begin{equation*}
 S_\eps[x,t,z,\phi] - \phi(x) \geq \frac{1}{2} (\eps^{1-\alpha} - d(x)) \left(s  m_\eps^{x}[\phi]- 4 \norm{D^2 \phi(x)}\eps^{1-\alpha} \right) 
                                    - \eps^2 f(t,x,z, p_{\text{opt}}^m(x), \Gamma_\text{opt}(x)) , 
 \end{equation*}
where $s=-1$ if $m_\eps^x[\phi]\geq 0$ and $s=3$ if $m_\eps^x[\phi]<0$.
 \end{enumerate}
\end{prop}
\begin{proof} 
 If $d(x) \geq \eps^{1-\alpha}$, the usual estimate \cite[Lemma~4.1]{kohns} holds.  
We now focus on the case  $d(x) \leq \eps^{1-\alpha}$. 
By the definition of $m_\eps^{x}[\phi]$ given by \eqref{m_eps_par} and the positivity of $\norm{\Delta \hat x - \Delta x}$, 
for all $\norm{\Delta \hat x} \leq \eps^{1-\alpha}$, we have
\begin{equation}\label{min_fond_m}
 \norm{\Delta \hat x - \Delta x}\left\{  h(x+\Delta x) - D\phi(x) \cdot n(x+\Delta x) \right\} \geq \norm{\Delta \hat x - \Delta x} m_\eps^{x}[\phi].
\end{equation}
Therefore it is sufficient to find a lower bound for 
\begin{equation*}
 \max_{p, \Gamma} \min_{\Delta \hat x} \left[(D\phi(x) - p)\cdot \Delta \hat x +m_\eps^{x}[\phi] \norm{\Delta \hat x - \Delta x}
 + \frac{1}{2} \langle D^2\phi(x)  \Delta x ,\Delta x\rangle - \frac{1}{2} \langle \Gamma \Delta \hat x,\Delta \hat x\rangle    \\
    - \eps^2 f(t,x,z,p,\Gamma) \right].
\end{equation*}
where $p$, $\Gamma$ and $\Delta \hat x$ are constrained by \eqref{p_beta_gamma_new}--\eqref{moving1_new} and $\Delta x$ determined by \eqref{exp_delta_x}. 
In other words, by taking advantage of the monotonicity of the operator $S_\eps$ with \eqref{min_fond_m},  we shall  
look for a lower bound for an approximated operator bounding $S_\eps$ from below and very close to it when $\eps \rightarrow 0$. 

Then, we also observe that for every choice $p$ and $\Gamma$,  
\begin{multline*}
 S_\eps[x,t, z, \phi]  - \phi(x)  \geq    - \eps^2 f \left(  t,x, z,  p,\Gamma \right)\\
   + \min_{\Delta \hat x} \left[(D\phi(x) - p)\cdot \Delta \hat x +
\frac{1}{2} \left\langle D^2\phi(x)  \Delta x, \Delta x \right\rangle  - \frac{1}{2} \left\langle \Gamma  \Delta \hat x, \Delta \hat  x \right\rangle 
+ \norm{\Delta \hat x - \Delta x} m_\eps^{x}[\phi] \right].
\end{multline*}
We now distinguish two particular strategies for Helen. For part \ref{case1_lower}, we consider the particular choice $p=D\phi(x)$, $\Gamma=D^2\phi(x)$ and obtain
\begin{align*}
 S_\eps[x,t, z, \phi] - \phi(x) & \geq  - \eps^2 f \left(t,x, z,  D\phi(x),D^2\phi(x) \right)  \\
  & \phantom{\geq}  + \min_{\Delta \hat x}
\left[ \frac{1}{2} ( \left\langle D^2\phi(x)\Delta x, \Delta x \right\rangle - \frac{1}{2}
\left\langle D^2\phi(x)\Delta \hat x, \Delta \hat x \right\rangle                 
+ \norm{\Delta \hat x - \Delta x}m_\eps^{x}[\phi]    
 \right]  \\
& \geq  - \eps^2 f \left(t,x, z,  D\phi(x),D^2\phi(x)  \right),
\end{align*}
by applying Lemma \ref{lem_sign_min_BC_pos} with $A=\frac{1}{2}D^2\phi(x)$. 
For part \ref{case2_lower},  we consider the choice $p=p_{\text{opt}}^m(x)$, $\Gamma= \Gamma_{\text{opt}}(x)$ and find 
\begin{align*}
 S_\eps[x,t,z, \phi] &- \phi(x) \geq  - \eps^2 f(t,x, z,  p_{\text{opt}}^m(x), \Gamma_{\text{opt}}(x)) +l^{x}[\phi], 
\end{align*}
with $l^{x}[\phi]$ defined by  
\begin{equation}\label{def_m_super}
l^{x} [\phi] = \min_{\Delta \hat x}  \left[ 
 (D\phi(x) - p_{\text{opt}}^m) \cdot  (\Delta \hat x)  
+ \frac{1}{2} \left\langle D^2\phi(x)  \Delta x, \Delta x \right\rangle  - \frac{1}{2} \left\langle \Gamma_{\text{opt}}(x) \Delta \hat x, \Delta \hat  x \right\rangle 
 + \norm{\Delta \hat x - \Delta x} m_\eps^{x}[\phi]  \right].   
\end{equation} 
It now remains to give a lower bound for $l^{x} [\phi]$. By plugging the expression \eqref{p_opt_m} of $p_{\text{opt}}^m(x)$  in \eqref{def_m_super}, we have
\begin{multline*}
l^{x} [\phi] 
 = \min_{\Delta \hat  x}  \left[  \left( -  \frac{\eps^{1-\alpha} - d(x) }{2\eps^{1-\alpha}} 
(\Delta \hat x)_1  + \norm{\Delta \hat x - \Delta x} \right) m_\eps^{x}[\phi]  \right.  \\
\left.
+ \frac{1}{2} \left\langle D^2\phi(x)  \Delta x, \Delta x \right\rangle  - \frac{1}{2} \left\langle \Gamma_\text{opt}(x) \Delta \hat x, \Delta \hat  x \right\rangle  
+  \frac{1}{4} \left( \eps^{1-\alpha}- \frac{d^2(x)}{\eps^{1-\alpha}} \right) (D^2\phi(x))_{11}  (\Delta \hat x)_1  \right].
\end{multline*}
It is clear that $ l^{x}[\phi] \geq l_1^{x}[\phi]+ l_2^{x}[\phi]$ with 
$l_1^{x}[\phi]$ and $l_2^{x}[\phi]$  respectively defined by
\begin{equation} \label{def_l1}
 l_1^{x}[\phi] := \min_{\Delta \hat x}  \left[
\left( -  \frac{\eps^{1-\alpha} - d(x) }{2\eps^{1-\alpha}} (\Delta \hat x)_1  + \norm{\Delta \hat x - \Delta x} \right) m_\eps^{x}[\phi]
\right],
\end{equation}
and 
\begin{equation}\label{def_l2}
l_2^{x}[\phi] : =   \frac{1}{2}
 \min_{\Delta \hat x}  \left[  \left\langle D^2\phi(x)  \Delta x, \Delta x \right\rangle 
 - \left\langle \Gamma_\text{opt}(x) \Delta \hat x, \Delta \hat  x \right\rangle    
+ \frac{\eps^{1-\alpha}}{2} \left( 1- \frac{d^2(x)}{\eps^{2- 2 \alpha}} \right) (D^2\phi(x))_{11}  (\Delta \hat x)_1 \right].  
\end{equation}
By using Lemmas \ref{lm_estimate_l1} and \ref{lm_estimate_l2} stated below,  giving lower bounds respectively for $l_1^{x}[\phi]$ 
and $l_2^{x}[\phi]$, one obtains
\begin{align*}
l^{x}[\phi]
& \geq \frac{s}{2}(\eps^{1-\alpha} - d(x)) m_\eps^{x}[\phi]- 2 \norm{D^2\phi(x) } \eps^{1-\alpha} \left( \eps^{1-\alpha}-d(x) \right)  
= \frac{1}{2}(\eps^{1-\alpha} - d(x)) \left( s m_\eps^{x}[\phi] - 4 \norm{D^2\phi(x) }  \eps^{1-\alpha} \right), 
\end{align*}
which gives the desired estimate. 
\end{proof}

The three following lemmas provide the required estimates for $l_1^{x}[\phi]$ and $l_2^{x}[\phi]$.
\begin{lemma} \label{lm_estimate_l1}   
 For any $x \in \Omega(\eps^{1-\alpha})$ and any function $\phi$ defined at $x$, $l_1^{x}[\phi]$ being defined by \eqref{def_l1}, we have
\begin{equation*}
 \frac{s}{2}(\eps^{1-\alpha} - d(x))  m_\eps^{x}[\phi] \leq l_1^{x}[\phi]  \leq 0, 
\end{equation*}
with $s= -1$ if  $m_\eps^{x}[\phi]$ is positive and $s= 3$ if  $m_\eps^{x}[\phi]$ is nonpositive.
\end{lemma}
\begin{proof} By considering $\Delta \hat x = 0$, $ l_1^{x}[\phi]$ is negative. 
To find a lower bound on $ l_1^{x}[\phi]$, if $m_\eps^{x}[\phi]$ is negative, we may write  
\begin{equation*}
  \left[ -  \frac{\eps^{1-\alpha} - d(x) }{2 \eps^{1-\alpha}}  (\Delta \hat x)_1
+ \norm{\Delta \hat x - \Delta x} \right] m_\eps^{x}[\phi]  \geq \frac{3}{2}(\eps^{1-\alpha} - d(x)) m_\eps^{x}[\phi],    
\end{equation*}
the last inequality being provided by the right-hand side inequality given in Lemma \ref{lemma_key_bound_geo}
since by hypothesis $m_\eps^{x}[\phi]$ is negative. If $m_\eps^{x}[\phi]$ is nonnegative, 
the result follows from applying the left-hand side inequality given in Lemma \ref{lemma_key_bound_geo}.
\end{proof}

\begin{lemma}\label{est2_phi_gopt}   
 Let $x\in  \Omega(\eps^{1-\alpha})$ and $\phi \in C^2(\overline \Omega)$. For all  $\Delta \hat x$ constrained by \eqref{moving1_new}, we have 
\begin{equation} \label{est_D2phi_x_hatx}
  \left| \frac{1}{2} \langle D^2\phi(x)  \Delta x, \Delta x \rangle - \frac{1}{2} \langle D^2\phi(x) \Delta \hat x ,\Delta \hat x\rangle  \right|
 \leq \frac{1}{2}\norm{D^2\phi(x)}    \left(3 \eps^{1-\alpha}-d(x)\right) \norm{\Delta \hat x -\Delta x}  , 
\end{equation} 
and
\begin{equation}\label{cp_D2phi_gamma_opt_0}
 \left| \frac{1}{2} \langle D^2\phi(x) \Delta x,\Delta x \rangle - \frac{1}{2} \langle \Gamma_\text{opt}(x) \Delta \hat x ,\Delta \hat x\rangle  \right|
 \leq \frac{1}{4}\norm{D^2\phi(x)} \left(\eps^{1-\alpha} - d(x)\right) \left(7 \eps^{1-\alpha} -  d(x) \right), 
\end{equation} 
 where $\Gamma_\text{opt}(x)$ is the optimal choice defined by \eqref{Gamma_opt} in an orthonormal basis $\mathcal{B}=(e_1=n(\bar x), \cdots, e_N)$.
\end{lemma}
\begin{proof} The first inequality is an immediate consequence of \eqref{diff_x_hat_x}. For the second inequality,
  all the coordinates $ \langle (D^2\phi(x)-\Gamma_\text{opt}(x)) e_i,e_j \rangle $ in the basis $\mathcal{B}$
 are equal to zero, except for $i=j=1$. By using the vector decomposition given by \eqref{decoup_geo1}, we have
\begin{multline*}
 \frac{1}{2} \langle D^2\phi(x) \Delta x , \Delta x \rangle - \frac{1}{2} \langle \Gamma_\text{opt}(x) \Delta \hat x ,\Delta \hat x\rangle
 =\frac{1}{2} (D^2\phi(x)-\Gamma_\text{opt}(x))_{11} |(\Delta \hat x)_1|^2\\
   +\frac{1}{2} \norm{\Delta \hat x - \Delta x}^2 
\langle (D^2\phi(x) n(x+\Delta x), n(x+\Delta x) \rangle - \norm{\Delta \hat x - \Delta x} \langle D^2\phi(x) n(x+\Delta x), \Delta \hat x \rangle.
\end{multline*} 
Since  $\ds (D^2\phi(x)-\Gamma_\text{opt}(x))_{11}=\frac{1}{2}\Big(1-\frac{d^2(x)}{\eps^{2-2\alpha}} \Big) (D^2\phi(x))_{11}$ by \eqref{Gamma_opt},
one obtains  
\begin{multline*}
  \left| \frac{1}{2}  \langle D^2\phi(x)  \Delta x ,  \Delta x \rangle  -\frac{1}{2}  \langle \Gamma_\text{opt}(x) 
  \Delta \hat x ,\Delta \hat x\rangle \right| \\
  \leq  \norm{D^2\phi(x)} \left\{ \frac{1}{4} \left(1-\frac{d^2(x)}{\eps^{2-2\alpha}} \right) |(\Delta \hat x)_1|^2
      +\frac{1}{2} \norm{\Delta \hat x - \Delta x}^2  + \norm{\Delta \hat x - \Delta x} \norm{ \Delta \hat x} \right\} .
\end{multline*}  
The estimate \eqref{cp_D2phi_gamma_opt_0} now follows from \eqref{moving1_new} and \eqref{eq_11}.  
\end{proof}

\begin{lemma}\label{lm_estimate_l2} 
For any $x \in \Omega(\eps^{1-\alpha})$ and any function $\phi$ defined at $x$, $l_2^{x}[\phi]$ being defined by \eqref{def_l2}, we have
\begin{equation*} 
 - 2 \norm{D^2\phi(x) }  \eps^{1-\alpha} \left( \eps^{1-\alpha}-d(x) \right)  \leq l_2^{x}[\phi] \leq 0. 
\end{equation*}
\end{lemma}
\begin{proof} By considering $\Delta \hat x = 0$, $l_2$ is negative. We seek now to find a lower bound on $l_2$.  
By combining  the triangle inequality with Lemma~\ref{est2_phi_gopt}, the explicit expression of $\Gamma_{\text{opt}}(x)$ given by \eqref{Gamma_opt}
and  $\norm{\Delta \hat x} \leq \eps^{1-\alpha}$, we deduce that
\begin{align}
 \frac{1}{2} \Big|   \langle D^2\phi(x)  \Delta x, \Delta x \rangle 
 - & \left\langle \Gamma_{\text{opt}}(x) \Delta \hat x, \Delta \hat  x \right\rangle  
+  \frac{1}{2} \left( \eps^{1-\alpha}- \frac{d^2(x))}{\eps^{1-\alpha}} \right)    (D^2\phi(x))_{11}  (\Delta \hat x)_1 \Big| \notag \\
 &  \leq \frac{1}{4}\norm{D^2\phi(x)} (\eps^{1-\alpha} - d(x)) \left(7 \eps^{1-\alpha} -  d(x) \right)  
+ \frac{1}{4} \norm{D^2\phi(x) }  \left(\eps^{1-\alpha}-\frac{d^2(x)}{\eps^{1-\alpha}} \right)  \eps^{1-\alpha} \notag \\
 &  \leq 2\norm{D^2\phi(x)} (\eps^{1-\alpha} - d(x))  \eps^{1-\alpha},  \label{est_opt_order2}
\end{align}
which is precisely the proposed estimate.
\end{proof}

We shall now provide the consistency estimates about the upper bound of \eqref{operator_fonction_C2}.
Before stating our main estimate in Proposition \ref{cons_new_sub}, we can give a simple case for which the usual estimate holds. 

\begin{lemma}\label{lem_elem_M_neg}
 Let $f$ satisfy \eqref{ellipticity_f} and \eqref{loc_lip_p_Gamma}--\eqref{cont_growth_p_Gamma}  
 and assume $\alpha$, $\beta$, $\gamma$ satisfy \eqref{condition_pas}--\eqref{cd_coeff_classiq}.
For any $x$, $t$, $z$ and any smooth function $\phi$ defined near $x$, 
$ S_\eps[x,t,z,\phi]$ being defined by \eqref{def_new_op}, 
if $d(x)\leq \eps^{1-\alpha}$ and
$M_\eps^{x}[\phi] \leq -  \frac{1}{2} \norm{D^2\phi(x)} \left(3\eps^{1-\alpha} - d(x) \right)$, then we have 
\begin{equation*}
 S_\eps[x,t,z, \phi] - \phi(x) \leq  -\eps^2 f(t,x,z, D\phi(x), D^2\phi(x)) +o(\eps^2).
 \end{equation*}
Moreover, the implicit constant in the error term is uniform as $x$, $t$ and $z$ range over a compact subset of $\ol \Omega \times \R \times \R$.
\end{lemma}

In the rest of the section, we now accurately focus on the case $d(x)\leq \eps^{1-\alpha}$. 
The goal is to obtain precise estimates on \eqref{operator_fonction_C2} in the following three cases:
 $M_\eps^{x}[\phi]$ very negative, $M_\eps^{x}[\phi]$ very positive and $M_\eps^{x}[\phi]$ close to zero, the bounds between 
the cases depending on some powers of $\eps$.  We have formally shown in Section \ref{heuristic_derivation} that the first case is favorable to Mark 
since Helen can undergo a big penalty if Mark chooses to cross the boundary. 
On the contrary, the second case is preferable to Helen because she can receive a big coupon if the boundary is crossed.
In the last case, the boundary is transparent (think of $M_\eps^{x}[\phi]=0$) and the penalization due to the boundary 
is to be considered only through second order terms.
In order to establish the precise estimates, we successively introduce  two additional parameters $\rho, \kappa>0$ such that   
\begin{equation}  \label{def_nul} 
 1-\alpha<\rho  < \min \left(1 -\frac{\gamma(r-1)}{2}  , 2- 2 \alpha - \gamma \right), 
\end{equation}
and 
 \begin{equation} \label{relation_gamma_nul_tilde}
\gamma +\rho - (1-\alpha) <  \kappa  < 1-\alpha. 
\end{equation}
These coefficients are well-defined by virtue of \eqref{condition_pas} and \eqref{cd_coeff_classiq}.

\begin{prop}\label{cons_new_sub}
Let $f$ satisfy \eqref{ellipticity_f} and \eqref{loc_lip_p_Gamma}--\eqref{cont_growth_p_Gamma} 
 and assume $\alpha$, $\beta$, $\gamma$, $\rho$, $\kappa$  satisfy \eqref{condition_pas}--\eqref{cd_coeff_classiq} 
and \eqref{def_nul}--\eqref{relation_gamma_nul_tilde}.    
Let  $p_\text{opt}^M$ and $\Gamma_\text{opt}$ be respectively defined in the orthonormal basis $(e_1=n(\bar x), e_2,\cdots, e_N)$ by 
\eqref{p_opt_M} and \eqref{Gamma_opt}.
 For any $x$, $t$, $z$ and any smooth function $\phi$ defined near $x$, $S_\eps[x,t,z,\phi]$ being defined by \eqref{def_new_op}, we distinguish four cases:
\begin{enumerate} [label=\textup{ }{\roman*.}\textup{ },ref=({\roman*})] 
 \item Big bonus: \label{cons_sub_cas1} If $d(x) \leq \eps^{1-\alpha}$ and $M_\eps^{x}[\phi] >\frac{4}{3}\norm{D^2\phi(x)} \eps^{1-\alpha} $, then 
\begin{equation*} 
 S_\eps[x,t,z, \phi] - \phi(x)  \leq 3 (\eps^{1-\alpha}-d(x)) M_\eps^{x}[\phi] - \eps^2 f(t,x,z,p_\text{opt}^M(x), \Gamma_\text{opt}(x)) +o(\eps^2).
 \end{equation*}
 \item \label{cons_sub_cas2} Far from the boundary with a small bonus: if $\eps^{1-\alpha} - \eps^{\rho} \leq d(x) \leq \eps^{1-\alpha}$ 
and $M_\eps^{x}[\phi]  \leq \frac{4}{3} \norm{D^2\phi(x)} \eps^{1-\alpha}$, or if $d(x)\geq \eps^{1-\alpha}$, then 
\begin{equation*}
S_\eps[x,t,z, \phi] - \phi(x) \leq  -\eps^2 f(t,x,z,D\phi(x), D^2\phi(x)) +o(\eps^2) . 
\end{equation*} 
\item \label{cons_sub_cas3} Close to the boundary with a small bonus/penalty:
if $d(x) \leq \eps^{1-\alpha} - \eps^{\rho}$ and $-\eps^{1-\alpha -\kappa}\leq M_\eps^{x}[\phi]\leq \frac{4}{3}\norm{D^2\phi(x)}\eps^{1-\alpha}$, then  
\begin{equation*}  
 S_\eps[x,t,z, \phi] - \phi(x) \leq  -\eps^2 f(t,x,z,D\phi(x), D^2\phi(x)+C_1 I) +o(\eps^2), 
 \end{equation*}
with $C_1=\frac{20}{3} \norm{D^2\phi(x)} \left(1 - \frac{d(x)}{\eps^{1-\alpha}}\right)$. 
\item \label{cons_sub_cas4} Close to the boundary with a big penalty:  
if $d(x) \leq \eps^{1-\alpha} - \eps^{\rho}$   and    $M_\eps^{x}[\phi]\leq  - \eps^{1-\alpha - \kappa}$, then  
\begin{equation} \label{est_upp_neg}
 S_\eps[x,t,z, \phi] - \phi(x)  
 \leq  \frac{1}{4} (\eps^{1-\alpha} - d(x)) M_\eps^{x}[\phi] 
- \eps^2  \min_{p\in B(p_\text{opt}^M(x), r)} f(t,x,z,p, \Gamma_\text{opt}(x)) +o(\eps^2), 
 \end{equation}
with $r$ defined by $r= 3\Big(1-\frac{d(x)}{\eps^{1-\alpha}}\Big) |M_\eps^{x}[\phi]|$.
\end{enumerate}
Moreover, the implicit constants in the error term is uniform as $x$, $t$ and $z$ range over a compact subset of $\ol \Omega \times \R \times \R$.
\end{prop}

Before proving these estimates,  it is worth drawing a parallel with the formal derivation done at Section~\ref{heuristic_derivation}. 
The lower bound proposed by Proposition \ref{cons_lower_bound} case \ref{case1_lower} corresponds to the formal analysis when $m>0$. 
The upper bound proposed by Proposition \ref{cons_new_sub} case \ref{cons_sub_cas4} is associated to the formal analysis when $m<0$. 
Furthermore, we can observe in the proof that the factor $1/4$ in \eqref{est_upp_neg} could be replaced by any number in $[1/4,1/2)$,   
the bound $1/2$ corresponding to the heuristic derivation given by \eqref{eq_formal_final}.

\subsection{Proof of Lemma \ref{lem_elem_M_neg} and Proposition \ref{cons_new_sub}}
\label{tech_proof}
For sake of notational simplicity, we write $\lambda_{\text{min}}(A)$ for the smallest eigenvalue of the symmetric matrix $A$ and we omit 
the $x$-dependence of $p_\text{opt}^M(x)$ and $\Gamma_\text{opt}(x)$.
Moreover,  by the definition of $M_\eps^{x}[\phi]$ given by \eqref{M_eps_par} and the positivity of $\norm{\Delta \hat x - \Delta x}$, 
for all $\norm{\Delta \hat x} \leq \eps^{1-\alpha}$, we have
\begin{equation}\label{maj_fond_M}
 \norm{\Delta \hat x - \Delta x} \left\{ h(x+\Delta x) - D\phi(x) \cdot n(x+\Delta x) \right\}  \leq \norm{\Delta \hat x - \Delta x} M_\eps^{x}[\phi].
\end{equation}
Therefore it is sufficient to find an upper bound for 
\begin{equation*}
 \max_{p, \Gamma} \min_{\Delta \hat x} \left[(D\phi(x) - p)\cdot \Delta \hat x +M_\eps^{x}[\phi] \norm{\Delta \hat x - \Delta x}
 + \frac{1}{2} \langle D^2\phi(x)  \Delta x ,\Delta x\rangle - \frac{1}{2} \langle \Gamma \Delta \hat x,\Delta \hat x\rangle    \\
    - \eps^2 f(t,x,z,p,\Gamma) \right].
\end{equation*}
In other words, by taking advantage of the monotonicity of the operator $S_\eps$ with \eqref{maj_fond_M},  we shall  
look for an upper bound for an approximated operator bounding $S_\eps$ above and very close to it as $\eps \rightarrow 0$.

\subsubsection{Proof of Lemma \ref{lem_elem_M_neg}}

We introduce 
\begin{equation}\label{def_A_ronde}
\mathcal{A}^ x(p,\Gamma,\Delta \hat x) : = (D\phi( x) - p)\cdot \Delta \hat x  + \norm{\Delta \hat x - \Delta x}M_\eps^ x[\phi]  
+\frac{1}{2} \langle D^2\phi( x)  \Delta x , \Delta x \rangle  -\frac{1}{2} \langle \Gamma \Delta \hat x ,\Delta \hat x\rangle 
 - \eps^2 f(t,x,z,p,\Gamma), 
\end{equation}
where $\Delta x=\proj_{\ol \Omega}(x+\Delta \hat x) - x$. We give the following useful decomposition: 
\begin{equation}\label{decomp_ord2_useful}
 \frac{1}{2} \langle D^2\phi( x) \Delta x , \Delta x\rangle  - \frac{1}{2} \langle \Gamma \Delta \hat x , \Delta \hat x \rangle  
 = \frac{1}{2} \langle D^2\phi( x) \Delta x , \Delta x\rangle  
- \frac{1}{2} \langle D^2\phi(x) \Delta \hat x , \Delta \hat x \rangle  
+ \frac{1}{2} \langle (D^2\phi(x) - \Gamma ) \Delta \hat x , \Delta \hat x \rangle, 
\end{equation}
which will be used repeatedly in this section. We clearly have by \eqref{est_D2phi_x_hatx} that  
\begin{multline*}
 \norm{\Delta \hat x - \Delta x} M_\eps^x[\phi] +\frac{1}{2} \langle D^2\phi(x)  \Delta x , \Delta x \rangle  
- \frac{1}{2} \langle D^2\phi(x)  \Delta \hat x , \Delta \hat x \rangle  \\
\leq  \norm{\Delta \hat x - \Delta x}  \left( M_\eps^{x}[\phi] +\frac{1}{2}  \norm{D^2\phi(x)}  \Big(3\eps^{1-\alpha} - d(x) \Big) \right)
 \leq 0.
\end{multline*}
From the previous inequality  and \eqref{decomp_ord2_useful} we deduce that 
for all $p,\Gamma,\Delta \hat x$ constrained by \eqref{p_beta_gamma_new}--\eqref{moving1_new}, 
\begin{align*}
\mathcal{A}^x(p,\Gamma,\Delta \hat x) & \leq  (D\phi(x) - p)\cdot \Delta \hat x +\frac{1}{2} \langle (D^2\phi(x) - \Gamma)  \Delta \hat x , \Delta \hat x \rangle 
 - \eps^2 f(t,x,z,p, \Gamma).
\end{align*}
By monotonicity of the operator $ S_\eps$ and by using \cite[Lemma 4.1]{kohns} to estimate the max min, we conclude that
\begin{align*}
S_\eps[x,t,z,\phi] - \phi(x) &  \leq \max_{p, \Gamma} \min_{\Delta \hat x} \Big[
 (D\phi(x) - p)\cdot \Delta \hat x +\frac{1}{2} \langle (D^2\phi(x) - \Gamma)  \Delta \hat x , \Delta \hat x \rangle 
 - \eps^2 f(t,x,z,p, \Gamma)\Big]  \\
 &  \leq  - \eps^2 f(t,x,z, D\phi(x), D^2\phi(x)) +o(\eps^2), 
\end{align*}
which gives the desired result.

\subsubsection{Proof of Proposition \ref{cons_new_sub} case \ref{cons_sub_cas1}}

We define the function $\mathcal{A}_b^x$ of $\Delta \hat x$ associated  to the particular choice $p = p_\text{opt}^M$ and $ \Gamma= \Gamma_\text{opt}$ by 
 \begin{equation}\label{def_Ab}
\mathcal{A}^x_b(\Delta \hat x)  = (D\phi(x)  - p_\text{opt}^M)\cdot \Delta \hat x  +  \norm{\Delta \hat x - \Delta x} M_\eps^{x}[\phi] 
  +\frac{1}{2} \langle D^2\phi(x)  \Delta x , \Delta x \rangle - \frac{1}{2}\langle \Gamma_\text{opt} \Delta \hat x ,\Delta \hat x\rangle,  
\end{equation}
where $\Delta x=\proj_{\ol \Omega}(x+\Delta \hat x) - x$. Thus, the operator $S_\eps$ can be written in the form
\begin{equation}  \label{transfo_operator}
S_{\eps} [x,t,z,\phi] - \phi(x)=\max_{p,\Gamma} \min_{\Delta \hat x} \left[  \mathcal{A}^x_b(\Delta \hat x)  + (p_\text{opt}^M - p)\cdot \Delta \hat x 
+\frac{1}{2} \langle (\Gamma_\text{opt}- \Gamma) \Delta \hat x ,\Delta \hat x\rangle  - \eps^2 f(t,x,z,p,\Gamma)
\right].
\end{equation}
To compute an upper bound of \eqref{transfo_operator}, we now introduce two preliminary lemmas.

\begin{lemma}\label{lemma_sup_Ab}
Assume that $M_\eps^{x}[\phi]\geq 0$. Then $\mathcal{A}^{x}_b$ defined by \eqref{def_Ab} is $\Delta \hat x$-bounded by 
\begin{equation} \label{sup_Ab}
0\leq \sup_{\Delta  \hat x}\mathcal{A}^{x}_b(\Delta \hat x) \leq  \frac{1}{2}(\eps^{1-\alpha}- d(x))  \left( 3 M_\eps^{x}[\phi]
+ 4 \norm{D^2\phi(x)} \eps^{1-\alpha} \right),
\end{equation}
where $\Delta \hat x$ is constrained by \eqref{moving1_new}.
\end{lemma}
\begin{proof} This estimate follows exactly the same lines as for Lemmas \ref{lm_estimate_l1}--\ref{lm_estimate_l2}. The sup is clearly positive by considering 
$\Delta \hat x =0$. Then, by plugging the expression of $p^M_\text{opt}$ in  $\mathcal{A}_b ( \Delta \hat x)$, we have
\begin{multline*}
\mathcal{A}^{x}_b ( \Delta \hat x) 
= \Big\{ -\frac{\eps^{1-\alpha}- d(x)}{2 \eps^{1-\alpha} }(\Delta \hat x)_1 + \norm{\Delta \hat x - \Delta x}\Big\} M_\eps^{x}[\phi]  \\
   +\frac{1}{4} \Big(\eps^{1-\alpha} - \frac{d^2(x)}{\eps^{1-\alpha}}\Big) (D^2\phi(x))_{11}   (\Delta \hat x)_1 
     +\frac{1}{2} \langle D^2\phi(x)  \Delta x , \Delta x \rangle  -\frac{1}{2} \langle \Gamma_\text{opt} \Delta \hat x ,\Delta \hat x\rangle .
\end{multline*}
Since $M_\eps^{x}[\phi]\geq 0$, using the estimates \eqref{ineq_key} and \eqref{est_opt_order2}, we obtained the desired estimate.
\end{proof}

\begin{lemma} \label{consistence_bounded_sequence}
Let $f$ satisfy \eqref{ellipticity_f} and \eqref{loc_lip_p_Gamma}--\eqref{cont_growth_p_Gamma} and assume $\alpha$, $\beta$, $\gamma$ satisfy 
\eqref{condition_pas}- \eqref{cd_coeff_classiq}.
Let $(p_\eps)_{0<\eps\leq 1}$ and $(\Gamma_\eps)_{0<\eps \leq1}$ be two sequences  bounded respectively in $\R^N$ and $\mathcal{S}^N$. 
Then for any $x$, $t$ and $z$, we have
\begin{equation*}
\max_{\substack{\norm{p}\leq \eps^{-\beta}\\ \norm{\Gamma}\leq \eps^{-\gamma}}} \min_{\norm{\Delta \hat x}\leq \eps^{1-\alpha}} 
\left[ (p_\eps  - p)\cdot \Delta \hat x +\frac{1}{2} \langle (\Gamma_\eps - \Gamma) \Delta \hat x ,\Delta \hat x\rangle
   - \eps^2 f \left(  t,x, z,  p,\Gamma \right)\right]  \\
 = -\eps^2 f(t,x,z,p_\eps, \Gamma_\eps) +o(\eps^2)  .
\end{equation*}
Moreover, the implicit constant in the error term is uniform as $x$, $t$, and $z$ range over a compact subset of $\ol \Omega \times \R \times \R$.
\end{lemma}
\begin{proof}
It is a direct adaptation of \cite[Lemma 4.1]{kohns} by 
distinguishing three cases according to the size of $\norm{p_\eps - p}$ and  $\lambda_\text{min}(\Gamma_\eps - \Gamma)$. 
\end{proof}

We can now provide an upper bound on  \eqref{transfo_operator}. 
By Lemma \ref{lemma_sup_Ab}, $\mathcal{A}_b$ is upper bounded independently of all possible moves $\Delta \hat x$. 
It follows from \eqref{transfo_operator} that  
\begin{equation*}
S_{\eps} [x,t,z,\phi] - \phi(x)  \leq \sup_{\Delta \hat  x}\mathcal{A}^x_b(\Delta \hat x)
+  \max_{p,\Gamma} \min_{\Delta \hat x} 
\left[  (p_\text{opt}^M - p)\cdot \Delta \hat x -\frac{1}{2} \langle \Gamma_\text{opt} \Delta \hat x ,\Delta \hat x\rangle 
- \eps^2 f \left( t, x, z,  p,\Gamma \right) \right]  .
\end{equation*}
The consistency Lemma \ref{consistence_bounded_sequence} provides an estimate of the max min and one obtains
\begin{align*}
S_{\eps}[x,t,z,\phi]& -\phi(x)\leq \sup_{\Delta \hat x}\mathcal{A}^x_b(\Delta \hat x) -\eps^2 f(t,x,z,p_\text{opt}^M,\Gamma_\text{opt})+o(\eps^2).
\end{align*}
By plugging the upper bound in \eqref{sup_Ab} of $\mathcal{A}^x_b$ in the previous inequality, we obtained the desired result.

\subsubsection{Proof of Proposition \ref{cons_new_sub} case \ref{cons_sub_cas2}}

It is sufficient to show that for any $\norm{p}\leq \eps^{-\beta}$ and $\norm{\Gamma}\leq \eps^{-\gamma}$, 
there exists $\norm{\Delta \hat x}\leq \eps^{1-\alpha}$, determining $\Delta x$ by~\eqref{exp_delta_x}, such that
\begin{multline}\label{eq_fund}
 (D\phi(x) - p)\cdot \Delta \hat x +M_\eps^{x}[\phi] \norm{\Delta \hat x - \Delta x}
 + \frac{1}{2} \langle D^2\phi(x)  \Delta x ,\Delta x\rangle - \frac{1}{2} \langle \Gamma \Delta \hat x,\Delta \hat x\rangle    \\
    - \eps^2 f(t,x,z,p,\Gamma)  \leq  -\eps^2 f(t,x,z,D\phi(x), D^2\phi(x)) +o(\eps^2)  , 
\end{multline}
with an error estimate $o(\eps^2)$ that is independent of $p$ and $\Gamma$ and locally uniform in $x$, $t$, $z$. 
 In view of the conditions \eqref{cd_coeff_classiq} and \eqref{def_nul}, we can pick $\mu>0$ and  $\delta >0$ such that 
\begin{align}
 \mu+\gamma & <1-\alpha \text{ and }  \mu +\gamma r <1+\alpha,  \label{cond_mu}\\
\delta & <\min(2\alpha, \rho - (1-\alpha)).                     \label{definition_delta} 
\end{align}
Now we consider separately the following three cases:
\begin{enumerate}
 \item \label{cas1_easy} $\norm{D\phi(x) - p}\leq \eps^\mu$ and $\lambda_\text{min}(D^2\phi(x)- \Gamma)\geq - \eps^{\delta}$,
 \item \label{cas2_easy} $\norm{D\phi(x) - p}\leq \eps^\mu$ and $\lambda_\text{min}(D^2\phi(x)- \Gamma)\leq - \eps^{\delta}$,
 \item \label{cas3_easy} $\norm{D\phi(x) - p}\geq \eps^\mu$.
\end{enumerate} 

For case \ref{cas1_easy}, we choose $\Delta \hat x=0$. 
By a reasoning similar to Case 1 in the proof of \cite[Lemma 4.1]{kohns}, we obtained the inequality given by \eqref{eq_fund}.

For cases \ref{cas2_easy} and \ref{cas3_easy}, in order to use the decomposition \eqref{decomp_ord2_useful}, we now give  a preliminary inequality. 
 By the inequality \eqref{est_D2phi_x_hatx} in Lemma~\ref{est2_phi_gopt},  we have
\begin{equation*} 
\left|    \frac{1}{2} \langle D^2\phi(x) \Delta x ,\Delta x\rangle -  \frac{1}{2} \langle D^2\phi(x) \Delta \hat x ,\Delta \hat x\rangle \right| 
  \leq \frac{3}{2} \norm{D^2\phi(x)}    \norm{\Delta \hat x - \Delta x} \eps^{1- \alpha} , 
\end{equation*}
which yields  with the assumption $M_\eps^{x}[\phi] \leq \frac{4}{3} \norm{D^2 \phi(x)} \eps^{1-\alpha}$ that
\begin{equation}\label{term1_neg}
M_\eps^{x} [\phi] \norm{\Delta \hat x - \Delta x}+\frac{1}{2} \langle D^2\phi(x) \Delta x ,\Delta x\rangle
 -  \frac{1}{2} \langle D^2\phi(x) \Delta \hat x ,\Delta \hat x\rangle 
  \leq   \frac{17}{6}\norm{D^2\phi(x)} \eps^{1- \alpha} \norm{\Delta \hat x - \Delta x}.
\end{equation}
By combining the geometric estimate \eqref{eq_11} with the assumption $d(x)\geq \eps^{1-\alpha} - \eps^{\rho}$, 
we get that the left-hand side of \eqref{term1_neg} is upper bounded by $\frac{17}{6}\norm{D^2\phi(x)} \eps^{1- \alpha+\rho}$.
By using the decomposition \eqref{decomp_ord2_useful} we deduce that it is sufficient to show that 
there exists $\norm{\Delta \hat x}\leq \eps^{1-\alpha}$ such that
\begin{multline*}
 (D\phi(x) - p) \cdot \Delta \hat x  +  \frac{1}{2} \langle (D^2\phi(x)- \Gamma) \Delta \hat x , \Delta \hat x  \rangle 
+ \frac{17}{6} \norm{D^2\phi(x)} \eps^{1-\alpha+\rho}   - \eps^2 f(t,x,z,p,\Gamma) \\
  \leq  - \eps^2 f(t,x,z,D\phi(x), D^2\phi(x)).
\end{multline*}

For case \ref{cas2_easy}, we choose $\Delta \hat x$ to be an eigenvector for the minimum eigenvalue $\lambda = \lambda_{\text{min}}(D^2\phi(x)- \Gamma)$ 
of norm $\eps^{1-\alpha}$. Notice that since $f$ is monotone in its last input,  we have 
\begin{equation*}
f(t,x,z,p,\Gamma) \geq f(t,x,z, D^2\phi(x) - \lambda I). 
\end{equation*}
Choosing $\Delta \hat x$ as announced, and changing the sign if necessary to make 
$(D\phi(x) - p) \cdot \Delta \hat x\leq 0$, we deduce that
\begin{multline*}
(D\phi(x) - p) \cdot \Delta \hat x  +  \frac{1}{2} \langle (D^2\phi(x)- \Gamma) \Delta \hat x , \Delta \hat x  \rangle 
+ \frac{17}{6} \norm{D^2\phi(x)} \eps^{1-\alpha+\rho} - \eps^2 f(t,x,z,p,\Gamma)   \\
\leq  \frac{1}{2} \eps^{2-2\alpha} \lambda  + \frac{17}{6} \norm{D^2\phi(x)} \eps^{1-\alpha+\rho}
 - \eps^2 f(t,x,z,p, D^2\phi(x) - \lambda I).
\end{multline*}
If $-1\leq \lambda \leq -\eps^\delta $ then $\eps^{2-2\alpha} \lambda \leq - \eps^{2-2\alpha+\delta}$ 
  and  $f(t,x,z,p, D^2\phi(x)- \lambda I)$ is bounded. Since $\eps^{1-\alpha+\rho}\ll \eps^{2-2\alpha+\delta} $ by~\eqref{definition_delta},
for such $\lambda$ we have
\begin{equation*}
 \frac{1}{2} \eps^{2-2\alpha} \lambda + \frac{17}{6} \norm{D^2\phi(x)} \eps^{1-\alpha+\rho}
- \eps^2 f(t,x,z, p, D^2\phi(x) - \lambda I) \leq - \frac{1}{4} \eps^{2-2\alpha+\delta} +O(\eps^2).
\end{equation*}
In this case, we are done by~\eqref{definition_delta}, since the right-hand side is $\leq \eps^2 f(t,x,z,D\phi(x), D^2\phi(x)) $
when $\eps$ is small enough.

To complete case \ref{cas2_easy}, suppose $\lambda \leq  - 1$. 
 Then using the growth hypothesis \eqref{cont_growth_p_Gamma} and recalling that $p$ is near $D\phi(x)$ we have 
\begin{equation*}
 \frac{1}{2} \eps^{2-2 \alpha} \lambda - \eps^2 f(t,x,z,p, D^2\phi(x) -\lambda I)
\leq -\frac{1}{2}  \eps^{2-2 \alpha} |\lambda| +C\eps^2 (1+|\lambda|^r). 
\end{equation*}
Now notice that $|\lambda| \leq C(1+\norm{\Gamma}) \leq C\eps^{-\gamma}$. 
Since $\gamma(r-1)<2\alpha$ we have $\eps^{2-2\alpha} |\lambda| \gg \eps^2 |\lambda|^r$. Therefore we deduce by \eqref{definition_delta} that
\begin{equation*}
 -\frac{1}{2} \eps^{2-2 \alpha} |\lambda|  +C\eps^2|\lambda|^r + \frac{17}{6} \norm{D^2\phi(x)} \eps^{1-\alpha+\rho}
 \leq - \frac{1}{4} \eps^{2-2 \alpha}  \leq - \eps^2 f(t,x,z,D\phi(x),D^2\phi(x)),  
\end{equation*}
when $\eps$ is sufficiently small. Case \ref{cas2_easy} is now complete.

Finally, to treat case \ref{cas3_easy}, 
we take $\Delta \hat x$ parallel to $D\phi (x)-p$ with norm $\eps^{1-\alpha}$, and with the sign chosen such that
\begin{equation*}
(D\phi(x)-p)\cdot \Delta \hat x = -\eps^{1-\alpha} \norm{D\phi(x)-p} \leq - \eps^{1-\alpha+\mu}. 
\end{equation*}
By observing that 
$\ds  \frac{17}{6} \norm{D^2\phi(x)} \eps^{1-\alpha+\rho} \ll \eps^{1-\alpha} \norm{D\phi(x)-p}$,
 this case follows exactly the sames lines as \cite[Lemma 4.1]{kohns}.

\subsubsection{Proof of Proposition \ref{cons_new_sub} case \ref{cons_sub_cas3}}

This proof is quite similar to case \ref{cons_sub_cas2}. 
Since this estimate will not be needed in the rest of the paper, we just indicate 
that we need to distinguish three cases according to the respective sizes of $\norm{D\phi(x) - p}$
and $\lambda_\text{min}(D^2\phi(x) - \Gamma)$ with respect to  $\eps^\mu$ and $- C_1 - \eps^{\alpha}$, where $\mu$ is defined by \eqref{cond_mu}.

\subsubsection{Proof of Proposition \ref{cons_new_sub} case \ref{cons_sub_cas4}}

This case corresponds to the heuristic derivation presented at Section \ref{heuristic_derivation} when $m<0$. 
Recalling that $p_\text{opt}^M$ and $\Gamma_\text{opt}$ are defined by \eqref{p_opt_M}--\eqref{Gamma_opt}, 
our task is to show that for any $\norm{p}\leq \eps^{-\beta}$ and $\norm{\Gamma}\leq \eps^{-\gamma}$, 
there exists $\norm{\Delta \hat x}\leq \eps^{1-\alpha}$, determining $\Delta x$ by \eqref{exp_delta_x}, such that
\begin{multline} \label{eq_gene3}
 (D\phi(x) - p)  \cdot  \Delta \hat x   + \norm{\Delta \hat x - \Delta x} M_\eps^{x}[\phi]  
  + \frac{1}{2} \langle D^2\phi(x) \Delta x ,\Delta x\rangle - \frac{1}{2} \langle \Gamma \Delta \hat x ,\Delta \hat x\rangle  \\
 - \eps^2 f \left(t,x,z,p,\Gamma \right)\leq \frac{1}{4} (\eps^{1-\alpha} - d(x)) M_\eps^{x}[\phi]
 -\eps^2 \min_{p\in B (p_\text{opt}^M,r)} f(t,x,z,p, \Gamma_\text{opt}) +o(\eps^2),
\end{multline}
with  an error estimate $o(\eps^2)$ that is independent of $p$ and $\Gamma$ and locally uniform in $x$, $t$, $z$.
We can notice in~\eqref{eq_gene3} that the function $M_\eps^{x}[\phi]$ is $\eps, x$-bounded by $\norm{h}_{L^\infty}+\norm{D\phi}_{L^\infty}$. 
Moreover, by Lemma~\ref{est2_phi_gopt} we have
\begin{equation}
   \frac{1}{2} \langle D^2\phi(x)  \Delta x , \Delta x \rangle 
    -\frac{1}{2} \langle \Gamma_\text{opt} \Delta \hat x ,\Delta \hat x\rangle 
 \leq  \frac{7}{4}\norm{D^2\phi(x)} (\eps^{1-\alpha} - d(x)) \eps^{1-\alpha}. \label{cp_D2phi_gamma_opt}
\end{equation} 
Thus,  it is sufficient to examine,  for any $\norm{p}\leq \eps^{-\beta}$ and $\norm{\Gamma}\leq \eps^{-\gamma}$,  
\begin{equation} \label{min_fun} 
 \min_{\Delta \hat x} \left[
(D\phi(x) - p)\cdot \Delta \hat x + \norm{\Delta \hat x - \Delta x} M_\eps^{x}[ \phi]
+\frac{1}{2}\langle  (\Gamma_\text{opt} - \Gamma)  \Delta \hat x, \Delta \hat x \rangle  - \eps^2 f(t,x,z,p, \Gamma) 
\right].
\end{equation}
We consider separately the following three cases:
\begin{enumerate} [label=\textup{ }{\alph*.}\textup{ },ref=({\alph*})] 
 \item \label{cas1} $\norm{p_\text{opt}^M - p}\leq 3 \Big(1 - \frac{d(x)}{\eps^{1-\alpha}}\Big) |M_\eps^{x}[\phi]|$,   
and $\lambda_\text{min}(\Gamma_\text{opt} - \Gamma)\geq - \eps^{\alpha}$, 
 \item \label{cas2} $\norm{p_\text{opt}^M - p}\leq 3 \Big(1 - \frac{d(x)}{\eps^{1-\alpha}}\Big) |M_\eps^{x}[\phi]|$,   
and $\lambda_\text{min}(\Gamma_\text{opt} - \Gamma)\leq - \eps^{\alpha}$, 
 \item \label{cas5} $\norm{p_\text{opt}^M - p}\geq 3 \Big(1 - \frac{d(x)}{\eps^{1-\alpha}}\Big) |M_\eps^{x}[\phi]|$. 
\end{enumerate}
For case \ref{cas1}, we choose $\Delta \hat x=\pm \eps^{1-\alpha} n(\bar x)$ with the sign chosen such that
\begin{equation*}
(p-p_\text{opt}^M)\cdot \Delta \hat x \leq 0. 
\end{equation*}
 Since $\lambda_\text{min}(\Gamma_\text{opt}- \Gamma)\geq - \eps^{\alpha}$
 we have $\Gamma_\text{opt}- \Gamma +  \eps^{\alpha} I \geq 0$ and thus 
$\Gamma \leq \Gamma_\text{opt}+  \eps^{\alpha} I$. Using the monotonicity of $f$ with respect to its last entry, this gives
$f(t,x,z,p, \Gamma)\geq f(t,x,z,p, \Gamma_\text{opt} +\eps^\alpha I)$. 
Since $f$ is locally Lipschitz, we conclude that
\begin{equation} \label{est_loc_lip_M_very_neg}
 f(t,x,z,p,\Gamma)\geq f(t,x,z,p, \Gamma_\text{opt})+O(\eps^\alpha) \geq  \min_{p\in B (p_\text{opt}^M,r)} f(t,x,z,p, \Gamma_\text{opt}) +O(\eps^\alpha).
\end{equation}
The constant in the error term is independent of $p$ and $\Gamma$, since we are assuming in case  \ref{cas1} that
 $\norm{p- p_\text{opt}^M}\leq  3 (\norm{h}_{L^\infty}+\norm{D\phi}_{L^\infty})$.
Moreover we directly compute
\begin{equation} \label{est_ord1_cas4}
(D\phi(x) - p_\text{opt}^M)\cdot \Delta \hat x + \norm{\Delta \hat x - \Delta x} M_\eps^{x}[\phi]= \frac{1}{2} (\eps^{1-\alpha} - d(x)) M_\eps^{x}[ \phi] .
\end{equation}
Since $\eps^{1-\alpha} - d(x) \geq \eps^{\rho }$ and $M_\eps^{x}[\phi]< 0$, we have
\begin{equation}\label{est_dM_rho}
\frac{1}{2} (\eps^{1-\alpha}-d(x)) M_\eps^{x}[\phi]\leq \frac{1}{2}\eps^{\rho} M_\eps^{x}[\phi]\leq -\frac{1}{2} \eps^{1 -\alpha -\kappa+ \rho}.
\end{equation}
By noticing that $\eps^{2-2\alpha - \gamma}\ll \eps^{1 - \alpha - \kappa + \rho}$ using \eqref{relation_gamma_nul_tilde}, we deduce from \eqref{est_dM_rho} that
\begin{equation}\label{est_ord1_cas4_gamma}
\Big| \frac{1}{2} \langle (\Gamma_\text{opt} - \Gamma) \Delta \hat x, \Delta \hat x \rangle \Big| 
 \leq \frac{1}{2}(\norm{D^2 \phi(x)} +\eps^{-\gamma}) \eps^{2-2\alpha}
\leq \frac{3}{4}\eps^{2-2\alpha - \gamma}  
 \ll (\eps^{1-\alpha} - d(x)) M_\eps^{x}[\phi].  
\end{equation}
Therefore, by combining \eqref{cp_D2phi_gamma_opt}, \eqref{est_ord1_cas4} and \eqref{est_ord1_cas4_gamma}, 
the choice $\Delta \hat x= \pm \eps^{1-\alpha} n(\bar x)$ in the left-hand side of \eqref{eq_gene3} yields 
\begin{align*}
(D& \phi(x) - p) \cdot \Delta \hat x  + \norm{\Delta \hat x - \Delta x} M_\eps^{x}[\phi]   + \frac{1}{2} \langle D^2\phi(x) \Delta x ,\Delta x\rangle 
 - \frac{1}{2} \langle \Gamma \Delta \hat x ,\Delta \hat x\rangle - \eps^2 f \left(t,x, z,  p,\Gamma \right) \\
&  \leq  \frac{1}{2} ( \eps^{1-\alpha} - d(x)) \big( M_\eps^{x}[ \phi] + \frac{7}{2}\norm{D^2\phi(x)} \eps^{1-\alpha} \big)
    + \frac{3}{4}\eps^{2-2\alpha - \gamma} 
 - \eps^2 \min_{p\in B (p_\text{opt}^M,r)} f(t,x,z,p, \Gamma_\text{opt}) +o(\eps^2) \\
& \leq  \frac{1}{4} (\eps^{1-\alpha}-d(x)) M_\eps^{x}[\phi]  - \eps^2 \min_{p\in B (p_\text{opt}^M,r)} f(t,x,z,p, \Gamma_\text{opt}) +o(\eps^2), 
\end{align*}
as desired.

For  case \ref{cas2}, in view of the condition \eqref{def_nul}, we can pick  $\sigma >1-\alpha$ such that 
\begin{equation}\label{cond_mu_l}
 \rho <\sigma <1 -  \frac{\gamma(r-1)}{2}. 
\end{equation}  
Let $v^\lambda$ be a unit eigenvector for the minimum eigenvalue $\lambda = \lambda_{\text{min}}(\Gamma_\text{opt}- \Gamma)$.
We choose $\Delta \hat x$ of the form 
\begin{equation}\label{def_move3}
\Delta \hat x=\pm \left[\left(\eps^{1-\alpha}-\eps^{\sigma}\right)n(\bar x)+\text{sgn}(\langle n(\bar x),v^\lambda \rangle) \eps^{\sigma} v^\lambda \right]
 = \pm \left[ a_1  n(\bar x) +b  v^\lambda \right], 
 \end{equation}
where 
$a_1=  \left(\eps^{1-\alpha} -  \eps^{\sigma}\right)$, $ b =\text{sgn}  (\langle n(\bar x), v^\lambda  \rangle ) \eps^{\sigma}$
 and  sgn denotes the sign function with the convention that $\text{sgn}(0)=1$. 
The sign $\pm$ will be chosen later. This move fulfills the following estimate.

\begin{lemma} The move $\Delta \hat x$ defined by \eqref{def_move3} is authorized by the game and satisfies
\begin{equation}\label{bonus_bdry_3}
 (D\phi(x)  - p_\text{opt}^M) \cdot \Delta \hat x  + \norm{\Delta \hat x - \Delta x} M_\eps^{x}[\phi]
 \leq \frac{\eps^{1-\alpha} -d(x)}{2} (   M_\eps^{x}[\phi] +  \norm{D^2\phi(x)} \eps^{1-\alpha} )
- 4 \eps^{\sigma} M_\eps^{x}[\phi], 
\end{equation}
independently of the choice on $\pm$ in \eqref{def_move3}.
\end{lemma}
\begin{proof} To authorize this move, it suffices to check that $\norm{\Delta \hat x }\leq \eps^{1-\alpha}$. 
After some calculations and by rearranging the terms, we compute
\begin{align*}
\norm{\Delta \hat x}^2 & =\eps^{2-2\alpha} +2 \eps^{2 \sigma}  
 - 2 \eps^{1-\alpha+\sigma} +  \eps^{\sigma} \Big( \eps^{1-\alpha} -  \eps^{\sigma} \Big)   |\langle n(\bar x), v^\lambda  \rangle|  \\
& = \eps^{2-2\alpha} - 2 \eps^{1-\alpha+\sigma} ( 1 -  \eps^{ \sigma-1+\alpha}  )   \Big(1 - \frac{1}{2} |\langle n(\bar x), v^\lambda  \rangle| \Big) 
\leq \eps^{2-2\alpha}.
\end{align*}
For the second part, we distinguish successively the two cases $\pm$.
By \eqref{def_move3}, we directly compute
\begin{equation}\label{move3_proj_n}
 \Delta \hat x \cdot n(\bar x)= \pm \left[ \left(\eps^{1-\alpha} -  \eps^{\sigma}\right) +
 |\langle n(\bar x), v^\lambda  \rangle|  \eps^{\sigma} \right] 
= \pm \left[ \eps^{1-\alpha} -   \left(1- |\langle n(\bar x), v^\lambda  \rangle| \right)  \eps^{\sigma} \right].
\end{equation}
If  $\Delta \hat x \cdot n(\bar x) \leq 0$, this move corresponds to the sign $-$ in \eqref{def_move3} by \eqref{move3_proj_n}
and  we observe that $\hat x \in \Omega$ by  Lemma~\ref{boundary_bounce1}. 
As a result, by introducing the explicit expressions of $p_\text{opt}^M$ and $(\Delta \hat x)_1$ 
respectively given by \eqref{p_opt_M} and~\eqref{move3_proj_n},  we get 
\begin{multline}\label{c1}
 (D\phi(x)   - p_\text{opt}^M)  \cdot \Delta \hat x + \norm{\Delta \hat x - \Delta x} M_\eps^{x}[ \phi]
 = (D\phi(x) - p_\text{opt}^M)_1  (\Delta \hat x)_1  \\
  =  - \left( -\frac{1}{2} (1 -\frac{d(x)}{\eps^{1-\alpha}})  M_\eps^{x}[ \phi]
+\frac{1}{4} (\eps^{1-\alpha} - \frac{d^2(x)}{\eps^{1-\alpha}})   (D^2\phi(x))_{11}  \right)
 \left(\eps^{1-\alpha} - (1-|\langle n(\bar x), v^\lambda  \rangle|) \eps^{\sigma}\right) .
\end{multline}
Since $0 \leq \eps^{1-\alpha}  -  (1- |\langle n(\bar x), v^\lambda  \rangle|) \eps^{\sigma}\leq \eps^{1-\alpha}$, we observe that 
\begin{align*}
\left| \frac{1}{4} (\eps^{1-\alpha} - \frac{d^2(x)}{\eps^{1-\alpha}})   (D^2\phi(x))_{11}
 \left(\eps^{1-\alpha}  -  (1- |\langle n(\bar x), v^\lambda  \rangle|) \eps^{\sigma}\right)  \right| 
 & \leq \frac{1}{4} \norm{D^2\phi(x)} (\eps^{2-2 \alpha} - d^2(x))  \\
 & \leq \frac{1}{2} \norm{D^2\phi(x)} (\eps^{1- \alpha} - d(x))\eps^{1- \alpha}.
\end{align*}
By plugging this inequality in \eqref{c1} and rearranging the terms, we obtain
\begin{align*}
(D\phi(x) - p_\text{opt}^M)  \cdot \Delta \hat x + & \norm{\Delta \hat x - \Delta x} M_\eps^{x}[ \phi]  \\
 & \leq (\eps^{1-\alpha} -d(x) )  \left\{
\frac{1}{2}  \left(1 - (1- |\langle n(\bar x), v^\lambda  \rangle|) \eps^{\sigma -1+\alpha}\right) M_\eps^{x}[\phi]
 + \frac{1}{2}  \norm{D^2\phi(x)} \eps^{1-\alpha} \right\} \\
& \leq \frac{1}{2} (\eps^{1-\alpha} -d(x))  ( M_\eps^{x}[\phi] +  \norm{D^2\phi(x)}  \eps^{1-\alpha}  )  
- \frac{1}{2}   \eps^{\sigma } M_\eps^{x}[\phi] .
 \end{align*}
Otherwise, if  $\Delta \hat x \cdot n(\bar x)\geq 0$, this move corresponds  to the sign $+$ in \eqref{def_move3} by \eqref{move3_proj_n}. We have
\begin{equation*}
\norm{\Delta \hat x-\eps^{1-\alpha} n(\bar x)}=\norm{-\eps^{\sigma}  n(\bar x) +\text{sgn}(\langle n(\bar x), v^\lambda  \rangle)\eps^{\sigma}v^\lambda } 
 = \sqrt{2}  \eps^{ \sigma} \sqrt{ 1 - |\langle n(\bar x), v^\lambda  \rangle|}
\leq \sqrt{2}  \eps^{ \sigma}.
 \end{equation*} 
By using Lemma \ref{boundary_bounce1},  we deduce from the previous inequality that, for $\eps$ small enough, 
the intermediate point $\hat x=x +\Delta \hat x$ is outside $\Omega$ and 
\begin{equation} \label{eq122}
\eps^{1-\alpha} - d(x)  - \sqrt{2}\eps^\sigma -2 C_1 \eps^{2-2\alpha}  \leq \norm{\Delta \hat x - \Delta x}, 
\end{equation} 
where $C_1$ is a certain constant depending on the principal curvatures of $\partial \Omega$ in a neighborhood of~$x$.  
By repeating the computations above, we find
\begin{align*}
(D\phi(x) - p_\text{opt}^M)_1  (\Delta \hat x)_1 
 & \leq \frac{1}{2} (\eps^{1-\alpha} -d(x) )  \left\{
  -   (1 - ( 1- |\langle n(\bar x), v^\lambda  \rangle|)\eps^{\sigma -1 +\alpha}) M_\eps^{x}[\phi]
 +   \norm{D^2\phi(x)} \eps^{1-\alpha} \right\}   \\ 
 & \leq \frac{1}{2} (\eps^{1-\alpha} -d(x) ) (-  M_\eps^{x}[\phi] + \norm{D^2\phi(x)} \eps^{1-\alpha} ).
\end{align*}
Recalling that $M_\eps^{x}[\phi]<0$, by combining \eqref{eq122} with the previous estimate, we are led to 
\begin{align*}
(D\phi(x) & - p_\text{opt}^M) \cdot \Delta \hat x  + \norm{\Delta \hat x - \Delta x} M_\eps^{x}[\phi]  \\
 & \leq \frac{1}{2} (\eps^{1-\alpha} -d(x) ) (-  M_\eps^{x}[\phi] + \norm{D^2\phi(x)} \eps^{1-\alpha} )
+\left(\eps^{1-\alpha} - d(x) -  \sqrt{2} \eps^{\sigma}- 2 C_1\eps^{2-2\alpha} \right)  M_\eps^{x}[\phi] \\ 
 & \leq \frac{1}{2}  (\eps^{1-\alpha} -d(x) )  ( M_\eps^{x}[\phi]  +  \norm{D^2\phi(x)} \eps^{1-\alpha} )
- \eps^{\sigma} M_\eps^{x}[\phi] ( \sqrt{2}  + 2 C_1\eps^{2-2\alpha - \sigma } ).
\end{align*}
 Putting together the two cases, the proof of the inequality given by \eqref{bonus_bdry_3} is complete. 
\end{proof}

Now we turn back to the analysis of case \ref{cas2}. 
Note that since $f$ is monotone in its last input
\begin{equation*}
f(t,x,z,p,\Gamma) \geq f(t,x,z,p, \Gamma_\text{opt} - \lambda I). 
\end{equation*}
The direct evaluation of the  second order  terms in $\Delta \hat x$ of \eqref{min_fun} gives 
\begin{align*}
\langle (  \Gamma_{\text{opt}}- \Gamma)   \Delta \hat x  ,  \Delta \hat x \rangle
& =  a_1^2 \langle (\Gamma_{\text{opt}}- \Gamma)    n(\bar x) ,  n(\bar x) \rangle
    + 2 a_1 b \langle (\Gamma_{\text{opt}}- \Gamma)    v^\lambda ,  n(\bar x) \rangle
    + b^2 \langle (\Gamma_{\text{opt}}- \Gamma)    v^\lambda , v^\lambda \rangle \\
 & \leq a_1^2 (\norm{\Gamma_{\text{opt}}}  + \norm{\Gamma} )+ 2 a_1 b \lambda  \langle  v^\lambda ,  n(\bar x) \rangle  + b^2 \lambda .
\end{align*}
With our choice for $\Delta \hat x$, we have $a_1 b \langle  v^\lambda ,  n(\bar x) \rangle \geq 0$. 
Hence, since $\lambda \leq 0$ in case \ref{cas2}, it follows that 
\begin{align*}
\langle (  \Gamma_{\text{opt}}- \Gamma)   \Delta \hat x  ,  \Delta \hat x \rangle
 & \leq  a_1^2 (\norm{\Gamma_{\text{opt}}}  + \norm{\Gamma} ) + b^2 \lambda 
\leq \eps^{2- 2\alpha}  (\norm{D^2\phi(x)}  + \eps^{-\gamma} ) +  \eps^{2 \sigma} \lambda .
\end{align*}
Choosing $\Delta \hat x$ as announced, using \eqref{cp_D2phi_gamma_opt} and \eqref{bonus_bdry_3} 
 and changing the sign $\pm$ in \eqref{def_move3} if necessary to make 
$(p_{\text{opt}} - p) \cdot \Delta \hat x\leq 0$,
\begin{align}
(D\phi(x)&  - p)  \cdot \Delta \hat x  + \norm{\Delta \hat x - \Delta x} M_\eps^{x}[\phi]   + \frac{1}{2} \langle D^2\phi(x) \Delta x ,\Delta x\rangle 
 - \frac{1}{2} \langle \Gamma \Delta \hat x ,\Delta \hat x\rangle - \eps^2 f \left(t,x, z,  p,\Gamma \right)  \notag \\
 & \leq
\frac{1}{2}  (\eps^{1-\alpha} -d(x) ) \left(M_\eps^{x}[\phi] +\frac{9}{2}\norm{D^2\phi(x) } \eps^{1-\alpha} \right)
 + \frac{1}{2} \eps^{2-2\alpha} (\norm{D^2\phi(x)}   + \eps^{-\gamma} ) - 4 \eps^{\sigma} M_\eps^{x}[\phi]  \notag  \\
 &  \phantom{\leq \frac{1}{2}  (\eps^{1-\alpha} -d(x))====================}
 + \frac{1}{2} \eps^{\sigma} \lambda - \eps^2 f(t,x,z,p, \Gamma_\text{opt}- \lambda I). \label{eq_lambda_cas3}
\end{align}
 Since $d(x) \leq \eps^{1-\alpha}-\eps^\rho$ in case \ref{cons_sub_cas4}, 
we deduce from the assumption \eqref{cond_mu_l} that
\begin{equation}\label{est_dM_sigma}
 \eps^{1-\alpha} -d(x) \geq \eps^{\rho} \gg    \eps^{\sigma} .
\end{equation}
Since $M_\eps^{x}[\phi] \leq - \eps^{1-\alpha - \kappa}$ 
 and $\eps^{2-2\alpha-\gamma}\ll \eps^{1-\alpha -\kappa + \rho}$ using \eqref{relation_gamma_nul_tilde},  we conclude by \eqref{est_dM_sigma} that
\begin{multline}\label{est_case2_line1}
 \frac{1}{2}  (\eps^{1-\alpha} -d(x) ) \Big(M_\eps^{x}[\phi]   +\frac{9}{2}\norm{D^2\phi(x) } \eps^{1-\alpha} \Big)
   + \frac{1}{2} \eps^{2-2\alpha} (\norm{D^2\phi(x)}   + \eps^{-\gamma} )- 4 \eps^{\sigma} M_\eps^{x}[\phi] \\
  \leq \frac{1}{4} (\eps^{1-\alpha} - d(x)) M_\eps^{x}[\phi] .
\end{multline}
It remains to control the terms in \eqref{eq_lambda_cas3} depending on $\lambda$.  
If $-1\leq \lambda \leq -\eps^\alpha $,  then $\eps^{2 \sigma} \lambda \leq - \eps^{2 \sigma + \alpha}$ 
 and  $f(t,x,z,p, \Gamma_\text{opt} - \lambda I)$ is bounded. So for such $\lambda$ we have
\begin{equation}\label{est_case2_lambda}
 \frac{1}{2} \eps^{2  \sigma} \lambda 
- \eps^2 f(t,x,z, p, \Gamma_\text{opt} - \lambda I) \leq - \frac{1}{2} \eps^{2 \sigma+ \alpha} +O(\eps^2).
\end{equation}
In this case, the right-hand side is $\ds \leq  - \eps^2 \min_{p\in B (p_\text{opt}^M,r)} f(t,x,z,p, \Gamma_\text{opt})$
when $\eps$ is sufficiently small since $\eps^{2  \sigma  + \alpha}\gg \eps^2$ by \eqref{cond_mu_l}.

To complete case \ref{cas2}, suppose $\lambda \leq  - 1$. 
 Then using the growth hypothesis \eqref{cont_growth_p_Gamma} and recalling that $p$ is near $p_\text{opt}$ we have 
\begin{equation}\label{est_case2_lambda_2}
 \frac{1}{2} \eps^{2  \sigma} \lambda  - \eps^2 f(t,x,z, p, D^2\phi(x) - \lambda I)
\leq -\frac{1}{2}  \eps^{2  \sigma} |\lambda| +C\eps^2 (1+|\lambda|^r). 
\end{equation}
Now notice that $|\lambda| \leq C(1+\norm{\Gamma}) \leq C\eps^{-\gamma}$. 
Since $\gamma(r-1)<2 - 2 \sigma$  by \eqref{cond_mu_l}, we have $\eps^{2  \sigma} |\lambda| \gg \eps^2 |\lambda|^r$. Therefore
\begin{equation*}
 -\frac{1}{2} \eps^{2 \sigma} |\lambda|  +C\eps^2|\lambda|^r \leq - \frac{1}{4} \eps^{2 \sigma} 
\leq - \eps^2 \min_{p\in B (p_\text{opt}^M,r)} f(t,x,z,p, \Gamma_\text{opt}), 
\end{equation*}
for $\eps$ small enough. Case \ref{cas2} is now complete.

Finally in case \ref{cas5},  we take $\Delta \hat x$ to be parallel to $p_\text{opt}^M - p$ with norm $\eps^{1-\alpha}$, and with the sign chosen such that
\begin{equation}\label{estc43_popt}
(p_\text{opt}^M-p)\cdot \Delta \hat x = -\eps^{1-\alpha} \norm{p_\text{opt}^M -p} 
\leq  - 3 (\eps^{1-\alpha} - d(x)) |M_\eps^{x}[\phi]|
\leq  - 3   \eps^{1-\alpha - \kappa + \rho }. 
\end{equation}
Estimating the other terms on the left-hand side of \eqref{eq_gene3}, some manipulations analogous to those made in Lemma \ref{lemma_sup_Ab} led us to
\begin{align*}
\left|(D\phi(x)  - p_\text{opt}^M)  \cdot \Delta \hat x + \norm{\Delta \hat x - \Delta x} M_\eps^{x}[ \phi] \right|
 & \leq \frac{1}{2}(\eps^{1-\alpha}- d(x))  \left( 3 |M_\eps^{x}[\phi]| + 4 \norm{D^2\phi(x)} \eps^{1-\alpha}   \right) .
\end{align*}
From \eqref{estc43_popt}, we deduce that
\begin{align*}
 (D\phi(x)  - p)  \cdot \Delta \hat x + \norm{\Delta \hat x - \Delta x} M_\eps^{x}[ \phi] 
\leq - \frac{1}{2}\eps^{1-\alpha} \norm{p_\text{opt}^M -p} +2 \norm{D^2\phi(x)} \eps^{2-2\alpha} .
\end{align*}
Estimating the other terms
\begin{equation} \label{est_gamma_opt_cas4}
| \langle (\Gamma_\text{opt}(x)-\Gamma)\Delta \hat x, \Delta \hat x \rangle| 
\leq 
(C+\norm{\Gamma})\norm{\Delta \hat x}^2 \leq C\eps^{-\gamma+2-2\alpha}, 
\end{equation}
 and 
\begin{equation} \label{est_case3_lambda}
\eps^2 |f(t,x,z,p, \Gamma)|\leq C\eps^2 (1+\norm{p}^q+\norm{\Gamma}^r) \leq C ( \eps^2+\eps^2 \norm{p}^q + \eps^{2-\gamma r} ) .
\end{equation}
Thus 
\begin{multline*}
(D\phi(x)  - p)  \cdot \Delta \hat x  + \norm{\Delta \hat x - \Delta x} M_\eps^{x}[\phi]    + \frac{1}{2} \langle D^2\phi(x) \Delta x ,\Delta x\rangle 
 - \frac{1}{2} \langle \Gamma \Delta \hat x ,\Delta \hat x\rangle - \eps^2 f \left(t,x, z,  p,\Gamma \right) \\ 
  \leq  -\frac{1}{2} \eps^{1-\alpha} \norm{p_\text{opt}^M - p} 
 +C\eps^2 \norm{p}^q +O(\eps^{2-2\alpha} + \eps^{-\gamma+2-2\alpha}+\eps^{2-\gamma r} ).
\end{multline*}
Since $ \eps^{1-\alpha} \norm{p_\text{opt}^M - p}\geq 2 \eps^{1-\alpha -\kappa + \rho } $ by using  \eqref{estc43_popt}, we obtain that
\begin{equation}\label{comp_puissance_eps_popt_c4}
\eps^{-\gamma +2-2\alpha} +\eps^{2-\gamma r} \ll \eps^{1-\alpha} \norm{p_\text{opt}^M-p}, 
\end{equation}
noticing that $\min(- \gamma +2 - 2 \alpha , 2-\gamma r) >1-\alpha - \kappa + \rho $ by using \eqref{cd_coeff_classiq} and \eqref{relation_gamma_nul_tilde}. 
Thus, by combining~\eqref{est_gamma_opt_cas4}--\eqref{comp_puissance_eps_popt_c4}, we conclude that 
\begin{multline*}
(D\phi(x)  - p)  \cdot \Delta \hat x  + \norm{\Delta \hat x - \Delta x} M_\eps^{x}[\phi]    + \frac{1}{2} \langle D^2\phi(x) \Delta x ,\Delta x\rangle 
 - \frac{1}{2} \langle \Gamma \Delta \hat x ,\Delta \hat x\rangle - \eps^2 f \left(t,x, z,  p,\Gamma \right) \\
  \leq -\frac{1}{2\sqrt{2}} \eps^{1-\alpha} \norm{ p_\text{opt}^M - p} +C\eps^2 \norm{p}^q . 
\end{multline*}
If $\norm{p}\leq 2 \norm{p_\text{opt}^M}$,  then $\eps^2 \norm{p}^q \ll  \eps^{1-\alpha -\kappa + \rho }$. 
If $\norm{p}\geq 2 \norm{p_\text{opt}^M}$, we infer from the condition on $\beta$ in~\eqref{cd_coeff_classiq} that
$\eps^{1-\alpha} \norm{p_\text{opt}^M -p} \sim \eps^{1-\alpha} \norm{p} \gg \eps^2 \norm{p}^q$. 
In either case the term $\eps^{1-\alpha} \norm{p_\text{opt}^M - p}$ dominates and we get
\begin{equation*}
( p_\text{opt}^M - p)\cdot \Delta \hat x  +\frac{1}{2} \langle (\Gamma_\text{opt}^M -\Gamma) \Delta \hat x, \Delta \hat x \rangle - \eps^2 f(t,x,z,p, \Gamma)
   \leq -\frac{1}{4} \eps^{1-\alpha} \norm{ p_\text{opt}^M - p}
\leq  \frac{3}{4} (\eps^{1-\alpha} - d(x)) M_\eps^{x}[\phi] .  
\end{equation*}
 The right-hand side of this inequality is certainly $\ds \leq \frac{1}{4} (\eps^{1-\alpha} - d(x)) M_\eps^{x}[\phi] 
 - \eps^2 \min_{p\in B (p_\text{opt}^M,r)} f(t,x,z,p, \Gamma_\text{opt})$ when $\eps$ is small.
Case \ref{cas5} is now complete which finishes the proof of Proposition~\ref{cons_new_sub}.

\subsection{Application to stability}
To prove stability in Section \ref{stability}, we will need some global variants of Propositions~\ref{cons_lower_bound} and \ref{cons_new_sub}. 
It is at this point that the uniformity of the constants in \eqref{loc_lip_p_Gamma}--\eqref{cont_growth_p_Gamma} in $x$ and $t$, 
and the growth condition \eqref{cont_growth_p_Gamma} intervene.
We must also take care of the Neumann boundary condition. Unlike the Dirichlet problem solved in \cite{kohns}, 
it is no longer appropriate to consider constant functions as test functions. 
For this reason, we are going to consider a $C_b^2(\ol \Omega)$-function $\psi$ such that
\begin{equation} \label{der_neumann_h}
 \frac{\partial \psi}{\partial n}=\norm{h}_{L^\infty}+1 \quad \text{ on } \partial \Omega.
\end{equation}
It is worth noticing that  $\psi$ has exactly the same properties as the function
introduced in Section~\ref{rules_el_game} for the game associated to the elliptic PDE with Neumann boundary condition.
If we take \mbox{$\psi=(\norm{h}_{L^\infty}+1)\psi_1$} where $\psi_1 \in C_b^2(\ol \Omega)$ such that 
$\dfrac{\partial \psi_1}{\partial n}=  1$ on $\partial \Omega$, it is clear that
\mbox{$ \norm{\psi}_{C_b^{2}(\ol \Omega)} =\norm{\psi_1}_{C_b^{2}(\ol \Omega)}  (1 + \norm{h}_{L^\infty})$}.

The next lemma is the crucial point to obtain stability in both parabolic and elliptic settings.
\begin{lemma} \label{est_mM_psi_bord}
If $\psi \in C_b^2(\ol \Omega)$ satisfies \eqref{der_neumann_h}, 
then there exists $\eps_0>0$ such that for all $\eps<\eps_0$ and for all $x\in \Omega(\eps^{1-\alpha})$, 
 \begin{equation} \label{est_psi_u} 
  - \norm{h}_{L^\infty}  - \norm{D\psi}_{L^\infty(\ol \Omega)}   \leq   M_\eps^x[ \psi] \leq  - \frac{1}{2} 
\quad \text{ and } \quad 
\frac{1}{2} \leq  m_\eps^x[  - \psi] \leq \norm{h}_{L^\infty} +\norm{D\psi}_{L^\infty(\ol \Omega)} .
 \end{equation}
\end{lemma}
\begin{proof} 
We shall demonstrate the bounds on  $M_\eps^x[ \psi]$ in  \eqref{est_psi_u}; the proof for  $m_\eps^x[- \psi]$ is entirely parallel. 
The left-hand side inequality is clear by the Cauchy-Schwarz inequality.  
 Let us consider $0<\eps<\eps_0$, where \mbox{$\ds \eps_0= \left(4  \norm{D^2\psi}_{L^\infty(\ol \Omega)}+2 \right)^{-\frac{1}{1-\alpha}} $}.
By the geometric relation \eqref{eq_11}, we observe that every move $\Delta x$ associated to the move $\Delta \hat x$ decided by Mark satisfies 
\begin{equation*}
   \norm{\Delta x}\leq 2\eps^{1-\alpha} \leq  \frac{1}{2  \norm{D^2\psi}_{L^\infty(\ol \Omega)} +1}.
\end{equation*} 
By the Cauchy-Schwarz inequality and using that $\psi \in C_b^{2}(\overline \Omega)$, we have 
\begin{align*}
h(x+\Delta x)  - D\psi(x) \cdot n(x+\Delta x)    
& \leq  \norm{h}_{L^\infty} - D\psi(x+\Delta x)\cdot n(x+\Delta x)+  (D\psi(x+\Delta x)-D\psi(x))\cdot n(x+\Delta x) \\
& \leq  -1+   \norm{D^2\psi}_{L^\infty(\ol \Omega)} \norm{\Delta x} \leq -  \frac{1}{2}.
\end{align*}
Then, by passing to the sup, we get the desired result. 
\end{proof}

\begin{lemma}\label{control_p_Gamma_opt_gene}
Let $\phi \in C_b^{2}(\ol \Omega)$.
Assume that $p_\text{opt}^m$ $p_\text{opt}^M$ and $\Gamma_\text{opt}$ are the strategies, associated to $\phi$, 
 respectively defined by \eqref{p_opt_m}, \eqref{p_opt_M} and \eqref{Gamma_opt}. 
Then, for all $x\in \Omega(\eps^{1-\alpha})$, we have 
\begin{equation*}
\max \left(\norm{ p_\text{opt}^m(x)}, \norm{ p_\text{opt}^M(x)} \right) \leq \frac{1}{2} \left( \norm{h}_{L^\infty}+3 \norm{D\phi}_{C_b^1(\ol \Omega)} \right)  
 \quad  \text{ and } \quad   \norm{ \Gamma_\text{opt}(x)} \leq  \frac{3}{2}  \norm{ D^2\phi}_{L^\infty(\ol \Omega)} .
\end{equation*}
\end{lemma}
\begin{proof} The proof being exactly the same for $p_\text{opt}^m$, it is sufficient to show the result for $p_\text{opt}^M$. 
By the triangle inequality and \eqref{p_opt_M}, we have
\begin{align*}
 \norm{p_\text{opt}^M(x)  - D\phi(x)} 
 & \leq \frac{1}{2}\left(1-\frac{ d(x)}{\eps^{1-\alpha}}\right) \left( | M_\eps^{x}[\phi] |
 + \frac{1}{2}\eps^{1-\alpha} \left( 1+ \frac{d(x)}{\eps^{1-\alpha}} \right) \norm{D^2\phi(x)} \right) 
 \\ &
 \leq \frac{1}{2} \left( | M_\eps^{x}[\phi] |  + \eps^{1-\alpha}  \norm{D^2\phi(x)} \right).
\end{align*}
Since $ M_\eps^{x}[\phi]$ is $\eps,x$-bounded by  $\norm{h}_{L^\infty} + \norm{D\phi}_{L^\infty(\ol \Omega)}  $, 
we deduce the desired inequality on $\norm{p_\text{opt}^M (x)}$.
Similarly, the estimate on $ \norm{ \Gamma_\text{opt}(x)}$ stems directly from \eqref{Gamma_opt} and the triangle inequality.
\end{proof}

In preparation for stability, we need to compute the action of $S_\eps$ on $\psi$. According to Lemma~\ref{est_mM_psi_bord}, 
only some cases proposed in Proposition \ref{cons_new_sub} must be considered. The next proposition gives the required estimates for $S_\eps$
concerning these cases. 
\begin{prop} \label{est_stab_sub_ord2} 
Let $f$ satisfy \eqref{ellipticity_f} and \eqref{loc_lip_p_Gamma}--\eqref{cont_growth_p_Gamma} 
 and assume $\alpha$, $\beta$, $\gamma$, $\rho$ 
 fulfill \eqref{condition_pas}--\eqref{cd_coeff_classiq} and \eqref{def_nul}.
Then for any $x$, $t$, $z$ and any $C_b^{2}(\ol \Omega)$-function $\phi$ defined near $x$, 
$S_\eps[x,t,z,\phi]$ being defined by \eqref{def_new_op}, we have
\begin{multline}\label{estsup_par_g}
 S_\eps[x,t,z, \phi] - \phi(x) \\
 \leq 
\begin{cases}
C \eps^2 (1+|z|) , &  \text{ if  } d(x) \geq \eps^{1-\alpha}, \\
3 \eps^{1-\alpha}   M_\eps^{x}[\phi]  +C \eps^2 (1+|z|), & 
\text{ if } d(x)\leq \eps^{1-\alpha} \text{ and } M_\eps^{x}[\phi] \geq \frac{4}{3} \norm{D^2\phi(x)} \eps^{1-\alpha} , \\
     C \eps^2 (1+|z|) , &  \text{ if  } \eps^{1-\alpha} - \eps^{\rho} \leq d(x) \leq \eps^{1-\alpha} 
                          \text{ and } M_\eps^{x}[\phi] \leq \frac{4}{3} \norm{D^2\phi(x)} \eps^{1-\alpha} , \\
 \frac{1}{4} \eps^{\rho}   M_\eps^{x}[\phi]  +C \eps^2 (1+|z|),  & 
\text{ if }d(x) \leq \eps^{1-\alpha} - \eps^{\rho} \text{ and }M_\eps^{x}[\phi] \leq  - \eps^{1-\alpha - \kappa} , \\
\end{cases}
\end{multline}
with a constant $C$ that depends on $\norm{D\phi}_{C_b^{1}(\ol \Omega)}+ \norm{h}_{L^\infty}$ but is independent of $x$, $t$ and $z$.

Moreover, if $d(x)\geq \eps^{1-\alpha}$, or  if $ d(x)\leq  \eps^{1-\alpha}$ 
and $m_\eps^{x}[\phi]  >\frac{1}{2} (3\eps^{1-\alpha}-d(x)) \norm{D^2\phi(x)}$,  then  
\begin{equation}\label{estinf_par_g}
- C\eps^2 (1+|z|) \leq   S_\eps[x,t,z, \phi] - \phi(x), 
 \end{equation}
with a constant $C$ that depends on $\norm{D\phi}_{C_b^{1} (\ol \Omega)}$ but is independent of $x$, $t$ and $z$.
\end{prop}

\begin{proof} The arguments in the different cases are the same as those given in the proof of Proposition~\ref{cons_new_sub}
but  we must pay attention to the uniformity of the constant. For the second part, 
since $f$ grows linearly by \eqref{lin_grow} and $\norm{(D\phi(x), D^2\phi(x))}\leq  \norm{D\phi}_{C_b^{1} (\ol \Omega)} $, we have
\begin{equation}\label{est_f_z_basik}
|f(t,x,z,D\phi(x), D^2\phi(x))|  \leq C(1+|z|), 
\end{equation}
with a constant $C$ that depends on $\norm{D\phi}_{C_b^{1} (\ol \Omega)}$  but is independent of $x$, $t$ and $z$. 
The lower bound 
\begin{equation*}
S_\eps[x,t,z,\phi] - \phi(x) \geq - \eps^2 f(x,t,z,D\phi(x), D^2\phi(x)) \geq -  C \eps^2 (1+|z|) 
\end{equation*}
is a consequence of Proposition \ref{cons_lower_bound} and \eqref{est_f_z_basik}. 

Similarly, since we know by Lemma \ref{control_p_Gamma_opt_gene} 
that $\max(\norm{p_\text{opt}^m(x)}, \norm{p_\text{opt}^M(x)})+ \norm{\Gamma_\text{opt}(x)}$
is uniformly bounded by $ \frac{1}{2}  \norm{h}_{L^\infty}+3 \norm{D\phi}_{C_b^{1}(\ol \Omega)}$, 
we get that 
\begin{equation}\label{est_p_g_opt_z}
\max(|f(t,x,z,p_\text{opt}^m(x), \Gamma_\text{opt}(x))|,  | f(t,x,z,p_\text{opt}^M(x), \Gamma_\text{opt}(x))|) \leq C(1+|z|), 
\end{equation}
with a constant $C$ that depends on $\norm{D\phi}_{C_b^{1}(\ol \Omega)}$ and $\norm{h}_{L^\infty}$ but is independent of $x$, $t$ and $z$.

We shall prove the estimate for the fourth alternative of \eqref{estsup_par_g}  by examining the proof 
of Proposition \ref{cons_new_sub} case \ref{cons_sub_cas4}, the proofs for the other alternatives being quite similar.
Since $f$ is locally Lipschitz by \eqref{loc_lip_p_Gamma}, 
\begin{equation*}
\min_{p \in B(p_\text{opt}^M,r)} f(t,x,z,p,  \Gamma_\text{opt}) \geq  f(t,x,z,p_\text{opt}^M(x), \Gamma_\text{opt}(x)) 
- C(1+|z|) \left( 1-\frac{d(x)}{\eps^{1-\alpha}}\right) \left( \norm{h}_{L^\infty}+\norm{D\phi}_{L^\infty(\ol \Omega)} \right), 
\end{equation*}
where $C$ depends only on $\norm{D\phi}_{C_b^{1}(\ol \Omega)}$ and $\norm{h}_{L^\infty}$ by the estimates on $p_\text{opt}^M$ and $\Gamma_\text{opt}$ given by 
Lemma~\ref{control_p_Gamma_opt_gene}. By using \eqref{est_p_g_opt_z}, we deduce   
that there exists a constant $C$ depending only on $\norm{D\phi}_{C_b^{1}(\ol \Omega)}$ and $\norm{h}_{L^\infty}$ such that
\begin{equation}\label{est_min_global}
\min_{p \in B(p_\text{opt}^M,r)} f(t,x,z,p, \Gamma_\text{opt}) \geq - C(1+|z|).
\end{equation}
In case \ref{cas1}, by combining \eqref{est_min_global} and the locally Lipschitz character  \eqref{loc_lip_p_Gamma} of $f$ on $\Gamma$, 
 the estimate \eqref{est_loc_lip_M_very_neg} gets replaced by
\begin{equation*}
 f(t,x,z,p,\Gamma) \geq - C(1+|z|) (1+\eps^\alpha),  
\end{equation*}
whence by \eqref{est_p_g_opt_z} there exists a constant $C$ depending on  $\norm{D\phi}_{C_b^{1} (\ol \Omega)}+\norm{h}_{L^\infty}$ such that
\begin{equation*}
-\eps^2 f(t,x,z,p,\Gamma)\leq C (1+|z|)\eps^2.
\end{equation*}
In case \ref{cas2}, since the domain satisfies both the uniform interior and exterior ball conditions, we notice that the constant $C_1$ 
corresponding to the curvature of the boundary (see Lemma \ref{boundary_bounce1}) is $x$-bounded.
This implies that the first order estimate \eqref{bonus_bdry_3} is valid  independently of $x$ for $\eps$ sufficiently small.
Thus, the estimate \eqref{est_case2_line1} is valid uniformly in $x$. 
Besides, the estimate \eqref{est_case2_lambda} gets replaced by
\begin{equation*}
 \frac{1}{2} \eps^{2  \sigma} \lambda - \eps^2 f(t,x,z, p, \Gamma_\text{opt}(x) - \lambda I) 
\leq - \frac{1}{2} \eps^{2 \sigma + \alpha} +C\eps^2(1+|z|) \norm{p} \norm{\Gamma_\text{opt} (x)- \lambda I}, 
\end{equation*}
where $C$ depends on $\norm{D\phi}_{C_b^{1} (\ol \Omega)}+\norm{h}_{L^\infty}$. 
We obtain an estimate of the desired form by dropping the first term and observing that $\lambda$ is bounded. 
In second half of case \ref{cas2} and in case \ref{cas5} we used the growth estimate \eqref{cont_growth_p_Gamma}; since $z$ 
enters linearly on the right-hand side of \eqref{cont_growth_p_Gamma}, the previous calculation still applies but we get an additional term 
of the form $C|z|\eps^2$ in \eqref{est_case2_lambda_2}--\eqref{est_case3_lambda}.
\end{proof}

The following corollary provides the key estimate for stability in the parabolic setting.

\begin{cor} \label{est_bar_neum} 
Let $f$ satisfy \eqref{ellipticity_f} and \eqref{loc_lip_p_Gamma}--\eqref{cont_growth_p_Gamma}  and 
 assume $\alpha$, $\beta$, $\gamma$  fulfill \eqref{condition_pas}--\eqref{cd_coeff_classiq}.
Then,  for any $x$, $t$, $z$ and $\psi \in C_b^2(\ol \Omega)$ satisfying \eqref{der_neumann_h},  we have
\begin{equation}\label{est_bar_neum_inf}
  S_\eps[x,t,z,\psi ] -\psi(x)   \leq  C(1+|z|) \eps^2    \quad \text{ and }  \quad
  S_\eps[x,t,z, - \psi] -(-\psi)(x)   \geq  - C(1+|z|) \eps^2,   
\end{equation}
with a constant $C$ that is independent of $x$, $t$, $z$ but depends on $\norm{D\psi}_{C_b^{1}(\ol \Omega)}$ and $\norm{h}_{L^\infty}$.
\end{cor}

\begin{proof}
 We shall prove the first estimate, the second follows exactly the same lines.
By applying Lemma~\ref{est_mM_psi_bord},  we have that  $M_\eps^{x}[ \psi] \leq  - \frac{1}{2}$ for all $x \in \Omega(\eps^{1-\alpha})$.
We introduce $\rho$ fulfilling  \eqref{def_nul}.
 By putting together the estimates obtained from \eqref{estinf_par_g} and the third alternative in \eqref{estsup_par_g}, 
we get that there exists a constant $C$ depending only on $\norm{D\psi}_{C_b^{1}(\ol \Omega)}$ and $\norm{h}_{L^\infty}$ such that
\begin{equation*} 
  S_\eps[x,t,z, \psi] - \psi(x)  \leq 
\begin{cases}  
 C \eps^2 (1+|z|), &\text{ if } d(x)\geq \eps^{1-\alpha} - \eps^{\rho }, \\
\frac{1}{4} \eps^{\rho}   M_\eps^{x}[\psi]  +C \eps^2 (1+|z|) , &\text{ if } d(x)\leq \eps^{1-\alpha} - \eps^{\rho }.
\end{cases}
\end{equation*}
Noticing that  $M_\eps^{x}[\psi]$ is negative, we get the proposed result.
 \end{proof}

\subsection{The elliptic case}

For the game corresponding to the stationary equation, 
we consider the operator $Q_{\eps}$ defined for any $x\in \ol \Omega$, $z\in \R$, and any continuous function 
$\phi$: $\ol \Omega \rightarrow \R$, by
\begin{multline}\label{oper_el}
 Q_{\eps}[x,z,\phi]=
\sup_{p,\Gamma} \inf_{\Delta \hat x} \left[   e^{ - \lambda \eps^2}  \phi(x+\Delta x) \right.  \\
 \left.  - \left( p\cdot \Delta \hat x  +\frac{1}{2}  \left\langle \Gamma \Delta \hat  x,\Delta \hat x \right\rangle
+\eps^2 f(x, z,p ,\Gamma) - \norm{\Delta \hat x - \Delta x} h(x+\Delta x) \right) \right] , 
\end{multline}
with the usual conventions that $p$, $\Gamma$ and $\Delta \hat x$ are constrained by \eqref{p_beta_gamma_new} and  \eqref{moving1_new}
and that $\Delta x$ is determined by \eqref{exp_delta_x}.
We can easily check that the operator $Q_\eps$ is still monotone but its action on shifted functions by  a constant is described by the following way: 
for all function $\phi \in C(\overline \Omega)$ and $c\in \R$, 
\begin{equation}\label{action_on_cte_el}
Q_{\eps}  \left[x,z, c + \phi \right]=e^{-\lambda \eps^2}c + Q_{\eps}  \left[x,z, \phi \right] .
\end{equation}
The dynamic programming inequalities \eqref{dyn_prog_ineq_sub_new_el}--\eqref{dyn_prog_ineq_super_new_el} 
can be concisely written as
\begin{equation*}
u^\eps(x)\leq Q_\eps[x, u^\eps(x), u^\eps] \quad \text{ and } \quad 
v^\eps(x)\geq Q_\eps[x, v^\eps(x), v^\eps]. 
\end{equation*}
In the elliptic setting, we can formally derive the PDE  by following the same lines as for the parabolic framework. 
We keep the optimal strategies  $p_\text{opt}^m$, $p_\text{opt}^M$ and $\Gamma_\text{opt}$ for Helen, 
defined by \eqref{p_opt_m}, \eqref{p_opt_M} and \eqref{Gamma_opt} in an orthonormal basis  $\mathcal{B}=(e_1=n(\bar x), e_2,\cdots, e_N)$.
The next proposition is the elliptic analogue of Propositions~\ref{cons_lower_bound} and \ref{cons_new_sub}.
It establishes the consistency estimates for $Q_\eps$ defined by \eqref{oper_el}.

\begin{prop} \label{cons_new_sub_el}
Let $f$ satisfy \eqref{ellipticity_f} and \eqref{loc_lip_p_Gamma_el}--\eqref{cont_growth_p_Gamma_el} and assume $\alpha$, $\beta$, $\gamma$ and $\rho$ 
 fulfill \eqref{condition_pas}--\eqref{cd_coeff_classiq} and \eqref{def_nul}.  
Let  $p_\text{opt}^m$, $p_\text{opt}^M$ and $\Gamma_\text{opt}$ be respectively defined in the orthonormal basis  $\mathcal{B}=(e_1=~n(\bar x), e_2,\cdots, e_N)$ by 
\eqref{p_opt_m}--\eqref{Gamma_opt}. 
For any $x$, $z$ and any smooth function $\phi$ defined near $x$, we distinguish two cases for the lower bound estimate:
\begin{enumerate}[label=\textup{ }{\roman*.}\textup{ },ref=({\roman*})] 
 \item Big bonus: if $ d(x)\geq  \eps^{1-\alpha}$ or $ m_\eps^x[\phi] >\frac{1}{2} (3\eps^{1-\alpha}-d(x)) \norm{D^2\phi(x)}$,  then  
\begin{equation}\label{est_super_el}
  - \eps^2 (f(x,z, D\phi(x), D^2\phi(x))+\lambda \phi(x)) \leq   Q_\eps[x,z,\phi] - \phi(x)  .   
 \end{equation}
\item Penalty or small bonus: if $d(x)\leq  \eps^{1-\alpha}$ and $ m_\eps^x[\phi] \leq \frac{1}{2} (3\eps^{1-\alpha}-d(x)) \norm{D^2\phi(x)}$,  then   
\begin{equation*}
 \frac{1}{2}   (\eps^{1-\alpha} - d(x)) \left(  s  m_\eps^x[\phi]  - 4 \norm{D^2\phi(x)} \eps^{1-\alpha}  \right)
- \eps^2 (f(x,z, p_{\text{opt}}^m(x), \Gamma_\text{opt}(x)) +\lambda \phi(x)) \leq Q_\eps[x,z,\phi] - \phi(x),    
 \end{equation*}
where $s=-1$ if   $m_\eps^x[\phi]\geq 0$ and $s=3$ if   $m_\eps^x[\phi]< 0$.
\end{enumerate}
For the upper bound estimate, we distinguish four cases:
\begin{enumerate}[label=\textup{ }{\roman*.}\textup{ },ref=({\roman*})] 
 \item Big bonus: if $d(x) \leq \eps^{1-\alpha}$ and $M_\eps^x[\phi] >\frac{4}{3} \norm{D^2\phi(x)}\eps^{1-\alpha}$,  then    
\begin{equation*}
 Q_\eps[x,z,\phi]-\phi(x) \leq 3(\eps^{1-\alpha}-d(x))M_\eps^x[\phi] -\eps^2 \left(f(x,z,p_\text{opt}^M(x),\Gamma_\text{opt}(x))+\lambda \phi(x)\right) +o(\eps^2).
 \end{equation*}
  \item \label{cas_el_far} Far from the boundary with a small bonus: if $\eps^{1-\alpha}-\eps^{\rho} \leq d(x)\leq \eps^{1-\alpha}$ 
and  $M_\eps^x[\phi] \leq \frac{4}{3} \norm{D^2\phi(x)}\eps^{1-\alpha}$, or if $d(x)\geq \eps^{1-\alpha}$, then   
\begin{equation}\label{close_bdary_Mneg_el}
 Q_\eps[x,z,\phi] - \phi(x) \leq  -\eps^2 \left(f(x,z, D\phi(x), D^2\phi(x))+\lambda \phi(x) \right) +o(\eps^2).
 \end{equation}
 \item Close to the boundary with a small bonus/penalty: if $d(x)\leq \eps^{1-\alpha} - \eps^{\rho}$ 
 and $- \eps^{1-\alpha -\kappa} \leq M_\eps^x[\phi] \leq \frac{4}{3} \norm{D^2\phi(x)}\eps^{1-\alpha} $,  then
\begin{equation*}
 Q_\eps[x,z,\phi] - \phi(x) \leq  -\eps^2 \left(f(x,z, D\phi(x), D^2\phi(x)+C_1I)+\lambda \phi(x) \right) +o(\eps^2), 
 \end{equation*}
with $C_1=\frac{20}{3} \norm{D^2\phi(x)} \left(1-\frac{d(x)}{\eps^{1-\alpha}}\right)$.
 \item \label{cas_el_tres_neg} Close to the boundary with a big bonus: if $d(x)\leq \eps^{1-\alpha}-\eps^{\rho}$ 
 and $M_\eps^x[\phi] \leq - \eps^{1-\alpha -\kappa}$,  then
\begin{equation}\label{close_bdary_Mneg_big}
Q_\eps[x,z,\phi] - \phi(x) \leq  \frac{1}{4} (\eps^{1-\alpha} - d(x)) M_\eps^x[\phi]  
-\eps^2 \left( \min_{p\in B(p_\text{opt}^M(x),r)} f(x,z,p, \Gamma_\text{opt}(x))+\lambda \phi(x)\right) +o(\eps^2), 
 \end{equation}
with $r$ defined by $r=3 \left(1-\frac{d(x)}{\eps^{1-\alpha}} \right) |M_\eps^x[\phi]|$.
\end{enumerate}
Moreover the implicit constants in the error term are uniform as $x$ and $z$ range over a compact subset of $\ol \Omega \times \R$.
\end{prop}

\begin{proof}
The arguments are entirely parallel to the proofs of Propositions \ref{cons_lower_bound} and \ref{cons_new_sub}.
\end{proof}

For stability we will need a variant of the preceding lemma. This is where we use the hypothesis~\eqref{est_f_el_neu} on the $z$-dependence of $f$. 

\begin{lemma}\label{estimate_el_psi_u}
Let $f$ satisfy \eqref{ellipticity_f} and \eqref{loc_lip_p_Gamma_el}--\eqref{cont_growth_p_Gamma_el}
 and assume as always that $\alpha$, $\beta$, $\gamma$ satisfy \eqref{condition_pas}--\eqref{cd_coeff_classiq}. 
Let $\psi \in C_b^2 (\ol \Omega)$ satisfy \eqref{def_psi_h_el}.
Fix $M$ and $m$ two positive constants such that $ m + 2 \norm{\psi}_{L^\infty(\overline \Omega)} \leq M$. 
Then, there exists $C_\ast=C_\ast (\norm{D\psi}_{C_b^1(\ol \Omega)}, \norm{h}_{L^\infty})$ 
such that for any $|z|\leq M$ and any $x\in \overline \Omega$, we have
\begin{equation*}
Q_\eps \left[x,z, m +\psi \right] - \left(m + \psi(x) \right) \leq \eps^2 \left(1+ (\lambda -\eta)|z|+C_\ast \right) - \lambda \eps^2 \left(m+ \psi(x)\right)   , 
\end{equation*}
and 
\begin{equation*}
Q_\eps \left[x,z,- m- \psi \right] - \left(- m -\psi(x)\right) \geq -\eps^2 \left(1+ (\lambda -\eta)|z|+C_\ast \right) - \lambda \eps^2 \left(-m-\psi (x)\right), 
\end{equation*}
for all sufficiently small $\eps$ (the smallness condition on $\eps$ depends on $M$, but not on $x$). 

Moreover, if $\phi \in C_b^2 (\ol \Omega)$,  
then there exists $C=C(M, \norm{D\phi}_{C_b^1(\ol \Omega)}, \norm{h}_{L^\infty})$ such that for any $|z|\leq M$ and any $x\in \overline \Omega$ such that 
$d(x)\leq \eps^{1-\alpha}-\eps^{\rho}$  and $M_\eps^x[\phi] \leq - \eps^{1-\alpha -\kappa}$, 
\begin{equation}\label{est_phi_ell_cas_neg_glolocal}
Q_\eps[x,z, \phi] - \phi(x) \leq 
\frac{1}{4} \left( \eps^{1-\alpha} - d(x) \right) M_\eps^x[\phi] +C\eps^2  - \lambda \eps^2 \phi(x) , 
\end{equation}
for all sufficiently small $\eps$ (the smallness condition on $\eps$ depends on $M$, but not on $x$). 
\end{lemma}
\begin{proof} We shall prove the first inequality, the proof of the second being entirely parallel. 
The assumption $|z| \leq M$ ensures that the constants in \eqref{loc_lip_p_Gamma_el} and \eqref{cont_growth_p_Gamma_el} are uniform. 
Then the implicit constants in the error terms of \eqref{close_bdary_Mneg_el} and \eqref{close_bdary_Mneg_big}
are $x$,$z$-uniform for $\eps$ small enough, and the smallness condition depends only on $M$. 
Since $ m + 2 \norm{ \psi}_{L^\infty(\overline \Omega)}\leq  M$ we can use the dynamic programming inequalities
\eqref{dyn_prog_ineq_sub_new_el}--\eqref{dyn_prog_ineq_super_new_el}. 
First of all, by the action of $Q_\eps$ on constant functions provided by \eqref{action_on_cte_el}, we have
\begin{equation*}
Q_\eps[x,z, m +\psi] - (m + \psi(x))=(e^{-\lambda \eps^2} - 1)m + Q_\eps[x,z, \psi] -  \psi(x), 
\end{equation*}
and noticing that $e^{-\lambda \eps^2}m=(1-\lambda \eps^2)m +O(\eps^4m)$, it is sufficient to get the estimate corresponding to $m=0$.
By Lemma~\ref{est_mM_psi_bord}, we observe that every $x\in \Omega(\eps^{1-\alpha})$ satisfies $M_\eps^{x}[ \psi] \leq  - \frac{1}{2}$.
We now need  to distinguish two cases according to the distance to the boundary by introducing $\rho$ fulfilling \eqref{def_nul}. 
If $x\in \ol \Omega$ such that $d(x) \geq \eps^{1-\alpha} - \eps^{\rho}$, 
since $\norm{(D\psi(x), D^2 \psi(x))} \leq  K_1= \norm{D\psi}_{C_b^1(\ol \Omega)}$,  
we deduce by assumption~\eqref{est_f_el_neu} on $f$ that there exists $C_{K_1}^\ast$ such that for all $x$ we have
\begin{equation*}
|f(x,z, D\psi(x), D^2 \psi(x))| \leq (\lambda -\eta)|z|+C_{K_1}^\ast,    
\end{equation*}
which gives by \eqref{close_bdary_Mneg_el} that for all $x\in \ol \Omega$ such that $d(x)\geq \eps^{1-\alpha}$, 
\begin{equation} \label{est_psi_int} 
Q_\eps[x,z,\psi]  - \psi(x)  \leq  \eps^2 \left((\lambda -\eta)|z|+C_{K_1}^\ast   \right) - \lambda \eps^2 \psi(x)+o(\eps^2). 
\end{equation}  
If $x\in \ol \Omega$ such that $d(x) \leq \eps^{1-\alpha} - \eps^{\rho}$, combining the triangle inequality with
the inequalities given by Lemma~\ref{control_p_Gamma_opt_gene} gives that, 
for all $ p\in B \left(p_\text{opt}^M(x),r \right)$ with $ r=3\left(1-\frac{d(x)}{\eps^{1-\alpha}}\right) |M_\eps^x[\psi]|$, 
\begin{equation*}
\norm{(p, \Gamma_\text{opt}(x) )} \leq \norm{p_\text{opt}^M(x)}_{L^\infty}+r+\norm{\Gamma_\text{opt}(x)}_{L^\infty} 
\leq K_2=   \frac{7}{2}  \norm{h}_{L^\infty}+6 \norm{D\psi}_{C_b^1(\ol \Omega)}, 
\end{equation*}
since $ M_\eps^{x}[\psi]$ is $\eps,x$-bounded by  $\norm{h}_{L^\infty} + \norm{D\psi}_{L^\infty}  $. 
The assumption~\eqref{est_f_el_neu} on $f$ yields that there exists $C_{K_2}^\ast$ such that, 
\begin{equation} \label{est_min_global_r}
\left|\min_{p\in B(p_\text{opt}^M(x),r)} f(x,z,p, \Gamma_\text{opt}(x) ) \right|
 \leq (\lambda -\eta)|z|+C_{K_2}^\ast , 
\end{equation}
By using this inequality in \eqref{close_bdary_Mneg_big} and recalling that $M_\eps^x[\psi]\leq -\frac{1}{2}$,   we conclude that, 
for all $x\in \overline \Omega$ such that $d(x) \leq \eps^{1-\alpha} - \eps^{\rho}$, 
\begin{equation} \label{est_psi_ext} 
Q_\eps[x,z,\psi] - \psi(x) \leq \eps^2 \left( (\lambda -\eta)|z|+C_{K_2}^\ast   \right) - \lambda \eps^2 \psi(x) +o(\eps^2). 
\end{equation}  
By comparing \eqref{est_psi_int} and \eqref{est_psi_ext} we get the desired result by taking $C_\ast=\max(C^\ast_{K_1},C^\ast_{K_2})$ . 

To prove the third inequality, it is sufficient to replace the assumption \eqref{est_f_el_neu}  by \eqref{cont_growth_p_Gamma_el} in the previous estimates. 
For instance, instead of \eqref{est_min_global_r}, there exists a constant $C$ depending only on $M$, $\norm{h}_{L^\infty}$, 
and $\norm{D\phi}_{C_b^1(\ol \Omega)}$ such that $\ds \left|\min_{p\in B(p_\text{opt}^M(x),r)} f(x,z,p, \Gamma_\text{opt}(x) ) \right| \leq C$. 
The rest of the proof remains unchanged.
\end{proof}

\section{Stability}  \label{stability}

In the time-dependent setting, we showed in Section~\ref{cv_par_case} that if $v^\eps$ and $u^\eps$ remain bounded as $\eps \rightarrow 0$ then 
$\udl v$ is a supersolution and $\bar u$ is a subsolution. The argument was local, using mainly the consistency of the game as a numerical scheme. 
It remains to prove that  $v^\eps$ and $u^\eps$ are indeed bounded; this is achieved in Section \ref{s_stability_para}.

For the stationary setting, we must do more. Even the existence of $U^\eps(x,z)$ remains to be proved. 
We also need to show that the associated functions $u^\eps$ and $v^\eps$ are bounded, away from $M$, so that we can apply the dynamic programming
 inequalities at each $x \in \ol \Omega $. These goals will be achieved in Section \ref{s_stability_ell}, provided the parameters $M$ and $m$ satisfy 
(i) $m=M-1 - 2 \norm{\psi}_{L^\infty}$ and (ii) $M$ is sufficiently large. We also 
show in Section~\ref{s_stability_ell} that if $f$ is a nondecreasing function on $z$ then $U^\eps$ is strictly decreasing on $z$. 
As a consequence, this result implies that $\udl v \leq \bar u$,
 allowing us to conclude that $\udl v = \bar u$ is the unique viscosity solution if the boundary value problem has a comparison principle.

\subsection{The parabolic case} 
\label{s_stability_para}

To obtain stability, we are going to consider one more time a $C^2_b(\ol \Omega)$-function $\psi$ such that 
$\frac{\partial \psi}{\partial n }=\norm{h}_{L^\infty}+1$ in order to take care of the Neumann boundary condition.

\begin{prop} \label{stability_para_prop}
Assume the hypotheses of Propositions \ref{cons_lower_bound} and \ref{cons_new_sub} hold,
 and suppose furthermore that the final-time data are uniformly bounded: 
\begin{equation*}
 |g(x)| \leq B \quad \text{ for all } x\in \ol \Omega.
\end{equation*}
 Then there exists a constant $s=s(\norm{\psi}_{C_b^2(\ol \Omega)})$, independent of $\eps$,  such that 
\begin{equation*} 
u^{\eps}(x,t) \leq  (B+\norm{\psi}_{L^\infty(\ol \Omega)}) s^{T-t} +\psi(x)  \quad \text{ for all } x\in \ol \Omega,
\end{equation*}
and
\begin{equation*} 
v^{\eps}(x,t) \geq  - (B+\norm{\psi}_{L^\infty(\ol \Omega)}) s^{T-t}  - \psi(x) \quad \text{ for all } x\in \ol \Omega,
\end{equation*}
for every $t<T$.
\end{prop}
\begin{proof}
We shall demonstrate the lower bound on $v^{\eps}$; the proof of the upper bound on $u^\eps$ is entirely parallel. 
The argument proceeds backward in time $t_k=T-k \eps^2$. At $k=0$, we have a uniform bound $v^{\eps}(x,T)=g(x)\geq  - B$  by hypothesis, 
and we may assume without loss of generality that $B\geq  1$. Since  $\psi$ is bounded on $\ol \Omega$, we can suppose that
\begin{displaymath}
v^{\eps}(x,T)=g(x)\geq  - B_0 - \psi(x), 
\end{displaymath}
where $B_0=  B + \norm{\psi}_{L^\infty(\ol \Omega)}$. Now suppose that for fixed $k\geq 0$ we already know a bound $v^\eps(\cdot,t_k)\geq - B_k -\psi$. 
By the dynamic programming inequality 
\eqref{dyn_prog_ineq_sub_new},  we have
\begin{equation*}
v^{\eps}(x,t_k-\eps^2)  \geq S_{\eps}  \left[x,t,v^{\eps}(x,t_k-\eps^2), v^{\eps}(.,t_k) \right] . 
\end{equation*}
Since $S_{\eps}$ is monotone in its last argument, we have
\begin{equation*}
v^{\eps}(x,t_k-\eps^2)  \geq S_{\eps}  \left[x,t,v^{\eps}(x,t_k-\eps^2),  -  B_k  - \psi \right] .
\end{equation*}
By applying successively \eqref{action_on_cte} and Corollary~\ref{est_bar_neum},  we deduce that  
\begin{align*}
S_{\eps}  \left[x,t,v^{\eps}(x,t_k-\eps^2), - B_k  - \psi \right]  & = - B_k + S_{\eps}  \left[x,t,v^{\eps}(x,t_k-\eps^2),  -  \psi \right]  \\
 & \geq  - B_k - \psi(x) - C(1+|v^{\eps}(x,t_k-\eps^2)|)\eps^2 , 
\end{align*}
where $C$ depends only on $\norm{D\psi}_{C_b^1(\ol \Omega)}$. 
If $v^{\eps}(x,t_k-\eps^2) \geq 0$, then it is over (recall we are looking for a lower bound $ - B_{k+1}\leq -1$).
Otherwise, we have
\begin{equation*}
(1-C\eps^2)v^{\eps}(x,t_k-\eps^2)  \geq  - B_k - C\eps^2  - \psi(x). 
\end{equation*}
By dividing by $1-C\eps^2$, we get
\begin{equation*}
v^{\eps}(x,t_k-\eps^2)  \geq  -  \frac{B_k + C\eps^2}{1-C\eps^2}  - \frac{1}{1-C\eps^2}\psi(x) 
=-  \frac{B_k + C\eps^2 (1+ \psi(x)) }{1 - C\eps^2}  -  \psi(x).  
\end{equation*}
Then,  by setting 
$\ds B_{k+1}= \frac{B_k + C(1+ \norm{\psi}_{L^\infty(\ol \Omega)}) \eps^2 }{1 - C\eps^2},$ 
we obtain 
\begin{equation*}
v^{\eps}(x,t_k-\eps^2) \geq -B_{k+1} - \psi(x). 
\end{equation*}
As it is clear that
$\ds B_{k+1}\leq B_k \frac{1 + C(1+ \norm{\psi}_{L^\infty(\ol \Omega)}) \eps^2}{1 - C\eps^2}$, 
we deduce that
$v^{\eps}(x,T-k\eps^2)\geq \tilde B_k - \psi(x)$ for all $k$ with
\begin{equation*}
\tilde B_k=B_0 \left( \frac{1 +C(1+ \norm{\psi}_{L^\infty(\ol \Omega)})\eps^2}{1 - C\eps^2}  \right)^k.
\end{equation*}
Since $k=(T-t)/\eps^2$ and recalling that $B_0=  B + \norm{\psi}_{L^\infty(\ol \Omega)}$, we have shown that
\begin{equation*}
v_{\eps}(x,t)\geq  - (B + \norm{\psi}_{L^\infty(\ol \Omega)}) s_\eps^{T-t} - \psi(x)  
\end{equation*}
with
\begin{equation*}
s_{\eps}=\left( \frac{1 +  C(1+ \norm{\psi}_{L^\infty(\ol \Omega)}) \eps^2}{1  -  C\eps^2}  \right)^{1/\eps^2}.
\end{equation*}
Since $s_{\eps}$ has a finite limit as $\eps\rightarrow 0$ we obtain a bound on $v^\eps$ of the desired form. 
\end{proof}

\begin{remark} By following the construction of the elliptic game we can  take 
 $\psi=(\norm{h}_{L^\infty}+1)\psi_1$ where $\psi_1$ is defined by \eqref{def_fonction_bord}.
In that case, $\norm{D\psi}_{C_b^1(\ol \Omega)}= \norm{D\psi_1}_{C_b^1(\ol \Omega)} (1+\norm{h}_{L^\infty})$.
This expression can be compared for a $C^{2,\alpha}$-domain to the estimate given by Remark~\ref{est_schauder} provided 
by the Schauder theory for which $\norm{D\psi_1}_{C_b^1(\ol \Omega)}$ plays the role of the constant $C_\Omega$ depending only on the domain. 
\end{remark}

\subsection{The elliptic case}
\label{s_stability_ell}

We shall assume throughout this section that the parameters $M$ and $m$ controlling the termination of the game are related by 
$m=M-1- 2 \norm{\psi}_{L^\infty(\ol \Omega)}$; 
in addition, we need to assume $M$ is sufficiently large. 
Our plan is to show, using a fixed point argument, the existence of a function $U^\eps(x,z)$ (defined for all $x\in \ol \Omega$ and $|z|<M$) 
satisfying \eqref{dpp_U_el} and also
\begin{equation}\label{bounds_U_eps}
 - z - \chi(x) \leq U^\eps(x,z) \leq  - z+ \chi(x).
\end{equation}
This implies that $U^\eps(x,z)<0$ when $z>\chi(x)$, and $U^\eps(x,z)>0$ when $z<- \chi(x)$. 
Recalling the definitions of $u^\eps$ and $v^\eps$, it follows from \eqref{def_subelr}--\eqref{def_supelr} that
\begin{equation} \label{bound_u_v_eps}
 |v^\eps(x)|  \leq \chi(x) ,\quad   |u^\eps(x) | \leq \chi(x), 
\end{equation}
for all $x \in \ol  \Omega$. 
It is convenient to work with $V^\eps(x,z) =U^\eps(x,z) + z$ rather than $U^\eps$, since this turns~\eqref{bounds_U_eps} into  
\begin{equation*}
|V^\eps(x,z)|  \leq \chi(x),
\end{equation*}
whose right-hand side is not constant. The dynamic programming principle \eqref{dpp_U_el} for $U^\eps$ is equivalent (after a bit of manipulation) 
to the statement that 
$V^\eps$ is a fixed point of the mapping $\phi(\cdot, \cdot) \mapsto R_\eps[\cdot,\cdot, \phi]$ where the operator $R_\eps$ 
is defined 
for any $L^\infty$-function $\phi$ defined on $\ol \Omega \times (-M,M)$ by 
\begin{equation}\label{op_fixed_pt}
R^\eps[x,z,\phi]= \sup_{p, \Gamma}\inf_{\Delta \hat x }
 \begin{cases}
 e^{-\lambda \eps^2} \phi(x',z') - \delta, & \text{if } |z'|<  M,   \\ 
-  \chi(x),                                & \text{if }  z'\geq  M, \\
   \chi(x),                                & \text{if }  z'\leq -M.
\end{cases}
\end{equation}
 where $x'=x+\Delta x$ and $z'=e^{\lambda \eps^2}(z+\delta)$, with $\delta$ defined as in \eqref{Helen_loss_el}. 
Here $p$, $\Gamma$ and $\Delta \hat x$ are constrained as usual by \eqref{p_beta_gamma_new}--\eqref{moving1_new}.

We shall identify $V^\eps$ as the unique fixed point of the mapping $\phi(\cdot, \cdot) \mapsto R_\eps[\cdot,\cdot, \phi]$ in $F_\chi$ defined by
\begin{equation} \label{def_F_chi}
 F_\chi= \left\{ \phi \in L^\infty \left(\ol \Omega \times \left(-M,M \right)\right) : \forall (x,z)\in \ol \Omega \times (-M,M), |\phi(x,z)| \leq \chi(x) \right\}.
\end{equation}

\begin{lemma}\label{pt_fix_estimate}
Let $f$ satisfy  \eqref{ellipticity_f} and \eqref{loc_lip_p_Gamma_el}--\eqref{cont_growth_p_Gamma_el}
and assume as always that $\alpha$, $\beta$, $\gamma$ fulfill \eqref{condition_pas}--\eqref{cd_coeff_classiq} and 
that $\Omega$ is a $C^2$-domain satisfying both the uniform interior and exterior ball conditions. 
 Then, there exists $M_0>0$ such that for all two positive constants $m$ and $M> M_0$ satisfying $m + 2\norm{ \psi}_{L^\infty(\overline \Omega)} = M-1$,
 for any $|z|\leq M$ and any $x\in \overline \Omega$, we have
 \begin{equation*}
Q_\eps[x,z,\chi]  \leq   \chi(x) \quad \text{ and } \quad   Q_\eps[x,z, - \chi]  \geq   -\chi(x).
 \end{equation*}
\end{lemma}
\begin{proof}
 We are going to establish the upper estimate for $\chi$. By Lemma \ref{estimate_el_psi_u}, we deduce that 
\begin{equation*}
Q_\eps[x,z,\chi]  -  \chi(x)  \leq \eps^2 \Big(1+ (\lambda -\eta)|z|+C_\ast \Big) - \lambda \eps^2 (m+ \norm{ \psi}_{L^\infty(\overline \Omega)} + \psi(x)) .
\end{equation*}
Since  $m + 2\norm{ \psi}_{L^\infty(\overline \Omega)} = M-1$ and $|z|\leq M$, we compute
\begin{equation*} 
Q_\eps[x,z,\chi]  -  \chi(x) 
 \leq \eps^2 \Big(1+ (\lambda -\eta)M +C_\ast \Big) - \lambda \eps^2 \left(M-1 - \norm{ \psi}_{L^\infty(\overline \Omega)} +  \psi(x) \right). 
\end{equation*}
By rearranging the terms, we obtain that
\begin{equation*} 
Q_\eps[x,z,\chi]  -  \chi(x) \leq  \eps^2 \left(1+  \lambda (1+2\norm{ \psi}_{L^\infty(\overline \Omega)}) + C_\ast -\eta M   \right) .
\end{equation*}
We can choose $M$ large enough such that the right-hand side is negative. It suffices to take 
\begin{equation*} 
M> M_0:=\frac{1}{\eta} \left( 1 +  \lambda  (1+2\norm{ \psi}_{L^\infty(\overline \Omega)})  +C_\ast \right). 
\end{equation*}
The case for $Q_\eps[x,z, - \chi]  \geq    - \chi(x)$ is analogous.
\end{proof}

\begin{prop} \label{stability_el}
Assume the hypotheses of Lemma \ref{pt_fix_estimate} hold. Suppose further that $m=M - 1-2 \norm{\psi}_{L^\infty(\overline \Omega)}$. 
Then for all sufficiently small $\eps$, the map 
$\phi(\cdot,\cdot) \mapsto R_\eps [\cdot, \cdot, \phi]$
is a contraction in the $L^\infty$-norm, which preserves $F_\chi$. In particular, it has a unique fixed point,
which has $L^\infty$-norm at most $m+2 \norm{\psi}_{L^\infty(\overline \Omega)}$.
\end{prop}
\begin{proof}
By the arguments already used in \cite[Proposition 5.2]{kohns}, the map is a contraction for any $\eps$ (this part of the proof works for any $M$). More precisely, 
if $\phi_i$, $i=1,2$ are two $L^\infty$-functions defined on $\ol \Omega \times (-M,M)$ to $\R$, then 
$\ds \norm{R_\eps [\cdot ,\cdot , \phi_1] - R_\eps [\cdot, \cdot, \phi_2]}_{L^\infty} \leq e^{-\lambda \eps^2} \norm{ \phi_1 -  \phi_2}_{L^\infty}$.

Now we prove that if $M$ is large enough and $m + 2\norm{ \psi}_{L^\infty(\overline \Omega)} =  M-1$,
the map preserves the ball $F_\chi$ defined by \eqref{def_F_chi}.  
 Since  $R_\eps[x,z, \phi]$ is monotone in its last argument,  it suffices to show that
\begin{equation} \label{ineq_chi}
 R_\eps[x,z, \chi] \leq \chi(x)  \quad  \text{ and } \quad   R_\eps[x,z, -\chi] \geq   - \chi(x).
\end{equation}
For the first inequality of \eqref{ineq_chi}, let $p$ and $\Gamma$  be fixed, and consider
\begin{equation}\label{fixed_point_up}
\inf_{\Delta \hat x }
 \begin{cases}
 e^{-\lambda \eps^2} \chi(x') - \delta, & \text{if } |z'|<  M  , \\ 
  - \chi(x) ,                           & \text{if } z' \geq  M, \\
    \chi(x) ,                           & \text{if } z' \leq -M.
\end{cases}
\end{equation}
If a minimizing sequence uses the second or third alternative then the inf is less than $\chi(x)$. 
In the remaining case, when all minimizing sequences use the first alternative, we apply Lemma \ref{pt_fix_estimate} to see that
\eqref{fixed_point_up} is bounded above by $ \chi(x) $. 
It follows that for all $x\in \ol \Omega$, $R_\eps[x,z, \chi] \leq \chi(x)$, as asserted.

For the second inequality of \eqref{ineq_chi}, the argument is strictly parallel by considering the function $-\chi$. 
We have shown that the map $\phi(\cdot,\cdot) \mapsto R_\eps[\cdot,\cdot, \phi]$ preserves the ball $F_{\chi}$.
Since it is also a contraction, the map has a unique fixed point. 
\end{proof}

This result justify the discussion of the stationary case given in Section \ref{games_presentation}, by showing
that the value functions $u^\eps$ and $v^\eps$ are well-defined, and bounded independently of $\eps$, 
and they satisfy the dynamic programming inequalities:
\begin{prop} \label{stability_prop_el}
Suppose $f$ satisfies \eqref{ellipticity_f} and \eqref{est_f_el_neu}--\eqref{loc_lip_p_Gamma_el}, 
 the $C^2$-domain $\Omega$ fulfills both the uniform interior and  exterior ball conditions, and the boundary condition $h$ is continuous, uniformly bounded.  
Assume the parameters of the game $\alpha, \beta, \gamma$ fulfill \eqref{condition_pas}--\eqref{cd_coeff_classiq},
 $\psi \in C_b^2(\ol \Omega)$ satisfy \eqref{def_psi_h_el}, $M$ large enough,  $m=M-1 - 2 \norm{\psi}_{L^\infty(\ol \Omega)}$, 
and  $\chi \in C_b^2(\ol \Omega)$ is defined by \eqref{def_chi_el}.
Let $V^\eps$ be the solution of \eqref{op_fixed_pt} obtained by Proposition~\ref{stability_el} and let $U^\eps(x,z)=V^\eps(x,z)-z$. Then the associated 
functions $u^\eps$, $v^\eps$ defined by \eqref{def_subelr}--\eqref{def_supelr} satisfy $ |u^\eps| \leq \chi$ and $ |v^\eps|  \leq \chi$ for all 
sufficiently small $\eps$, and they satisfy the dynamic programming inequalities 
\eqref{dyn_prog_ineq_sub_new_el} and \eqref{dyn_prog_ineq_super_new_el} at all points $x\in \overline \Omega$.
\end{prop}
\begin{proof} The bounds on $u^\eps$ and $v^\eps$ were demonstrated in \eqref{bound_u_v_eps}. 
The bounds assure that the dynamic programming inequalities hold for all $x\in \ol\Omega$, 
as a consequence of Proposition \ref{ineq_dyn_prog_el}.
\end{proof}

We close this section with the stationary analogue of Lemma \ref{lem_dec_U}.
\begin{lemma}\label{lem_dec_U_el}
Under the hypotheses of Proposition \ref{pt_fix_estimate}, suppose in addition that
\begin{equation*}\label{hyp_dec_f_z_el}
f(x,z_1,p, \Gamma)\geq f(x,z_0,p, \Gamma) \quad \text{ whenever }z_1>z_0.
\end{equation*}
Then $U^\eps$ satisfies 
\begin{equation*}\label{dec_U_el}
U^\eps(x,z_1)\leq U^\eps(x,z_0) - (z_1-z_0)  \quad \text{ whenever }z_1>z_0.
\end{equation*}
In particular, $U^\eps$ is strictly decreasing in $z$ and $v^\eps= u^\eps$.
\end{lemma}
\begin{proof} The Dirichlet case is provided in \cite[Lemma~5.4]{kohns}. 
For our game, it suffices to add $ - \norm{\Delta \hat x - \Delta x} h(x+\Delta x)$ in the expression of $\delta_0$ and $\delta_1$ defined in the proof of 
\cite[Lemma~5.4]{kohns}. Then the arguments can be repeated on the operator $R_\eps$ defined by \eqref{op_fixed_pt}, 
noticing that the function $\chi$ is independent of $z$.  
\end{proof}

\section{Some natural generalizations}
\label{generalizations}

In the precedent sections, we solved the Neumann boundary problem in both parabolic and elliptic settings. 
In the present section, we are going to explain without full proof how the previous work can be used to solve 
on the one hand the mixed Dirichlet-Neumann boundary conditions in the elliptic framework and on the other hand the oblique problem in the parabolic setting. 
For the definitions of the viscosity solutions on these frameworks which are the natural extensions of those 
presented in Section \ref{conv_def_visco}, the interested reader is referred to \cite{user_s_guide} or \cite{Barles_book}.

\subsection{Elliptic PDE with mixed  Dirichlet-Neumann boundary conditions}
We extend the games of Section~\ref{rules_el_game} devoted to the single Neumann problem to the mixed Dirichlet-Neumann boundary-value problem
\begin{equation}
\begin{cases}
 f(x,u,Du,D^2u)+\lambda u=0,      & \text{ in } \Omega,      \\ 
  u =g,                           & \text{ on } \Upsilon_D,  \\ 
\dfrac{\partial u}{\partial n}=h, & \text{ on }  \Upsilon_N, \\ 
\end{cases}
\label{eq_neumann_mixed}
\end{equation}
where $\Omega \subsetneq \R^N$ is a domain, $\Upsilon_D \cup \Upsilon_N= \partial \Omega$ is a partition of $\partial \Omega$
with $\Upsilon_D$ nonempty and closed and $\Upsilon_N$ is assumed to be $C^2$. 
Then,  $\Omega$ is assumed to satisfy the uniform exterior ball condition and, in a neighborhood of $\Upsilon_N$, the uniform interior  ball condition
explained in Definition~\ref{unif_int_ball_cd}. We will need a $C_b^2(\ol \Omega)$-function $\psi$  such that
\begin{equation} \label{psi_mixte}
 \frac{\partial \psi}{\partial n} = \norm{h}_{L^\infty}+1  \quad \text{on }  \Upsilon_N.
\end{equation}
From $m$ and  $\psi$, we construct a function $\chi$ defined by
\begin{equation}\label{chi_mixte}
\chi(x) = m+\norm{\psi}_{L^\infty}+\psi(x).
\end{equation}
As in Section \ref{rules_el_game}, we introduce $U^\eps(x,z)$, the optimal worst-case present value of Helen's wealth 
if the initial stock is $x$ and her initial wealth is $-z$. 
The definition of $U^\eps(x,z)$ for $x\in \Omega \cup \Upsilon_N$ involves here a game similar to that of Section \ref{rules_el_game}.
The rules are as follows: 
\begin{enumerate}
 \item Initially, at time $t_0=0$, the stock price is $x_0=x$ and Helen's debt is $z_0=z$. 
 \item Suppose, at time $t_j=j\eps^2$, the stock price is $x_j$ and Helen's debt is $z_j$ with $|z_j|<M$. 
Then Helen chooses $p_j \in \R^N$ and $\Gamma_j\in \mathcal{S}^N$, 
 restricted in magnitude by \eqref{p_beta_gamma_new}. Knowing these choices, Mark determines the next stock price
$ x_{j+1}=x_j+\Delta x$ so as to degrade Helen's outcome. 
Mark chooses an intermediate point $\hat x_{j+1}=x_j +\Delta \hat x_j \in \R^N$ such that $\left\|\Delta \hat x_j\right\| \leq \eps^{1-\alpha}$. 
This position $\hat x_{j+1}$ determines the new position $x_{j+1}=x_j +\Delta x_j$ by
\begin{equation*}
x_{j+1}= \proj_{\overline \Omega} (\hat x_{j+1}) \in \ol \Omega. 
\label{moving2_new_elmix} 
\end{equation*}
Helen experiences a loss at time $t_j$ of 
\begin{equation}\label{Helen_loss_elmix}
\delta_j =  p_j\cdot \Delta \hat x_j +\frac{1}{2}  \left\langle \Gamma_j \Delta \hat x_j,\Delta \hat x_j \right\rangle
+\eps^2 f(x_j,z_j,p_j,\Gamma_j) - \norm{\Delta \hat x_j - \Delta x_j}  h(x_j+\Delta x_j ) .
\end{equation}
As a consequence, her time $t_{j+1}=t_j+\eps^2$ debt becomes $z_{j+1}=e^{\lambda \eps^2} (z_j+\delta_j) $. 
\item  If $z_{j+1} \geq M$, the the game terminates, and Helen pays a 
``termination-by-large-debt penalty'' worth  $e^{\lambda \eps^2}(   \chi(x_j) -\delta_j)$ at time $t_{j+1}$. 
Similarly, if  $z_{j+1} \leq  - M$, the the game terminates, and Helen receives a 
``termination-by-large-wealth bonus'' worth $e^{\lambda \eps^2}(\chi(x_j) + \delta_j)$ at time $t_{j+1}$. 
If the game ends this way, we call $t_{j+1}$ the ``ending index'' $t_K$. 
\item If $|z_{j+1}|<M$  and $ x_{j+1}\in \Upsilon_D$, then the game terminates, and Helen gets an ``exit payoff'' worth 
$g( x_{j+1})$ at time $t_{j+1}$. If the game ends this way, we call $t_{j+1}$ the ``exit index''  $t_E$. 
\item If the game has not terminated then Helen and Mark repeat this procedure at time $t_{j+1}=t_j+\eps^2$. If the game never stops, the ``ending index'' $t_K$
is $+\infty$.
\end{enumerate}

All the possibilities, apart the end by exit, had already been investigated at Section \ref{rules_el_game}.  
If the game ends by exit at time $t_E$, then the present value of her income is 
\begin{align*}
 U^\eps(x_0,z_0) & = -z_0 - \delta_0 - e^{-\lambda \eps^2} \delta_1 - \cdots -  e^{-\lambda (E-1)\eps^2} \delta_{E-1} + e^{-\lambda E \eps^2} g(x_E) \\
                 & = e^{-\lambda E \eps^2} (g(x_E) - z_E).
\end{align*}
Since the game is stationary, the associated dynamic programming principle is that for $|z|<M$,
\begin{equation}\label{dpp_U_elmix}
U^{\eps}(x,z)=\sup_{p, \Gamma}\min_{\Delta \hat x }
 \begin{cases}
 e^{-\lambda \eps^2} U^{\eps}(x',z'),   & \text{if } x'\in \Omega \cup \Gamma_N \text{ and } |z'|<  M,   \\
 e^{-\lambda \eps^2} (g(x') -z')    ,   & \text{if } x'\in \Gamma_D \text{ and }|z'|<  M,   \\ 
 -z - \chi(x)                       ,   & \text{if } z' \geq  M, \\
 -z + \chi(x)                       ,   & \text{if } z' \leq -M,
\end{cases}
\end{equation}
where $x'=\proj_{\overline \Omega} (x+\Delta \hat x)$ and $z'=e^{\lambda \eps^2} (z+\delta)$, with $\delta$ defined by \eqref{Helen_loss_elmix}. 
Here $p$, $\Gamma$ and $\Delta \hat x$ are constrained as usual by \eqref{p_beta_gamma_new}--\eqref{moving1_new}. 

The definitions \eqref{def_subelr}--\eqref{def_supelr} of $u^{\eps}$ and  $v^{\eps}$ on $\Omega \cup \Gamma_N$ are conserved. 
The corresponding semi-relaxed limits are defined for any $x\in \ol \Omega$ by 
\begin{equation*} 
\overline u(x)= \limsup_{\substack{y\rightarrow x \\ \eps \rightarrow 0}} u^{\eps}(y)  \quad \text{ and }   \quad 
  \underline v(x)= \liminf_{\substack{y\rightarrow x \\ \eps \rightarrow 0}} v^{\eps}(y), 
\end{equation*}
with the convention that $y$ approaches $x$ from  $\Omega \cup \Gamma_N$ (since $u^{\eps}$ and $v^{\eps}$ are only defined on  $\Omega \cup \Gamma_N$).
Proposition \ref{ineq_dyn_prog_el} still holds without any modification for mixed-type Dirichlet-Neumann boundary conditions. 
Moreover, the definition of viscosity subsolutions and supersolutions is clear by relaxing the PDE condition  on $\Upsilon_D$ 
with the Dirichlet condition in the same way that has been done in \cite[Section 3]{kohns}.

Following the same steps as our proof for the Neumann problem (the main modification consists in the proof of convergence on $\Upsilon_D$ but 
has already been done in \cite{kohns}),   the following theorem is now immediate.

\begin{theorem}\label{theo_cv_obl_neu}
 Consider the stationary boundary value problem \eqref{eq_neumann_mixed} where $f$ satisfies 
\eqref{ellipticity_f} and \eqref{est_f_el_neu}--\eqref{cont_growth_p_Gamma_el},
$g$ and  $h$ are continuous, uniformly bounded and $\Omega$ is a $C^2$-domain satisfying the uniform exterior ball condition 
 and the uniform interior ball condition in a neighborhood of $\Upsilon_N$. 
Assume the parameters of the game $\alpha$, $\beta$, $\gamma$ fulfill \eqref{condition_pas}--\eqref{cd_coeff_classiq}, 
  $\psi \in C_b^2 (\ol \Omega)$ satisfies \eqref{psi_mixte}, \mbox{$\chi \in C^2(\ol \Omega)$} is defined by \eqref{chi_mixte}, 
$M$ is sufficiently large, and $m=M-1 -2 \norm{\psi}_{L^\infty(\ol \Omega)} $. 
Then $u^\eps$ and $v^\eps$ are well-defined when $\eps$ is sufficiently small, and they satisfy 
$|u^\eps | \leq \chi$ and  $|v^\eps|\leq \chi$.
Their relaxed semi-limits $\overline u$ and $\underline v$ are respectively a viscosity subsolution and a viscosity supersolution of \eqref{eq_neumann_mixed}. 
If in addition we have $\underline v \leq \overline u$ and the PDE has a comparison principle, then it follows that $u^\eps$ and $v^\eps$ 
converge locally uniformly in $\ol \Omega$ to the unique viscosity solution of \eqref{eq_neumann_mixed}.
\end{theorem} 

\subsection{Parabolic PDE with an oblique boundary condition}

The target of this section is to construct a game which could interpret the  PDE with an oblique condition $h$
and final-time data $g$ given by  
\begin{equation}
\left\{
\begin{array}{ll}
\partial_t u  - f(t,x,u,Du, D^2u)=0, & \text{for }x \in \Omega \text{ and } t<T,                \\
\dfrac{\partial u}{\partial \varsigma}(x,t)=h(x),  & \text{for }x \in \partial \Omega \text{ and } t<T,  \\
 u(x,T)=g(x),				& \text{for } x \in \overline \Omega , 
\end{array}
 \right.
\label{def_int_oblique}
 \end{equation}
where $\obl$ defines a smooth vector field, say $C^2$, on $\partial \Omega$ pointing outward such that 
\begin{equation}\label{hyp_vector_field}
\langle \obl(x) , n(x)\rangle \geq \theta >0 \quad \text{ for all }x\in \partial \Omega. 
\end{equation}
As usual, the domain $\Omega$ is supposed to be at least of boundary $C^2$ and to satisfy both the uniform and the exterior ball conditions.

First of all, following P.L. Lions \cite[Section 5]{lions_neumann_type}, P.L. Lions and A.S. Sznitman \cite{lions_sznitman},  
 we introduce some smooth functions $a_{ij}(x)=a_{ji}(x)$, say $C_b^2(\R^N)$,  such that 
 \begin{align}
 &\exists \theta >0, \forall x\in \R^N, (a_{ij}(x)) \geq \theta I_N ,  \label{theta_aij} \\
 &\forall x\in \partial \Omega, \sum_{j=1}^N a_{ij}(x) \obl_j(x)=n_i(x) \quad \text{ for } 1\leq i \leq N. \notag
 \end{align}
 Clearly if we had $ \obl =n$, we would just take $a_{ij}(x)=\delta_{ij}$. 
 Next, the matrices induce a metric $d_\obl$ on $\R^N$ defined by 
 \begin{equation}\label{d_oblique}
 d_\obl (x,y) = \inf \left\{   \int_0^1 \left[ \sum_{1 \leq i,j \leq N} a_{ij}(\xi(t))\dot{\xi}_i(t)\dot{\xi}_j(t)\right]^{1/2} dt :  
   \xi \in C^1([0,1], \R^N), \xi(0)=y,\xi(1)=x \right\}. 
 \end{equation}
 Then it is well known that for $\norm{x-y}$ small, there exists a unique minimizer in \eqref{d_oblique}. 
The interested reader is referred to \cite{lions_neumann_type} for additional properties about $d_\obl$. 
For this specific metric, we can now define for any $x$ lying on a small $\delta$-neighborhood of the boundary 
a unique projection according the vector field $\gamma$ along the boundary by 
\begin{equation}\label{def_proj_oblique}
\bar x^\gamma= \proj_{\ol \Omega}^\obl(x) \in \partial \Omega, 
\end{equation}
which corresponds to the unique minimum of $d_\obl(x,y)$ for $y$ lying on the boundary. 
Finally, $B_\obl(x,r)$ denotes the ball of center $x$ and radius $r$ induced by the metric $d_\obl$.

We can now explain the rules of the game corresponding to the oblique problem \eqref{def_int_oblique}. 
Let the parameters $\alpha$, $\beta$, $\gamma$  satisfy \eqref{condition_pas}--\eqref{cd_coeff_classiq}. 
When the game begins, the position can have any value $x_0\in \ol{\Omega}$; 
Helen's initial score is $y_0=0$. The rules are as follows: if at time $t_j=t_0+j\eps^2$ Helen's debt is $z_j$ and the stock price is $x_j$,  then
\begin{enumerate}
 \item Helen chooses a vector $p_j \in \R^N$ and a matrix $\Gamma_j\in \mathcal{S}^N$, restricted in magnitude by \eqref{p_beta_gamma_new}. 
 \item Taking Helen's choice into account, Mark chooses the stock price $x_{j+1}=x_{j}+\Delta x_j$ so as to degrade Helen's outcome. 
 Mark is going to choose an intermediate point $\hat x_{j+1}=x_j +\Delta \hat x_j \in \R^N$ such that 
\begin{equation}\label{moving1_new_obl}
\hat x_{j+1} \in B_\obl(x_j, \eps^{1-\alpha}), 
\end{equation}
which determines the new position $x_{j+1}=x_j+\Delta x_j \in \overline \Omega$  by the rule
\begin{equation*}\label{moving2_new_obl} 
x_{j+1} = \proj_{\ol \Omega}^\obl(\hat x_{j+1}), 
\end{equation*}
where $\proj_{\bar \Omega}^\obl$ is the projection defined by \eqref{def_proj_oblique}.
\item Helen's debt is changed to 
\begin{equation*}
z_{j+1}=z_j +  p_j\cdot \Delta \hat x_j +\frac{1}{2}  \left\langle \Gamma_j \Delta \hat x_j,\Delta \hat x_j \right\rangle
+\eps^2 f(t_j,x_j,z_j,p_j,\Gamma_j) - d_\obl(\hat x_{j+1}, x_{j+1})  h(x_j+\Delta x_j ) .
\label{Helens_debt_obl}
\end{equation*}
\item The clock steps forward to $t_{j+1}=t_j+\eps^2$  and the process repeats, stopping when $t_K=T$. 
At the final time Helen receives $g(x_K)$ from the option.
\end{enumerate}

Rather than repeating the arguments already used, we are going to explain the modifications to carry out the analysis. 
First of all, by the boundedness of the $a_{ij}$ and \eqref{theta_aij}, the distance $d_\obl$ defined by \eqref{d_oblique}
is equivalent to the euclidean distance. 
Since $\Omega$ satisfies the uniform exterior ball condition, there exists, for a certain $r_\obl>0$,
 a tubular neighborhood $\{x\in \R^N\backslash \Omega, d(x) < r_\obl\}$ of the boundary on which $\proj_{\bar \Omega}^\obl$ is well-defined. 
This guarantees the well-posedness of this game for all  $\eps>0$ small enough.
Then, if $d_\obl$ or the euclidean distance is used to compute $D\phi$ and $D^2\phi$ for a smooth function $\phi$, we will get the same results.
Therefore, we can introduce the oblique analogues $m_{\obl, \eps}^{x}[\phi]$ and $M_{\obl, \eps}^{x}[\phi]$ of \eqref{m_eps_par}--\eqref{M_eps_par} by  
\begin{align}
m_{\obl, \eps}^{x}[\phi] & :=\inf_{\substack{x+\Delta \hat x \notin \Omega \\ \Delta \hat x}} 
\left\{ h(x+\Delta x)  - D \phi(x)\cdot \obl(x+\Delta x)  \right\}, \label{m_eps_par_obl} \\
M_{\obl, \eps}^{x}[\phi] & :=\sup_{\substack{x+\Delta \hat x \notin \Omega \\ \Delta \hat x}} 
 \left\{ h(x+\Delta x)  - D \phi(x)\cdot \obl(x+\Delta x)  \right\}, \label{M_eps_par_obl}
\end{align}
where $\Delta \hat x$ is constrained by \eqref{moving1_new_obl} and $\Delta x$ is determined by $\Delta x = \proj_{\bar \Omega}^\obl(x+\Delta \hat x) - x$.
Thus, the particular choices $p_\text{opt}^{m_ \obl}$, $p_\text{opt}^{M_\obl}$ and  $\Gamma_\text{opt}^\obl$ will be now respectively defined in the 
orthonormal basis $\mathcal{B}_\obl=(e_1=\obl(\bar x^\gamma),e_2, \cdots, e_N)$ by 
\begin{align*}
p_\text{opt}^{m_\obl}(x) &  = D\phi(x)  +\left[
 \frac{1}{2}\left(1-\frac{ d_\obl(x)}{\eps^{1-\alpha}}\right) m_{\obl, \eps}^{x}[\phi] 
 - \frac{\eps^{1-\alpha}}{4} \left(1 - \frac{d^2_\obl(x)}{\eps^{2-2\alpha}} \right) (D^2\phi(x))_{11} \right]  \obl(\bar{x}^\gamma),  \\
p_\text{opt}^{M_\obl} (x)  & = D\phi(x)  +\left[\frac{1}{2}\left(1-\frac{ d_\obl(x)}{\eps^{1-\alpha}}\right) M_{\obl, \eps}^{x}[\phi]
- \frac{\eps^{1-\alpha}}{4} \left( 1 - \frac{d^2_\obl(x)}{\eps^{2- 2\alpha}} \right) (D^2\phi(x))_{11} \right] \obl(\bar{x}^\gamma),
\end{align*}
and
\begin{equation*} 
 \Gamma_\text{opt}^\obl(x) = D^2\phi(x) +\left[ \frac{1}{2}\left( - 1+\frac{d_\obl^2(x)}{\eps^{2-2\alpha}} \right) (D^2\phi(x))_{11}  \right] E_{11}, 
\end{equation*}
where $m_{\obl, \eps}^{x}[\phi]$ and $M_{\obl, \eps}^{x}[\phi]$ are defined by \eqref{m_eps_par_obl}--\eqref{M_eps_par_obl}, 
and $E_{11}$ denotes the unit-matrix $(1,1)$ in the basis~$\mathcal{B}_\obl$. 
The definitions of $u^\eps$, $v^\eps$ and their relaxed semi-limits $\ol u$ and $\udl v$, 
given by \eqref{def_u_eps_par}--\eqref{def_v_eps_par} and \eqref{def_bp_subsup_visc}, are conserved. 
The only change on the dynamic programming inequalities \eqref{dyn_prog_ineq_sub_new}--\eqref{dyn_prog_ineq_super_new}
concerning $u^\eps$  and $v^\eps$ is to replace $\norm{\Delta \hat x - \Delta x}$ 
by $d_\obl(x+\Delta \hat x, x+\Delta x)$,  
and to constrain $\Delta \hat x $ by~\eqref{moving1_new_obl}. 
For stability, we need to consider a $C_b^2(\ol \Omega)$-function $\psi$ such that 
\begin{equation*}
 \frac{\partial \psi}{\partial \obl}(x)= \norm{h}_{L^\infty}+1 \quad  \text{ on } \partial \Omega.
\end{equation*}
It is still allowed by the uniform interior ball condition applied to the $C^2$-domain $\Omega$. 
 By using exactly the same ingredients already used for the Neumann problem 
 and adapting the geometric estimates given by Section \ref{prel_geo_lem} in the oblique framework, we obtain the following theorem. 

\begin{theorem}\label{theo_cv_par_oblique}
Consider the final-value problem \eqref{def_int_oblique} where $f$ satisfies \eqref{ellipticity_f}--\eqref{cont_growth_p_Gamma}, 
 $g$  and $h$ are continuous, uniformly bounded, 
$\Omega$ is a $C^2$-domain satisfying both the uniform interior and exterior ball conditions, 
 and $ \obl $ is a continuous vector field on $\partial \Omega$ and satisfy \eqref{hyp_vector_field}. 
Assume the parameters $\alpha$, $\beta$, $\gamma$  fulfill  \eqref{condition_pas}-\eqref{cd_coeff_classiq}. 
Then $\overline u$ and $\underline v$ are uniformly bounded on $\ol \Omega \times [t_\ast, T]$ 
for any $t_\ast<T$, and they are respectively a viscosity subsolution and a viscosity supersolution of \eqref{def_int_oblique}. 
If the PDE has a comparison principle (for uniformly bounded solutions), then it follows that $u^\eps$ and $v^\eps$ 
converge locally uniformly to the unique viscosity solution of \eqref{def_int_oblique}. 
\end{theorem}

\textbf{Acknowledgements: } I thank Sylvia Serfaty for bringing the problem to my attention and numerous helpful discussions. 
I thank Scott N. Armstrong for fruitful and encouraging talks and Guy Barles for helpful comments about viscosity solutions.
Finally, I gratefully acknowledge support from the European Science Foundation through a EURYI award of Sylvia Serfaty.

\bibliography{ref}
\bibliographystyle{plain} 

\noindent
{\sc Jean-Paul Daniel}\\
  UPMC Univ  Paris 06, UMR 7598 Laboratoire Jacques-Louis Lions,\\
  Paris, F-75005 France ;\\
  CNRS, UMR 7598 LJLL, Paris, F-75005 France \\
    {\tt daniel@ann.jussieu.fr}

\end{document}